\documentclass[12pt]{amsart}
 \usepackage{amsfonts,amssymb}
  \usepackage{enumitem}
 \setlist[itemize]{noitemsep,nolistsep}
 \usepackage[toc,page]{appendix}
\usepackage{mathrsfs}
\usepackage{array}
\usepackage[all,ps,cmtip,rotate]{xy} 
 \usepackage[usenames,dvipsnames]{color}
\usepackage{hyperref}
\usepackage{bm,mathtools}

\def\Z{{\bf Z}}

\def\C{{\bf C}}

\def\Q{{\bf Q}}
\def\P{{\bf P}}
\def\PP{{\bf P}}

\def\cI{\mathscr{I}}
\def\cD{\mathscr{D}}

\def\cA{\mathscr{A}}

\def\cL{\mathscr{L}}
\def\cO{\mathscr{O}}

\def\cP{\mathscr{P}}

\def\cE{\mathscr{E}}

\def\cC{\mathscr{C}}

\def\cN{\mathscr{N}}

\def\cQ{\mathscr{Q}}
\def\cU{\mathscr{U}}
\def\cV{\mathscr{V}}

\def\cX{\mathscr{X}}
\def\cY{\mathscr{Y}}

\def\bp{\mathbf p}

\def\lra{\longrightarrow}
\def\llra{\hbox to 10mm{\rightarrowfill}}
\def\lllra{\hbox to 15mm{\rightarrowfill}}

\def\llla{\hbox to 10mm{\leftarrowfill}}
\def\lllla{\hbox to 15mm{\leftarrowfill}}

\def\thra{\twoheadrightarrow}
\def\hra{\hookrightarrow}
\def\lhra{\ensuremath{\lhook\joinrel\relbar\joinrel\to}}

\def\isom{\simeq}

\def\ie{\hbox{i.e.}}

 \def\vide{\varnothing}
  \def\emptyset{\varnothing}

\DeclareMathOperator{\isomto}{\stackrel{{}_{\scriptstyle\sim}}{\to}}

\DeclareMathOperator{\isomlra}{\stackrel{{}_{\scriptstyle\sim}}{\lra}}

\DeclareMathOperator{\Aut}{Aut}

\DeclareMathOperator{\codim}{codim}

\DeclareMathOperator{\Coker}{Coker}

\DeclareMathOperator{\Ext}{Ext}

\DeclareMathOperator{\GL}{GL}
\DeclareMathOperator{\Gr}{\mathsf{Gr}}

\DeclareMathOperator{\CGr}{\mathsf{CGr}}

\DeclareMathOperator{\Fl}{\mathsf{Fl}}
\DeclareMathOperator{\Hilb}{Hilb}
\DeclareMathOperator{\Hom}{Hom}

\DeclareMathOperator{\Ker}{Ker}

\DeclareMathOperator{\lin}{\underset{\mathrm lin}{\equiv}}

\DeclareMathOperator{\PGL}{PGL}
\DeclareMathOperator{\Pic}{Pic}

\DeclareMathOperator{\Spec}{Spec}

\DeclareMathOperator{\Sing}{Sing}
\DeclareMathOperator{\Sym}{\mathsf S}

\def\bw#1#2{\textstyle{\bigwedge\hskip-0.9mm^{#1}}\hskip0.2mm{#2}}
\def\sbw#1#2{{\bigwedge\hskip-0.9mm^{#1}}\hskip0.1mm{#2}}

\def\symv{\VV}

\newtheorem{lemm}{Lemma}[section]
\newtheorem{theo}[lemm]{Theorem}
\newtheorem{coro}[lemm]{Corollary}
\newtheorem{prop}[lemm]{Proposition}

\theoremstyle{remark}

\newtheorem{rema}[lemm]{Remark}

\def\ot{\otimes}

\def\VV{\mathbb{V}}

\def\ord{\mathrm{ord}}
\def\spe{\mathrm{spe}}
 
\newcommand{\cone}[1]{\mathsf{C}_{#1}}

\newcommand{\hY}{\widehat{Y}}
\newcommand{\tY}{\widetilde{Y}}
\newcommand{\tZ}{\widetilde{Z}}

\newcommand{\pr}{\mathrm{pr}}

\newcommand{\wwp}{\wp^{\mathrm{epw}}}
\newcommand{\tsi}{{\tilde{\sigma}}}
\newcommand{\bsi}{\bm{\sigma}}

\def\setminus{\smallsetminus}
\def\cong{\isom}

\subjclass[2010]{14J45,14J35, 14J40, 14M15}

\begin{document}
\title 
{Gushel--Mukai varieties: linear spaces and periods}

\author[O.~Debarre]{Olivier Debarre}
\address{\parbox{0.9\textwidth}{Univ  Paris Diderot, \'Ecole normale su\-p\'e\-rieu\-re, PSL Research University,
\\[1pt] 
CNRS, D\'epar\-te\-ment Math\'ematiques et Applications
\\[1pt]
45 rue d'Ulm, 75230 Paris cedex 05, France}}
\email{{olivier.debarre@ens.fr}}
 
 \author[A. Kuznetsov]{Alexander Kuznetsov}
\address{\parbox{0.9\textwidth}{Algebraic Geometry Section, Steklov Mathematical Institute,
\\[1pt]
8 Gubkin str., Moscow 119991, Russia
\\[5pt]
The Poncelet Laboratory, Independent University of Moscow
\\[5pt]
Laboratory of Algebraic Geometry, National Research University Higher School of Economics, Russian Federation}}
\email{{\tt  akuznet@mi.ras.ru}}

\thanks{A.K. was partially supported by the Russian Academic Excellence Project ``5-100'', by RFBR grants 15-01-02164, 15-51-50045, and by the Simons foundation.}

\begin{abstract}
Beauville and Donagi proved in 1985 that the primitive middle cohomology of a smooth complex cubic fourfold and the primitive second cohomology of its variety of lines, 
a smooth hyperk\"ahler fourfold, are isomorphic as  {polarized} integral Hodge structures.\ 
We prove analogous statements for smooth complex Gushel--Mukai varieties of dimension  4 (resp.\ 6), \ie, smooth dimensionally transverse intersections 
of the cone over the Grassmannian $\Gr(2,5)$, a quadric, and two hyperplanes (resp.\ of the cone over  $\Gr(2,5)$ and a quadric).\ The associated hyperk\"ahler fourfold is in both cases a smooth double cover of a hypersurface in $\P^5$ called an EPW sextic.
\end{abstract}

\maketitle

\section{Introduction}

We continue in this article our investigation of Gushel--Mukai (GM) varieties started in~\cite{DK}.\ We discuss   linear subspaces contained in smooth complex GM varieties and their relation to Eisenbud--Popescu--Walter (EPW) stratifications.\ These results are applied to the computation of the period map for GM varieties of   dimension 4 or 6.

We work over the field of complex numbers.\  A smooth {\sf Gushel--Mukai variety} is (\cite[Definition~2.1]{DK})   a smooth dimensionally transverse intersection
\begin{equation*}
   \CGr(2,V_5) \cap \P(W) \cap Q
\end{equation*}
of the  cone over the Grassmannian $\Gr(2,V_5)$ of 2-dimensional subspaces in a fixed 5-dimensional vector space $V_5$,
with a linear subspace $\P(W)$ and a quadric $Q$.\ This class of varieties includes all smooth prime Fano varieties $X$ of dimension $n \ge 3$, coindex~3, and degree~10 (\ie, such that there is an ample class $H$ {with} $\Pic(X)=\Z H$,   $K_X = -(n-2)H$, and $H^n = 10$; see \cite[{Theorem~2.16}]{DK}).

One can naturally associate  with any smooth GM variety of dimension $n $ a triple $(V_6,V_5,A)$, called a {\sf Lagrangian data}, 
where $V_6$ is a 6-dimensional vector space  containing $V_5$ as a hyperplane and the subspace $A \subset \bw3V_6$ is Lagrangian with respect to the   symplectic structure on~$\bw3V_6$ given by wedge product.\ Moreover, $\P(A ) \cap \Gr(3,V_6) = \emptyset$ in~$\P(\bw3V_6)$ when $n\ge 3$ (we   say that~{\sf $A$ has no decomposable vectors}).

Conversely, given a  Lagrangian data $(V_6,V_5,A)$ with no decomposable vectors in $A$, one can construct two smooth GM varieties of respective dimensions $n = 5 - \ell$ and $n = 6 - \ell$  (where $\ell := \dim(A \cap \bw3V_5)\le3$), with associated Lagrangian data $(V_6,V_5,A)$  (\cite[{Theorem~3.10 and Proposition~3.13}]{DK}; see Section~\ref{reminder} for more details).

Given a Lagrangian subspace $A \subset \bw3V_6$, we  define three   chains of subschemes
\begin{equation*}
Y^{\ge 3}_A \subset Y^{\ge 2}_A \subset Y^{\ge 1}_A \subset \P(V_6),
\qquad
Y^{\ge 3}_{A^\perp} \subset Y^{\ge 2}_{A^\perp} \subset Y^{\ge 1}_{A^\perp} \subset \P(V_6^\vee),
\end{equation*}
\begin{equation*}
Z^{\ge 4}_A \subset Z^{\ge 3}_A \subset Z^{\ge 2}_A \subset Z^{\ge 1}_A \subset \Gr(3,V_6),
\end{equation*}
called {\sf Eisenbud--Popescu--Walter (EPW) stratifications} (see Section~\ref{se22}).\ The first two were extensively studied by O'Grady (\cite{og1,og2,og4,og6,og7,og8})  and the third in~\cite{ikkr}.\  If~$A$ has no decomposable vectors, the strata  
\begin{equation*}
Y_A := Y^{\ge 1}_A \subset \P(V_6),
\qquad
Y_{A^\perp} := Y^{\ge 1}_{A^\perp} \subset \P(V_6^\vee),
\qquad\text{and}\qquad
Z_A := Z^{\ge 1}_A \subset \Gr(3,V_6),
\end{equation*}
are hypersurfaces of respective degrees 6, 6, and 4, called the {\sf EPW sextic}, the {\sf dual EPW sextic}, and the {\sf EPW quartic} associated with $A$.\ Moreover, there are canonical double coverings
\begin{equation*}
\tY_A \to Y_A,
\qquad 
\tY_{A^\perp} \to Y_{A^\perp},
\quad\text{and}\quad 
\tZ_A^{\ge 2} \to Z_A^{\ge 2},
\end{equation*}
called the {\sf double EPW sextic}, the {\sf double dual EPW sextic}, and the {\sf EPW cube} associated with~$A$, respectively.\ In general (more precisely, when $Y_A^{\ge 3} = \emptyset$, $Y_{A^\perp}^{\ge 3} = \emptyset$, and $Z_A^{\ge 4} = \emptyset$), these are  hyperk\"ahler manifolds which are  deformation equivalent to the Hilbert square or cube of a K3 surface.

We showed in~\cite{DK} that these EPW stratifications control many geometrical properties of GM varieties.\  For instance,   smooth GM varieties of dimension 3 or 4 are birationally isomorphic if their associated EPW sextics are isomorphic (\cite[Theorems~4.7 and~4.15]{DK}).\ In this article, we  describe the Hilbert schemes of linear spaces contained in smooth GM varieties in terms of their EPW stratifications  and relate the Hodge structures of smooth  GM varieties of dimension 4 or 6  to those of their associated {double} EPW sextics.

Let $X$ be a smooth GM variety.\  We denote by $F_k(X)$ the Hilbert scheme of linearly embedded projective $k$-spaces in $X$.\ The scheme $F_2(X)$ has two connected components $ F_2^\sigma(X)$ and $F_2^\tau(X)$  corresponding to the two types of projective planes in $\Gr(2,V_5)$.\ We construct maps
\begin{equation*}
F_3(X) \to \P(V_5),
\qquad
F_2^\sigma(X) \to \P(V_5),
\qquad
F_2^\tau(X) \to \Gr(3,V_5),
\qquad
F_1(X) \to \P(V_5)
\end{equation*}
and describe them in terms of the  EPW varieties defined by the Lagrangian $A$ associated with~$X$ (Theorems~\ref{p32}, \ref{theorem:f2}, \ref{th34}, and~\ref{theorem:f1}).\ We prove in particular the following results.

If $X$ is a smooth GM sixfold with associated Lagrangian $A$ such that $Y^{\ge 3}_{A} = \emptyset$, the scheme $ F_2^\sigma(X)$ has dimension 4 and the above map $F_2^\sigma(X) \to \P(V_5)$ factors as  
\begin{equation*}
F_2^\sigma(X) \to \tY_A \times_{\P(V_6)} \P(V_5) \to \P(V_5),
\end{equation*}
where the first map is a locally trivial (in the \'etale topology) $\P^1$-bundle  {(Theorem~\ref{theorem:f2}(a))}.

If $X$ is a smooth general (with explicit generality assumptions) GM fourfold, $ F_1(X)$ has dimension 3 and the map $F_1(X) \to \P(V_5)$   factors as  
\begin{equation*}
F_1(X) \to \tY_A \times_{\P(V_6)} \P(V_5) \to \P(V_5),
\end{equation*}
where the first map is a small resolution of singularities (a contraction of two rational curves; see Theorem~\ref{theorem:f1}(c)).

Consequently, the universal plane $\cL_2^\sigma(X)$ in the sixfold case and the universal line $\cL_1(X)$ in the fourfold case give correspondences\begin{equation*}
\vcenter{\xymatrix@R=5mm{
& \cL_2^\sigma(X) \ar[dl] \ar[dr] \\
X && \tY_A
}}
\qquad\qquad \text{and} \qquad\qquad
\vcenter{\xymatrix@R=5mm{
& \cL_1(X) \ar[dl] \ar[dr] \\
X && \tY_A
}}
\end{equation*}
 between  $X$ and its associated double EPW sextic $\tY_A$.\ We use them to construct, in dimensions $n=4$ or 6, isomorphisms 
\begin{equation*}
H^n(X;\Z)_{00} \cong H^2(\tY_A;\Z)_0
\end{equation*}
of  {polarized} integral Hodge structures  (up to Tate twists; see Theorem~\ref{th32} for precise statements) between the vanishing middle cohomology of $X$ (defined in~\eqref{defh00}) and the primitive second cohomology of~$\tY_A$ (defined in~\eqref{defprime}).

This isomorphism is the main result of this article.\ It implies that the period point of a  smooth GM variety of dimension~4 or~6 (defined as the class of its vanishing cohomology Hodge structure in the appropriate period space) with associated Lagrangian data $(V_6,V_5,A)$ depends only on the $\PGL(V_6)$-orbit of $A$ and not on $V_5$.

More precisely, the period maps from the moduli stacks of GM varieties of dimension~4 or~6 factor  through the period map from the moduli stack of double EPW sextics---the period spaces being the same.\ The first map in this factorization is a fibration with well understood fibers (see~\cite[{Theorem~3.25}]{DK}).\ Since double EPW sextics, when smooth,  are hyperk\"ahler manifolds,  the second map is an open embedding by Verbitsky's Torelli Theorem.

The article is organized as follows.

In Section~\ref{section:gm}, we recall some of the results from~\cite{DK} about the geometry of smooth GM varieties and their relation to EPW varieties.\ In Section~\ref{section:cohomology}, we discuss the singular cohomology of GM varieties.\ In Section~\ref{section:linear-spaces}, we describe the Hilbert schemes $F_k(X)$ for smooth GM varieties~$X$.\ In Section~\ref{section:periods}, we prove an isomorphism between the vanishing Hodge structure of a general GM variety of dimension~4 or~6 and the primitive Hodge structure of the associated double EPW sextic.\ We also define the period point of a GM variety and show that it coincides with the period point of the associated double EPW sextic.\ In Appendix~\ref{section:linear-spaces-quadrics}, we discuss the natural double coverings arising from the Stein factorizations of relative Hilbert schemes of quadric fibrations.\ In  Appendix~\ref{section:epw-surface}, we discuss a resolution of the structure sheaf of an EPW surface $Y_A^{\ge 2}$ in~$\P(V_6)$ and compute some cohomology spaces related to its ideal sheaf.
 
We are grateful to Grzegorz and Micha\l{} Kapustka, Giovanni Mongardi, Kieran O'Grady, and Kristian Ranestad for interesting exchanges.\ We would also like to thank the referee for her/his careful reading of our article.

\section{Geometry of Gushel--Mukai varieties}
\label{section:gm}

\subsection{Gushel--Mukai varieties}\label{reminder}

We work over the field of complex numbers.\ 
A  smooth Gushel--Mukai (GM) variety of dimension $n$  is (\cite[Definition~2.1 and {Proposition~2.28}]{DK}) a smooth dimensionally transverse intersection
 \begin{equation}\label{defgm}
X = \CGr(2,V_5) \cap \P(W) \cap Q, 
\end{equation}
where $V_5$ is a  vector space of dimension 5, $\CGr(2,V_5) \subset \P(\C \oplus \bw2V_5)$ is the cone (with vertex $\nu:=[\C]$) over the Grassmannian of 2-dimensional subspaces in $V_5$, $W \subset \C \oplus \bw2V_5$ is a vector subspace of dimension $n+5$, and $Q \subset \P(W)$ is a quadratic hypersurface. 

Being smooth, $X$ does not contain the vertex $\nu$, hence the linear projection from $\nu$ defines a regular map 
\begin{equation*}
\gamma_X\colon X \to \Gr(2,V_5 )
\end{equation*}
called the {\sf Gushel map} of $X$.\ We denote by $\cU_X$ the pullback to $X$ of the tautological rank-2 subbundle on the Grassmannian.\ It comes with an embedding  $\cU_X \hookrightarrow V_5 \otimes \cO_X$.

Following~\cite{DK}, we associate with every smooth GM variety $X$ as in \eqref{defgm} the intersection
\begin{equation*}
M_X := \CGr(2,V_5) \cap \P(W).
\end{equation*}
This is a variety of dimension $n+1$ with finite singular locus {(\cite[{Proposition~2.22}]{DK})}. 

If the linear space $\P(W)$ does not contain the vertex $\nu$,  the variety $M_X$ is itself a dimensionally transverse section of $\Gr(2,V_5)$ by the image of the linear projection $\P(W)\to \P(\bw2V_5)$ from~$\nu$.\ It is smooth if $n \ge 3$ 
 and  $X$ is its intersection with a quadratic hypersurface.\ These GM varieties are called {\sf ordinary}.

If $\P(W) $ contains  $\nu$,  the variety $M_X$ is itself a cone with vertex $\nu$ over the smooth dimensionally transverse linear section
\begin{equation*}
M'_X = \Gr(2,V_5) \cap \P(W')\subset \P(\bw2V_5),
\end{equation*}
where $W' = W/\C \subset \bw2V_5$, and $X$ is a double cover of $M'_X$ branched along the smooth GM variety $X' = M'_X \cap Q$ of dimension $n-1$.\ These GM varieties are called {\sf special}.

A GM variety $X \subset \P(W)$ is an intersection of quadrics.\ Following~\cite{DK}, we denote by~$V_6$ the 6-dimensional space of quadratic equations of $X$.\ The space $V_5$ can be naturally identified with the space of Pl\"ucker quadrics cutting out $\CGr(2,V_5)$ {in $\P(\C \oplus \bw2V_5)$}, hence also with the space of quadrics in $\P(W)$ cutting out the subvariety $M_X$.\ This gives a canonical embedding
$V_5 \subset V_6 $ which identifies $V_5$ with a hyperplane in $V_6$  called the {\sf Pl\"ucker hyperplane}.\ The corresponding point $\bp_X \in \P(V_6^\vee)$ {in the dual projective space} is called the {\sf Pl\"ucker point}.

\subsection{EPW sextics and quartics}\label{se22}

Let $X$ be a smooth GM variety  of dimension $n$.\ As explained in~\cite[Theorem~3.10]{DK}, one can associate with $X$ a  subspace $A \subset \bw3V_6$ which is Lagrangian for the $\det(V_6)$-valued symplectic form given by wedge product.\ Together with the pair $V_5\subset V_6$ defined above, it forms a triple $(V_6,V_5,A)$ called the {\sf Lagrangian data} of $X$.

The Lagrangian space $A$ has no decomposable vectors (\ie, $\P(A) \cap \Gr(3,V_6) = \emptyset$) when~$n \ge 3$  {(\cite[{Theorem~3.16}]{DK})} and the vector space $ A \cap \bw3V_5 $ has dimension $5 - n$ if~$X$ is ordinary, $6-n$ if $X$ is special (\cite[{Proposition~3.13}]{DK}).

Conversely, given a Lagrangian data $(V_6,V_5,A)$ such that $A$ has no decomposable vectors, we have $\ell := \dim(A \cap \bw3V_5)\le3$  {and} there are
\begin{itemize} 
\item an ordinary {smooth} GM variety $  X_\ord(V_6,V_5,A)$  of dimension $5 - \ell$,  
\item a special {smooth} GM variety $  X_\spe(V_6,V_5,A)$ of dimension $6 - \ell$,
\end{itemize}
unique up to isomorphism, whose  associated Lagrangian data is $(V_6,V_5,A)$.

Given a Lagrangian subspace $A \subset \bw3V_6$, one can construct interesting varieties that play an important role for the geometry of the associated GM varieties.\ Following O'Grady, one defines for all integers $\ell\ge 0$ closed subschemes
 \begin{alignat*}{5}
& Y^{\ge\ell}_A 		&& =  \{[ {U_1}]\in\P(V_6) 	&& \mid  \dim(A\cap ( {U_1} \wedge\bw{2}{V_6})) \ge \ell \}  && \subset  \P(V_6), \\
& Y^{\ge\ell}_{A^\bot} 	&& =  \{[U_5]\in\P(V^\vee_6)  && \mid \dim(A \cap \bw3U_5) \ge \ell \}  			&& \subset \PP(V_6^\vee),
\end{alignat*}
 and set
\begin{equation*} 
Y_A^\ell := Y_A^{\ge \ell} \setminus Y_A^{\ge \ell + 1}\qquad\text{and}\qquad  Y_{A^\bot}^\ell := Y_{A^\bot}^{\ge \ell} \setminus Y_{A^\bot}^{\ge \ell + 1}.
\end{equation*}

 Assume that $A$ has no decomposable vectors.\ Then, 
\begin{equation*}
Y_A := Y^{\ge 1}_A \subset \P(V_6)
\qquad\text{and}\qquad 
Y_{A^\perp} := Y^{\ge 1}_{A^\perp} \subset \P(V_6^\vee)
\end{equation*}
are normal  integral sextic hypersurfaces, called {\sf EPW sextics}; the singular locus of $Y_A$ is the integral surface $Y_A^{\ge 2}$, 
the singular locus of $Y_A^{\ge 2}$ is the finite set $Y_A^{\ge 3}$  {(empty for $A$ general)},  $Y_A^{\ge 4}=\emptyset$ (\cite[Proposition~B.2]{DK}),  and analogous properties hold for $Y_{A^\perp}^{\ge \ell}$.\ One can rewrite the dimensions of the GM varieties  $X$ associated with a Lagrangian data $(V_6,V_5,A)$ as follows: if the Pl\"ucker point $\bp_X $ is in $ Y^\ell_{A^\perp}$, we have
\begin{equation*}
\dim( X_\ord(V_6,V_5,A) )= 5 - \ell 
\qquad\text{and}\qquad 
\dim (X_\spe(V_6,V_5,A)) = 6 - \ell.
\end{equation*}

Still under the assumption that  $A$ contains no decomposable vectors,
O'Grady constructs in \cite[Section~1.2]{og4}  a canonical double cover
\begin{equation}\label{defyt}
f_A\colon \tY_A \lra Y_A 
 \end{equation}
branched over the integral surface $Y_A^{\ge 2} $.\ When the finite set $Y_A^{\ge 3} $ is empty, $\tY_A$ is a smooth hyperk\"ahler fourfold (\cite[Theorem~1.1(2)]{og1}).

The hypersurfaces $Y_A$ and $Y_{A^\perp}$ are mutually projectively dual and the duality is realized, 
inside the flag variety $\Fl(1,5;V_6):=\{(U_1,U_5)\in \P(V_6)\times \P(V_6^\vee)\mid U_1\subset U_5\subset V_6\}$, by the correspondence
 \begin{equation*}
\widehat{Y}_A := \{ ( U_1,U_5) \in \Fl(1,5;V_6) \mid  A \cap (U_1 \wedge \bw2U_5) \ne 0 \}
\end{equation*}
 (\cite[Proposition~B.3]{DK}) with its  {birational} projections 
 \begin{equation*}
\xymatrix@R=5mm{
& \widehat{Y}_A \ar@{->>}[dl]_-{\pr_{Y,1}} \ar@{->>}[dr]^-{\pr_{Y,2}} \\
 \ \ \P(V_6)\supset Y_A\ \ \quad&&\qquad\ \ Y_{A^\perp}\subset \P(V_6^\vee)
} 
\end{equation*}
{(these projections were denoted by $p$ and $q$ in \cite[Proposition~B.3]{DK}; we change the notation to $\pr_{Y,1}$ and $\pr_{Y,2}$ in this article, but we will switch back to $p$ and $q$ in Appendix~\ref{section:epw-surface})}.

We will need the following result.

\begin{lemm}\label{lemma:dual-y2}
Assume that $A$ has no decomposable vectors.\ If $E \subset \widehat{Y}_A$ 
is the exceptional divisor of the map $\pr_{Y,1}$, we have  two inclusions
\begin{equation*}
Y_{A^\perp}^{\ge 2 } \subset \pr_{Y,2}(E) \subset (Y^{\ge 2}_A)^\vee \cap Y_{A^\perp},
\end{equation*}
where $(Y^{\ge 2}_A)^\vee \subset \P(V_6^\vee)$ is the projective dual of $Y^{\ge 2}_A$.
\end{lemm}

\begin{proof}
Since $Y^{\ge 2}_A$  is smooth at points of $Y^2_A$, its projective dual is
\begin{equation*}
\left(Y^{\ge 2}_A\right)^\vee = \overline{\bigcup_{v \in Y^2_A} \langle v \wedge \xi_1 \wedge \xi_1, v \wedge \xi_1 \wedge \xi_2, v \wedge \xi_2 \wedge \xi_2 \rangle}, 
\end{equation*}
where we write $A \cap (v \wedge\bw2V_6) = \langle v\wedge\xi_1,v\wedge\xi_2 \rangle$  {for some $\xi_1,\xi_2 \in \bw2V_6$}, and identify $\bw5V_6 $ with~$V_6^\vee$.\  
Indeed, a vector $v'\in V_6$ is tangent to $Y^2_A$ at $v$ if one has
\begin{equation*}
(v + tv') \wedge (\xi_i + t \xi'_i) = a_i + t a'_i \pmod {t^2}
\end{equation*}
for some $\xi'_i \in \bw2V_6$ and $a'_i \in A$, for $i \in\{ 1,2\}$.\ Since $A$ is Lagrangian, this implies, 
{for $i,j \in \{1,2\}$,}
\begin{equation*}
0 = (v \wedge \xi_i) \wedge (v \wedge \xi'_j + v' \wedge \xi_j) = -v' \wedge (v \wedge \xi_i \wedge \xi_j).
\end{equation*}
This means that the  {embedded} tangent space to $Y^2_A$ at $v$ is contained in the orthogonal to the subspace of $V_6^\vee$ generated by $v \wedge \xi_i \wedge \xi_j$.\ Since the former, modulo $v$, is 2-dimensional  and the latter is 3-dimensional,  the tangent space coincides with this orthogonal, hence the above description of the dual variety.

On the other hand, by the argument in the proof of \cite[Proposition~B.3]{DK}, one has
\begin{equation*}
\pr_{Y,2}(E) = \bigcup_{v \in Y^{\ge 2}_A} 
 \bigcup_\xi \ v \wedge \xi \wedge \xi, 
\end{equation*}
where the second union is taken over all $v\wedge \xi \in A \cap (v \wedge \bw2V_6)$.\ In particular, we obtain the inclusion $\pr_{Y,2}(E) \subset \left(Y^2_A\right)^\vee$.  

For the second inclusion, since $Y_{A^\perp}^{\ge 2}$ is an integral surface (\cite[Theorem~B.2]{DK}),  
it is enough to show that $E$ intersects the general fiber $C$ of the map $E' = \pr_{Y,2}^{-1}(Y_{A^\perp}^{\ge 2}) \to Y_{A^\perp}^{\ge 2}$.
This fiber is mapped by $\pr_{Y,1}$ to a conic in $Y_A$ (\cite[Proposition~B.3]{DK}), hence $H\cdot C = 2$, where $H$ is the pullback of the hyperplane class of $Y_A$.
On the other hand, if $H'$ is the hyperplane class of~$Y_{A^\perp}$, then $H' \cdot C = 0$.
But $E$ is linearly equivalent to $5H - H'$ (\cite[proof of Lemma~B.5]{DK}), hence $E \cdot C = 10$, hence $E$ intersects $C$ non-trivially.
This finishes the proof of the lemma.
\end{proof}

Given a Lagrangian subspace $A \subset \bw3V_6$, one can also define the  closed subschemes
\begin{equation*}
Z^{\ge \ell}_A := \{ U_3 \subset V_6 \mid \dim (A \cap (\bw2U_3 \wedge V_6)) \ge \ell \} \subset \Gr(3,V_6). 
\end{equation*}
The {complements} $Z_A^\ell := Z_A^{\ge \ell} \setminus Z_A^{\ge \ell + 1}$ form a stratification of $\Gr(3,V_6)$.\ If $A$ has no decomposable vectors,  $Z_A := Z_A^{\ge 1}$ is a normal integral hypersurface in $\Gr(3,V_6)$ cut out by a quartic hypersurface in $\P(\bw3V_6)$.\ We call $Z_A$ an {\sf EPW quartic}.\ The singular locus of $Z_A$ is then the integral  {variety} $Z_A^{\ge 2}$ of dimension~6, the singular locus of $Z_A^{\ge 2}$ is the integral variety $Z_A^{\ge 3}$ of dimension~3, the singular locus of $Z_A^{\ge 3}$ is the finite set $Z_A^{\ge 4}$ (empty for~$A$ general), and $Z_A^{\ge 5} = \emptyset$  (\cite[Proposition~2.6]{ikkr}).

Moreover, there is a canonical double cover $ \widetilde{Z}_A^{\ge 2} \to Z_A^{\ge 2}$ branched over $Z_A^{\ge 3}$, and when~$Z_A^{\ge 4}$ is empty, $\widetilde{Z}_A^{\ge 2}$ is a smooth hyperk\"ahler sixfold (\cite[Theorem~1.1]{ikkr}).  
 
 The hypersurfaces $Z_A \subset \Gr(3,V_6)$ and $Z_{A^\perp} \subset \Gr(3,V_6^\vee)$ coincide under the natural identification $\Gr(3,V_6) \cong \Gr(3,V_6^\vee)$.\ They are related to the EPW sextics via the correspondence 
 \begin{equation*}
\widehat{Z}_A := \{ ( U_3, U_5) \in  {\Fl(3,5;V_6)} \mid  A \cap (\bw2U_3 \wedge U_5)\ne 0 \}
\end{equation*}
with  its projections
\begin{equation*} 
\vcenter{\xymatrix@R=5mm{
& \widehat{Z}_A \ar[dl]_{\pr_{Z,1}} \ar[dr]^{\pr_{Z,2}} \\
Z_A && Y_{A^\perp}.
}} 
\end{equation*}

\begin{lemm}
Assume that the Lagrangian $A$ contains no decomposable vectors.\ The map $\pr_{Z,2}$ is dominant; over $Y_{A^\perp}^1$, it is smooth and its fibers are $3$-di\-men\-si\-onal quadrics.\ The map $\pr_{Z,1}$ is birational onto a divisor in $Z_A$ containing $Z_A^{\ge 2}$.
\end{lemm}

\begin{proof}
Let $[U_5]$ be a point of $Y^1_{A^\perp}$ and let $a$ be a generator of the one-dimensional space~\mbox{$A \cap \bw3U_5$}.\ The 2-form on $U_5$ corresponding to $a \in \bw3U_5$ via the isomorphism $\bw3U_5 \cong \bw2U_5^\vee$ has rank~4 (because $a$ is not decomposable).\ The fiber $\pr_{Z,2}^{-1}([U_5])$ parameterizes all 3-dimensional subspaces $U_3 $ of $ U_5$ which are isotropic for the 2-form $a$.\ Since $a$ has rank 4, it is a smooth 3-dimensional quadric.

Analogously, let $[U_3]$ be a point of $Z_A^1$ and let $a$ be a generator of the one-dimensional space~\mbox{$A \cap (\bw2U_3 \wedge V_6)$}.\ Let $\bar{a}$ be the image of $a$ in the space $\bw2U_3 \otimes (V_6/U_3) \cong \Hom(U_3^\perp,\bw2U_3)$.\ Over $Z^1_A$, this defines a morphism of rank-3 vector bundles with fibers $U_3^\perp$ and $\bw2U_3$.\ Over its degeneracy locus (which a divisor in $Z_A$), the projection $\pr_{Z,1}$ is an isomorphism 
(and $\pr_{Z,1}^{-1}([U_3])$ is the unique hyperplane $U_5 \subset V_6$ such that $U_5^\perp \subset U_3^\perp$ is the kernel of $\bar{a}$).\ To prove  $Z_A^{\ge 2} \subset \pr_{Z,1}(\widehat{Z}_A)$, note that if $\dim(A \cap (\bw2U_3 \wedge V_6)) \ge 2$,  we have a pencil of maps $U_3^\perp \to \bw2U_3$;
some maps in this pencil are degenerate and their kernels give points in the preimage of $[U_3]$.
\end{proof}

For any hyperplane $V_5 \subset V_6$, we set
\begin{equation*}
\begin{aligned}
Y^\ell_{A,V_5} 	&:= Y^\ell_A \times_{\P(V_6)}\P(V_5) 		&&= Y^\ell_A \cap \P(V_5),\\
Z^\ell_{A,V_5} 	&:= Z^\ell_A \times_{\Gr(3,V_6)}\Gr(3,V_5) 	&&= Z^\ell_A \cap \Gr(3,V_5),\\
 \tY_{A,V_5} 	&:= \tY_A \times_{\P(V_6)}\P(V_5)		&&= f_A^{-1}(Y_{A,V_5}),
\end{aligned}
\end{equation*}
 and similarly for $Y^{\ge \ell}_{A,V_5}$ and others.\ These varieties will play an important role for the geometry of the associated GM varieties.\ We let
\begin{equation*}
f_{A,V_5} \colon  \tY_{A,V_5} \to Y_{A,V_5}
\end{equation*}
be the morphism induced by restriction of the double cover $f_A \colon \tY_A \to Y_A$.

 We will need the following simple observation.

\begin{lemm}\label{lemma:zav4-empty}
Let $(V_6,V_5,A)$ be a Lagrangian data with no decomposable vectors in $A$.\ If \mbox{$[U_3] \in Z_{A,V_5}^4$}, then $(U_3,V_5) \in \widehat{Z}_A$.\ In particular, if $A \cap \bw3V_5 = 0$, we have $Z_{A,V_5}^{\ge4} = \emptyset$.
\end{lemm}

\begin{proof}
Assume that $U_3 \subset V_6$ defines a point of $Z_{A,V_5}^4$.\ In other words, $\dim(A \cap (\bw2U_3 \wedge V_6)) \ge 4$ and $U_3 \subset V_5$.\ Since $\bw2U_3 \wedge V_5 $ has codimension 3 in $ \bw2U_3 \wedge V_6$, we have $A \cap  (\bw2U_3 \wedge V_5) \ne 0$.\ This means $(U_3,V_5) \in \widehat{Z}_A$.\ Since $\bw2U_3 \wedge V_5 \subset \bw3V_5$, this contradicts $A \cap \bw3V_5 = 0$.
\end{proof}

\subsection{The quadric fibrations}\label{sec24}

In~\cite[{Section~4}]{DK}, we defined two quadric fibrations associated with a  smooth  GM variety $X$  {of dimension $n$}.\ The first quadric fibration is the map 
\begin{equation*}
\rho_1 \colon  \PP_X(\cU_X) \to \P(V_5)
\end{equation*}
induced by the tautological embedding $\cU_X \hookrightarrow V_5 \otimes \cO_X$.\ It is flat over the complement of {the union $Y_{A,V_5}^{\ge n-1} \cup \Sigma_1(X)$, where $\Sigma_1(X)$ is} the {\sf kernel locus}
\begin{equation}\label{eq:kernel-locus}
\Sigma_1(X) := \pr_{Y,1}(\pr_{Y,2}^{-1}(\bp_X)) \subset Y_{A,V_5} \subset \P(V_5).
\end{equation}
If $A$ has no decomposable vectors, the map $\pr_{Y,1}\colon \pr_{Y,2}^{-1}(\bp_X) \to Y_{A,V_5}$ is a closed embedding {(\cite[Proposition~B.3]{DK})}.\  So, if $\bp_X \in Y^\ell_{A^\perp}$, the variety $\Sigma_1(X) $ is isomorphic to $ \P^{\ell - 1}$ embedded via the second Veronese embedding.

The fibers of $\rho_1$ can be described as follows.

\begin{lemm}[\protect{\cite[Proposition~4.5]{DK}}]
\label{lemma:fibers-quadric1}
Let $X$ be a smooth GM variety of dimension  $n\ge3$, with associated Lagrangian data $(V_6,V_5,A)$.\ For every $v \in \P(V_5)$, we have
\begin{enumerate}[noitemsep,nolistsep]
\item[\rm(a)] if $v \in Y^\ell_{A,V_5} \setminus \Sigma_1(X)$, the fiber $\rho_1^{-1}(v)$ is a quadric {in $\P^{n-2}$ of} corank $\ell$;
\item[\rm(b)] if $v \in Y^\ell_{A,V_5} \cap \Sigma_1(X)$, the fiber $\rho_1^{-1}(v)$ is a quadric {in $\P^{n-1}$ of} corank $\ell-1$.
\end{enumerate}
\end{lemm}

 {Since the corank of a quadric does not exceed the linear dimension of its span, we have $\ell \le n-1$ for $v \notin \Sigma_1(X)$, and $\ell \le n+1$ for $v \in \Sigma_1(X)$.\ This implies $Y^3_{A,V_5} \subset \Sigma_1(X)$ for $n=3$.}

The second quadric fibration is the map 
\begin{equation*}
\rho_2 \colon  \PP_X(V_5/\cU_X) \to \Gr(3,V_5)
\end{equation*}
induced by the natural embedding $(V_5/\cU_X) \otimes \bw2\cU_X \hookrightarrow \bw3V_5\otimes \cO_X$.\ 
It is flat over the complement of {the union $Z_{A,V_5}^{\ge n-2} \cup \Sigma_2(X)$, where $\Sigma_2(X)$ is} the {\sf isotropic locus}
\begin{equation}\label{eq:isotropic-locus}
\Sigma_2(X) := \pr_{Z,1}(\pr_{Z,2}^{-1}(\bp_X)) \subset Z_{A,V_5} \subset \Gr(3,V_5).
\end{equation}
By Lemma~\ref{lemma:zav4-empty}, we have
$
Z_{A,V_5}^4 \subset \Sigma_2(X)$.

In contrast with the case of the kernel locus, the map $\pr_{Z,1}\colon \pr_{Z,2}^{-1}(\bp_X) \to Z_{A,V_5}$ is no longer an embedding: its fiber over a point $U_3$ is   the projective space $\P(A \cap (\bw2U_3 \wedge V_5))$ and we set
\begin{equation*}
{\Sigma^{\ge k}_2(X) := \{ U_3 \in \Sigma_2(X) \mid \dim(A \cap (\bw2U_3 \wedge V_5)) \ge k \}
\quad\textnormal{and}\quad
\Sigma^k_2(X) := \Sigma^{\ge k}_2(X) \setminus \Sigma^{\ge k+1}_2(X),}
\end{equation*}
so that $\Sigma_2(X) = \Sigma^{\ge 1}_2(X)$.\ Note that $\Sigma_2^{\ge 3}(X) $ is empty if~$A$ has no decomposable vectors.

The fibers of $\rho_2$ can be described as follows.

\begin{lemm}[\protect{\cite[Proposition~4.10]{DK}}]
\label{lemma:fibers-quadric2}
Let $X$ be a smooth GM variety of dimension {$n \ge 3$}, with associated Lagrangian data $(V_6,V_5,A)$.\ For every $U_3 \in \Gr(3,V_5)$, we have
\begin{itemize} 
\item[{\rm (a)}] if $U_3 \in Z^\ell_{A,V_5} \setminus \Sigma_2(X)$, the fiber $\rho_2^{-1}(U_3)$ is a quadric {in $\P^{n-3}$ of} corank $\ell$;
\item[{\rm (b)}] if $U_3 \in Z^\ell_{A,V_5} \cap \Sigma^1_2(X)$, the fiber $\rho_2^{-1}(U_3)$ is a quadric {in $\P^{n-2}$ of} corank $\ell-1$;
\item[{\rm (c)}] if $U_3 \in Z^\ell_{A,V_5} \cap \Sigma^2_2(X)$, the fiber $\rho_2^{-1}(U_3)$ is a quadric {in $\P^{n-1}$ of} corank $\ell-2$.
\end{itemize}
\end{lemm}

Lemmas \ref{lemma:fibers-quadric2} and \ref{lemma:singularities} will be essential for the descriptions of the schemes of linear spaces contained in GM varieties.

\begin{lemm}\label{lemma:singularities}
Let $A\subset \bw3V_6$ be a Lagrangian subspace with no decomposable vectors,  {let $V_5 \subset V_6$ be a hyperplane},
and let $X = X_\ord(V_6,V_5,A)$ be the corresponding ordinary GM variety, of dimension $n:= 5-\dim(A \cap \bw3V_5)$.\ If $ n\ge 3$,
\begin{itemize} 
\item[{\rm (a)}] $Y^{\ge 2}_{A,V_5}$ is a curve which is smooth if and only if $Y^3_{A,V_5} = \emptyset$ and the Pl\"ucker point $\bp_X$ does not lie on the projective dual variety of $Y^{\ge 2}_A$;
\item[{\rm (b)}] $Y_{A,V_5}$ is a normal integral threefold and
\begin{equation*}
\Sing(Y_{A,V_5}) = Y^{\ge 2}_{A,V_5} \cup \Sigma_1(X),
\qquad
\Sing(\tY_{A,V_5}) = \Sing(Y^{\ge 2}_{A,V_5}) \cup f_{{A,V_5}}^{-1}(\Sigma_1(X)).
\end{equation*}
\end{itemize}
\end{lemm}

\begin{proof}
(a) 
The integral surface $Y^{\ge 2}_A$ is  not contained in a hyperplane (\cite[Lemma~B.6]{DK}),
hence its hyperplane section $Y^{\ge 2}_{A,V_5}$ is a curve.\ The statement about smoothness follows from the definition of projective duality.

(b)
Since $Y_A$ is an integral sextic, we have $\dim (Y_{A,V_5}) = 3$.\ If a point $P \in Y_{A,V_5} \setminus Y_{A,V_5}^{\ge 2}$
is singular, the tangent space to $Y_A$ at $P$ coincides with $\P(V_5)$.\ Therefore, $(P, \bp_X) \in \hY_A$,
hence~$P \in \Sigma_1(X) = \pr_{Y,1}(\pr_{Y,2}^{-1}(\bp_X))$.\ On the other hand, all points of $Y^{\ge 2}_{A,V_5}$
are singular on~$Y_{A,V_5}$, since $Y^{\ge 2}_A = \Sing(Y_A)$.\ This gives the required description of $\Sing(Y_{A,V_5})$.

Over $Y^1_{A,V_5}$, the map $f_{A,V_5}$ is \'etale, hence the singular locus of $\tY_{A,V_5}$  {over $Y_{A,V_5}^1$ is equal to} $f_{A,V_5}^{-1}(\Sigma_1(X))$.\ On the other hand, one checks that along the ramification locus $f_{{A,V_5}}^{-1}(Y^{\ge 2}_{A,V_5})  $, the double sextic $\tY_{A,V_5}$ is smooth if and only if $Y^{\ge 2}_{A,V_5}$ is.\ This gives the required description of~$\Sing(\tY_{A,V_5})$.

Finally, $Y_{A,V_5}$ is normal and integral because it is a hypersurface in~$\P^4$ 
with 1-dimensional singular locus ($Y^{\ge 2}_{A,V_5}$ has dimension 1 by part (a) 
and $\dim(\Sigma_1(X) )= (5-n)-1 = 4-n \le 1$ when $n \ge 3$  {by~\eqref{eq:kernel-locus} and the discussion after it}).
\end{proof}

\section{Cohomology of smooth GM varieties}
\label{section:cohomology}

\subsection{Hodge numbers}

Recall that the Hodge diamond of $\Gr(2,V_5)$ is
\begin{equation*}
\begin{smallmatrix}
&& &&&& 1 \\
&& &&& 0 && 0  \\
& &&&0&& 1 &&0 \\
&& &0&& 0 && 0 &&0  \\
&&0 &&0&& 2 &&0&&0 \\
& 0&& 0 && 0 &&0&&0&&0\\
  0&& 0 && 0 &&2&&0&&0&&0  \\
& 0&& 0 && 0 &&0&&0&&0\\
&&0 &&0&& 2 &&0&&0 \\
&& &0&& 0 && 0 &&0  \\
& &&&0&& 1 &&0 \\
&& &&& 0 && 0  \\
&& &&&& 1 
\end{smallmatrix} 
\end{equation*}
The  abelian group  $H^\bullet(\Gr(2,V_5);\Z)$ is  free  with basis the Schubert classes 
\begin{equation*}
\bsi_{i,j}  \in H^{2(i+j)}(\Gr(2,V_5),\Z), \qquad 3\ge i\ge j\ge 0.
\end{equation*}
We write $\bsi_i$ for $\bsi_{i,0}$; thus $\bsi_1$ is the hyperplane class,
$\bsi_{1,1} = c_2(\cU)$, 
and $\bsi_i = c_i(V_5/\cU)$,
where $\cU$ is the tautological rank-2 subbundle and $V_5/\cU$ the tautological rank-3 quotient bundle.

We compute the Hodge numbers of smooth GM varieties.

\begin{prop}\label{hn}
The Hodge diamond  of a smooth complex GM variety  of dimension $n$ is  
\begin{equation*} 
\begin{array}{cccccc}
(n=1) & (n=2) & (n = 3) & (n=4) & (n=5) & (n = 6) \\[1ex]  
\begin{smallmatrix}
 &1 \\
6 && 6 \\
& 1
\end{smallmatrix} &
\begin{smallmatrix}
&& 1 \\
 & 0&&0 \\
1&& 20&&1  \\
 & 0&&0 \\
&& 1 
\end{smallmatrix} &
\begin{smallmatrix}
 &&& 1 \\
  && 0 && 0  \\
&0&& 1 &&0 \\
 0&& 10 && 10 &&0  \\
&0&& 1 &&0 \\
  && 0 && 0  \\
 &&& 1 
\end{smallmatrix} &
\begin{smallmatrix}
&& &&&& 1 \\
&& &&& 0 && 0  \\
& &&&0&& 1 &&0 \\
&& &0&& 0 && 0 &&0  \\
&&0 &&1&& 22 &&1&&0 	\\
&& &0&& 0 && 0 &&0  \\
& &&&0&& 1 &&0 \\
&& &&& 0 && 0  \\
&& &&&& 1 
\end{smallmatrix} &
\begin{smallmatrix}
&& &&&& 1 \\
&& &&& 0 && 0  \\
& &&&0&& 1 &&0 \\
&& &0&& 0 && 0 &&0  \\
&&0 &&0&& 2 &&0&&0 \\
& 0&& 0 && 10 &&10&&0&&0  \\
&&0 &&0&& 2 &&0&&0 \\
&& &0&& 0 && 0 &&0  \\
& &&&0&& 1 &&0 \\
&& &&& 0 && 0  \\
&& &&&& 1 
\end{smallmatrix} &
\begin{smallmatrix}
&& &&&& 1 \\
&& &&& 0 && 0  \\
& &&&0&& 1 &&0 \\
&& &0&& 0 && 0 &&0  \\
&&0 &&0&& 2 &&0&&0 \\
& 0&& 0 && 0 &&0&&0&&0\\
  0&& 0 && 1 &&22&&1&&0&&0  \\
& 0&& 0 && 0 &&0&&0&&0\\
&&0 &&0&& 2 &&0&&0 \\
&& &0&& 0 && 0 &&0  \\
& &&&0&& 1 &&0 \\
&& &&& 0 && 0  \\
&& &&&& 1 
\end{smallmatrix} 
\end{array}
\end{equation*}
 \end{prop}

 \begin{proof}
When $n=1$, the Hodge numbers are those of a curve of genus 6.\ When $n=2$, the Hodge numbers are those of a K3 surface.

Assume $3\le n\le 5$.\  Since the Hodge numbers of smooth complex varieties are deformation invariant, we may assume that the GM variety $X$ is ordinary.\ It is then a smooth dimensionally transverse intersection of (ample) hypersurfaces in  $ G:=\Gr(2,V_5)$ and the Lefschetz Hyperplane Theorem (see Lemma \ref{lef}) implies that the Hodge numbers of $X$ of degree $<n$ are those of $G$.\ Moreover, $h^{n,0}(X) = 0$ because $X$ is a Fano variety.
 
When $n=3$, the missing Hodge number $h^{1,2}(X)$ was computed in \cite{lo}.\ 
When $n=4$, the   Hodge diamond  was computed in \cite[Lemma 4.1]{im}.\ 
When $n=6$, it was computed in \cite[Corollary~4.4]{DK-6fold}.\ When $n=5$, the missing Hodge numbers $h^{1,4}(X)$ and $h^{2,3}(X)$ were obtained by Nagel using a computer (see the introduction of \cite{nag}).\  We now present our own computation.

To compute $h^{1,4}(X)$, we assume that $X$ is an ordinary fivefold.\  Consider the exact sequences
\begin{equation*} 
0\to \cO_X(-2) \to \Omega^1_G \vert_X\to \Omega^1_X\to 0\quad{\rm and}\quad 0\to \Omega^1_G(-2) \to \Omega^1_G  \to \Omega^1_G \vert_X\to 0. 
\end{equation*}
The sheaf $\Omega^1_G(-2)$ is acyclic (by Bott's theorem) and so is $ \cO_X(-2)$ (by Kodaira vanishing), hence
$h^i(X,\Omega^1_X) = h^i(X,\Omega^1_G\vert_X) = h^i(G,\Omega^1_G)$ and $h^{1,i}(X) = h^{1,i}(G)$.\  In particular, we obtain $h^{1,4}(X)=0$. 

To compute  $h^{2,3}(X) $, we assume that $X$ is a special fivefold, \ie, is a double covering  of a smooth hyperplane section $M'_X$ of $G$ branched along a smooth GM fourfold $X'$.\ Using this double covering, we compute Euler characteristics 
\begin{equation*}
\chi_{\rm top}(X)=2\chi_{\rm top}(M'_X)-\chi_{\rm top}(X').
\end{equation*}
Since $X'$ is a GM fourfold, we have $\chi_{\rm top}(X')=1+1+24+1+1=28 $.\ On the other hand, the inclusion $M'_X\subset G$ induces isomorphisms $H^k(G;\Z)\isom H^k(M'_X;\Z)$ for all $k\in\{0,\dots,5\}$.\ In particular, $\chi_{\rm top}(M'_X)=8$, hence  $\chi_{\rm top}(X)=-12$.\  Since $\chi_{\rm top}(X)=1+1+2-2h^{2,3}(X)+2+1+1$, we obtain $h^{2,3}(X) = 10$.\ This finishes the proof of the proposition. 
 \end{proof}

\subsection{Integral cohomology}

We now prove that the integral cohomology groups of smooth GM varieties are  torsion-free.\ We start with a classical lemma. 

\begin{lemm}[Lefschetz]\label{lef}
Let 
$X$ be a dimensionally transverse intersection of dimension $n$ of ample hypersurfaces  in a smooth projective variety $M$. 

\noindent{\rm (a)} The induced map  $H^k(M;\Z)\isomto H^k(X;\Z)$ is bijective for $k< n$, injective for $k=n$.

\noindent{\rm (b)} The induced map  $H_k(X;\Z)\isomto H_k(M;\Z)$ is bijective for $k< n$, surjective for $k=n$.

\noindent{\rm (c)} If $X$ is moreover smooth and $H^\bullet(M;\Z)$ is torsion-free, so is $H^\bullet(X;\Z)$.
\end{lemm}

\begin{proof} 
Parts (a) and (b) are the Lefschetz Hyperplane Theorem 
and follow  from the fact that $M\setminus X$ is the union of $\dim(M)-n$ smooth affine open subsets (\cite[Chapter 5, Theorem~(2.6)]{di}).  

For (c),
{since} $X$ is smooth, the Poincar\'e duality isomorphisms $H^k(X;\Z)\isom H_{2n-k}(X;\Z)$ and (a)  {together with (b)} imply  that the integral homology and  cohomology groups of $X$ are torsion-free in all degrees except perhaps $n$.\ By the Universal Coefficient Theorem, the torsion subgroup of $H^n(X;\Z)$ is isomorphic to the torsion subgroup of $H_{n-1}(X;\Z)$, which is 0 by (b), hence all integral cohomology groups of $X$ are torsion-free.
\end{proof}

A similar result holds for cyclic covers (this is the main theorem of \cite{cor}; see also the remarks at the very end of the article).

\begin{lemm}\label{lef2}
Let $\gamma\colon X\to M$ be a cyclic cover between
smooth projective varieties of  dimension~$n$ whose branch locus is a smooth ample divisor on $M$. 

\noindent{\rm (a)} The induced map  $\gamma^*\colon H^k(M;\Z)\isomto H^k(X;\Z)$ is bijective for $k< n$, injective for $k=n$.

\noindent{\rm (b)} The induced map  $\gamma_*\colon H_k(X;\Z)\isomto H_k(M;\Z)$ is bijective for $k< n$, surjective for $k=n$.

\noindent{\rm (c)} If $H^\bullet(M;\Z)$ is torsion-free, so is $H^\bullet(X;\Z)$.
\end{lemm}

We now describe the integral cohomology groups of smooth GM varieties.

\begin{prop}\label{intcoh}
Let $X$ be a smooth GM variety of dimension $n$. 

\noindent{\rm (a)}
The group $H^\bullet(X;\Z)$ is torsion-free.
 
\noindent{\rm (b)} The map $\gamma_X^{*,k}\colon H^k(\Gr(2,V_5);\Z)\to H^k(X;\Z)$ is bijective for $k<n$ and injective for $k=n$.
 \end{prop}

\begin{proof}
When $X$ is ordinary, it is a dimensionally transverse intersection of (ample) hypersurfaces in  $\Gr(2,V_5)$,
hence Lemma~\ref{lef} implies both parts (a) and (b) of the proposition.

When $X$ is special, its Gushel map factors as $\gamma_X\colon X\xrightarrow{\ \gamma\ } M'_X\hra \Gr(2,V_5)$, 
where $\gamma$ is a double cover branched along an ample divisor, and $M'_X$ is a dimensionally transverse intersection of (ample) hypersurfaces in  $\Gr(2,V_5)$.\ Both parts (a) and (b) are then consequences of Lemma~\ref{lef} (applied to $M'_X\subset \Gr(2,V_5)$) and Lemma~\ref{lef2} (applied to the double cover $\gamma$).
\end{proof}

 \begin{coro}\label{deg}
Let $X$ be a smooth GM variety of dimension $n$.\  If $n \ge 3$, the degree of any hypersurface in $X$ is divisible by $10$.\  If $n \ge 5$, the degree of any subvariety of codimension $2$ in~$X$ is even.
  \end{coro}

\begin{proof} We use Proposition \ref{intcoh}(b).\  Let $Y\subset X$ be a subvariety of codimension $c$.\  If $c=1$,  the class of $Y$ in $H^2(X;\Z)$  is a multiple of the class $\gamma_X^*\bsi_1$, which has degree 10.

If $c=2$ (and $n\ge5$), the class of $Y$  in $H^4(X;\Z)$ is an integral combination  of $\gamma_X^*\bsi_2$, which has degree 6, and $\gamma_X^*\bsi_{1,1}$, which has degree 4.\  The degree of $Y$ is therefore even.
\end{proof}

We will need the following computation, which was already used in~\cite{dims}.

\begin{lemm}\label{lemma:class-q0}
Let $X$ be a smooth ordinary GM fourfold and let $Q_0 \subset X$ be its $\sigma$-quadric, \ie, the intersection of $X$ with the 3-space $\Pi := \P(v_0 \wedge V_5) \subset M_X$, 
where $v_0 \in \P(V_5)$ is the unique point in the kernel locus $\Sigma_1(X)$ defined by~\eqref{eq:kernel-locus}.\  Then $[Q_0] = \gamma_X^*(\bsi_2 - \bsi_{1,1}) \in H^4(X;\Z)$.
\end{lemm}

\begin{proof}
Let $\gamma_{M_X}$ be the inclusion $M_X\hra \Gr(2,V_5)$.\ By the Lefschetz Theorem (Lemma \ref{lef}), the map $\gamma_{M_X}^*:H^4(\Gr(2,V_5);\Z)\to H^4(M_X;\Z)$ is an isomorphism.\ Therefore, there exist integers $a$ and $b$ such that $[\Pi]=\gamma_{M_X}^*(a\bsi_2 +b \bsi_{1,1}) $, hence $[Q_0] = \gamma_X^*(a\bsi_2 +b \bsi_{1,1})$.\  Since the class of $\Pi$ in $H^6(\Gr(2,V_5);\Z)$ is $\bsi_3$, Gysin's formula and Schubert calculus give
\begin{equation*}
\bsi_3=\gamma_{M_X*}([\Pi])=\gamma_{M_X*}\gamma_{M_X}^*(a\bsi_2 +b \bsi_{1,1})=(a\bsi_2 +b \bsi_{1,1})\cdot \bsi_1=a(\bsi_3+\bsi_{2,1})+ b \bsi_{2,1}
\end{equation*}
in $H^6(\Gr(2,V_5);\Z)$.\ This implies $a=1$ and $a+b=0$, hence the lemma.
\end{proof}

 The following lemma  is also useful; we keep the notation of Lemma~\ref{lemma:class-q0}.

\begin{lemm}\label{lemma:ux-q0-splits}
Let $X$ be a smooth ordinary GM fourfold.\ The restriction of the bundle $\cU_X$ to the quadric~$Q_0$ splits as $ \cO_{Q_0} \oplus \cO_{Q_0}(-1)$.
\end{lemm}

\begin{proof}
Since $Q_0 \subset \Pi$, it is enough to show $\cU_X\vert_\Pi \cong \cO_\Pi \oplus \cO_\Pi(-1)$.\ Recall that $\Pi = \P(v_0 \wedge V_5)$ parameterizes all two-dimensional subspaces in $V_5$ that contain $v_0$.\ Consequently, we have an injection of vector bundles $\cO_\Pi \hookrightarrow \cU_X\vert_\Pi$ given by the vector $v_0$.\ Its cokernel is a line bundle isomorphic to $\det(\cU_X\vert_\Pi) \cong \cO_\Pi(-1)$, hence we have an exact sequence
\begin{equation*}
0 \to \cO_\Pi \to \cU_X\vert_\Pi \to \cO_\Pi(-1) \to 0.
\end{equation*}
It remains to note that $\Ext^1(\cO_\Pi(-1),\cO_\Pi) = H^1(\Pi,\cO_\Pi(1)) = 0$ since $\Pi \cong \P^3$.
\end{proof}

\subsection{Middle cohomology lattices of smooth GM varieties of dimension 4 or 6}\label{mcl}

Let~$X$ be a smooth GM variety of even dimension $n$ with Gushel map $ \gamma_X\colon X\to \Gr(2,V_5)$.\  The abelian group  $H^n(X;\Z)$ is torsion-free (Proposition~\ref{intcoh}) and, endowed with the intersection form, it is, by Poincar\'e duality, a unimodular lattice.\  We set $h:=\gamma_X^*\bsi_1\in H^2(X;\Z)$ and 
\begin{equation}\label{defh00}
\begin{aligned}
H^n(X;\Z)_0  & := \{x \in H^n(X;\Z) \mid x\cdot h = 0 \}, 
\\
H^n(X;\Z)_{00} & := \{ x \in H^n(X;\Z) \mid x\cdot \gamma_X^*(H^n(\Gr(2,V_5);\Z)) = 0 \}. 
\end{aligned}
\end{equation}  
These sublattices of $H^n(X;\Z)$ are called the {\sf primitive} and the {\sf vanishing} lattices of $X$.

\begin{lemm}\label{lemma:primitive-vanishing}
For every $n$, we have an injection $H^n(X;\Z)_{00} \subset H^n(X;\Z)_{0}$, and for $n = 2$ and $n = 6$, we have an equality $H^n(X;\Z)_{00} = H^n(X;\Z)_{0}$.
\end{lemm}

\begin{proof}
Since $h$ is pulled back from $\Gr(2,V_5)$, we have $(x \cdot h) \cdot \gamma_X^*(H^{n-2}(\Gr(2,V_5);\Z)) = 0$ for every $x \in H^n(X;\Z)_{00}$.\ By Lemma~\ref{lef}, the map $\gamma_X^* \colon H^{n-2}(\Gr(2,V_5);\Z) \to H^{n-2}(X;\Z)$ is an isomorphism, hence $(x\cdot h)\cdot H^{n-2}(X;\Z) = 0$.\ We conclude $x \cdot h = 0$ by Poincar\'e duality.

Since $H^2(\Gr(2,V_5);\Z)=\Z\bsi_1$, the definitions of $H^2(X;\Z)_0$ and $H^2(X;\Z)_{00}$ are the same for $n=2$.\ Furthermore, the product $H^4(\Gr(2,V_5);\Z) \xrightarrow{\ {}\cdot\bsi_1\ } H^6(\Gr(2,V_5);\Z)$ is an isomorphism by Schubert calculus,
hence for $n = 6$, the definitions are equivalent.
\end{proof}

Given a Lagrangian subspace $A \subset \bw3V_6$ with no decomposable vectors and such that $Y_A^{\ge3}=\vide$, the fourfold $\tY_A$ introduced in \eqref{defyt} is a  hyperk\"ahler manifold which is a deformation of the symmetric square of a K3 surface (\cite[Theorem~1.1(2)]{og1}).\  In particular, the group  $H^\bullet (\tY_A;\Z)$ is torsion-free (\cite[Theorem~1]{mark}).

We denote by $\tilde h\in H^2(\tY_A;\Z)$ the pullback by $f_A$ of the hyperplane class on $Y_A\subset \P(V_6)$ and
define the {\sf primitive} cohomology  
\begin{equation}\label{defprime}
H^2(\tY_A;\Z)_0 := \{ y \in H^2(\tY_A;\Z) \mid y\cdot \tilde h^3 = 0 \}.
\end{equation}
We consider $H^n(X,\Z)_{00}$ and $H^2(\tY_A,\Z)_0$ as polarized Hodge structures via the intersection pairing on the first and the Beauville--Bogomolov quadratic form $q_B$ on the second.\  Recall that~$q_B$
can be defined by  (\cite[{Theorem}~5(c)]{bea})
\begin{equation}\label{eq:bb-form}
\forall y \in H^2(\tY_A;\Z)_0\qquad q_B(y) = \tfrac12 y^2\cdot \tilde h^2.
\end{equation}
This form makes $H^2(\tY_A;\Z)_0$ into a lattice of rank~22.

Given a lattice $L$ and a non-zero integer $m$, we denote by $L(m)$ the lattice $L$ with the bilinear form multiplied by $m$.\ The {\sf  discriminant} of  $L$ is the finite abelian group 
\begin{equation*}
D(L) := L^\vee/L.
\end{equation*}
As usual, we denote by
\begin{itemize}
\item 
 $I_1$ the odd lattice $\Z$ with intersection form  $(1)$, 
\item 
 $I_{r,s}$ the odd lattice $I_1^{\oplus r}\oplus  I_1(-1)^{\oplus s}$, 
\item 
 $U$ the even hyperbolic lattice $\Z^2$ with intersection form  $\bigl(\begin{smallmatrix}0&1\\1&0\end{smallmatrix}\bigr)$, 
\item 
 $E_8$ the unique positive definite, even, unimodular lattice of rank 8. 
\end{itemize}

The following three lattices are important in this article
\begin{equation}\label{eq:gamma-lattices} 
\Gamma_4 := I_{22,2},
\qquad
\Gamma_6 := E_8(-1)^{\oplus 2} \oplus U^{\oplus 4},\qquad \Lambda:=  E_8^{\oplus 2}\oplus  U^{\oplus 2}\oplus I_{2,0}(2).
\end{equation}

\begin{prop}\label{ltf}
Let $X$ be a smooth GM variety of dimension $n=4$ 
or $6$.\ There are isomorphisms of lattices $H^n(X;\Z) \cong \Gamma_n$ and $H^n(X;\Z)_{00} \cong \Lambda((-1)^{n/2})$.
\end{prop}
 
 \begin{proof}
 When $n=4$, the lattices $H^4(X;\Z)$ and $H^4(X;\Z)_{00}$ are described in \cite[Proposition~5.1]{dims} (although the proof that these groups are torsion-free is missing).

When $n=6$, the class $c_1(X)= 4\gamma_X^*\bsi_1$ is divisible by 2, hence the Stiefel--Whitney class~$w_2(X)$, which is its image in $ H^2(X;\Z/2\Z)$, vanishes.\  Since $w_1(X)=0$  (as for any complex compact manifold) and we are in dimension 6, the (unimodular) lattice $H^6(X;\Z)$ is {\em even} (\cite[p.~115]{hbj}).\ Since its signature is $(4,20)$ by Proposition \ref{hn}, it is therefore isomorphic to~{$\Gamma_6$}.

The intersection form on the sublattice $\gamma_X^*(H^6(\Gr(2,V_5);\Z) )\subset H^6(X;\Z)$ has matrix $\bigl(\begin{smallmatrix}2&0\\0&2\end{smallmatrix}\bigr)$ in the Schubert basis $(\gamma_X^*\bsi_{2,1},\gamma_X^*\bsi_3 )$.\  This sublattice is  moreover primitive: if not, its saturation is unimodular, even, and positive definite of rank 2, which is absurd.\  By \cite[Proposition~1.6.1]{nik}, the discriminant group of its orthogonal $H^6(X;\Z)_{00}$ is therefore isomorphic to the discriminant group of $\gamma_X^*(H^6(\Gr(2,V_5);\Z))$, which is $(\Z/2\Z)^2$.\  The lattice $H^6(X;\Z)_{00}$ is moreover even and has signature $(2,20)$.\  As noted in \cite[Section~5.1]{dims}, there is only one lattice with these characteristics, to wit $\Lambda(-1)$.
\end{proof}
 
As  a lattice, $H^2(\tY_A,\Z)_0$  is also isomorphic to $\Lambda(-1)$ (\cite[(4.1.3)]{og6}).\  In Section~\ref{section:periods}, we will show that the polarized Hodge structures on  $H^n(X,\Z)_{00}$ and $H^2(\tY_A,\Z)_0$ are isomorphic (up to a twist).\ The isomorphism will be given by a correspondence constructed in the next section.

\section{Linear spaces on Gushel--Mukai varieties}
\label{section:linear-spaces}

\subsection{Linear spaces and their types}

Let $X$ be a smooth GM variety with its canonical embedding $X\subset \P(W)$.\ We let $F_k(X)$  be  the Hilbert scheme which parameterizes linearly embedded $\P^k$ in $X$, \ie, the closed subscheme of $\Gr(k+1,W)$ of linear subspaces $W_{k+1} \subset W$ such that  $\P(W_{k+1}) \subset X$.

The composition of the Gushel map $\gamma_X\colon X \to \Gr(2,V_5)$ with the Pl\"ucker embedding $\Gr(2,V_5) \subset \P(\bw2V_5)$ is induced by the linear projection $W\subset \C \oplus \bw2V_5 \to \bw2V_5$
from  the vertex $\nu$ of the cone~$\CGr(2,V_5)$.\ Since $\nu\notin X$, the Gushel map embeds {$\P(W_{k+1})$} linearly into~$\Gr(2,V_5)$.

We recall the description of linear subspaces contained in $\Gr(2,V_5)$.\ Any such subspace sits in a maximal linear subspace  and there are two types of those.\ First, for every {1-dimensional subspace} $U_1\subset  V_5$, there is a projective 3-space
\begin{equation*}
\P(V_5/U_1) \cong \P(U_1 \wedge V_5) \subset \Gr(2,V_5).
\end{equation*}
Second, for every 3-dimensional vector subspace $U_3 \subset V_5$, there is a projective plane
\begin{equation*}
\P(\bw2U_3) \cong \Gr(2,U_3) \subset \Gr(2,V_5).
\end{equation*}
We will say that a linear subspace $P \subset \Gr(2,V_5)$ is 
\begin{itemize} 
\item a {\sf $\sigma$-space} if it is contained in $\P(U_1\wedge V_5)$ for some $U_1 \subset V_5$;
\item a {\sf $\tau$-space} if it is contained in $\P(\bw2U_3)$ for some $U_3 \subset V_5$;
\item a {\sf mixed space} if it is both a $\sigma$- and a $\tau$-space.
\end{itemize}

In $\Gr(2,V_5)$, there are no projective 4-spaces and every projective 3-space  is a $\sigma$-space.\ For any distinct $U_1',U_1'' \subset V_5$, the intersection
$\P(U_1' \wedge V_5) \cap  \P(U_1'' \wedge V_5) 
 $
 is 
the point $ [U_1' \wedge U_1''] $.\ Hence, for every projective 3-space $P\subset\Gr(2,V_5)$,
there is a unique $U_1 \subset V_5$ such that $P = \P(U_1 \wedge V_5)$.\ This defines a map 
\begin{equation*}
\sigma \colon  F_3(X) \to \P(V_5).
\end{equation*}

If $U_1 \not\subset U_3$, we have $\P(U_1 \wedge V_5) \cap \P(\bw2U_3) = \emptyset$.\  
If instead $U_1 \subset U_3$, the intersection $\P(U_1 \wedge V_5) \cap \P(\bw2U_3) = \P(U_1 \wedge U_3)  $ is a line.\ Therefore, projective planes  in~$\Gr(2,V_5)$ are never of mixed type and we have a decomposition  {into connected components}
\begin{equation*}
F_2(X) = F_2^\sigma(X) \sqcup F_2^\tau(X),
\end{equation*}
where $F_2^\sigma(X)  $ is the subscheme of $\sigma$-planes and
$F_2^\tau(X)  $ the subscheme of $\tau$-planes.\ Again, {there is a} map 
\begin{equation*}
\sigma \colon  F_2^\sigma(X) \to \P(V_5)
\end{equation*}
taking a $\sigma$-plane $P$ to the unique $U_1\subset V_5$ such that $P \subset \P(U_1 \wedge V_5)$.\  
Analogously, for any distinct subspaces $U'_3,U''_3 \subset V_5$, the intersection 
$\Gr(2,U'_3) \cap \Gr(2,U''_3) = \Gr(2,U'_3 \cap U''_3)$
is either empty (if $\dim(U'_3 \cap U''_3) = 1$), or a point (if $\dim(U'_3 \cap U''_3) = 2$).\ Therefore, 
 for any $\tau$-plane $P\subset \Gr(2,V_5)$, there is a unique subspace $U_3 \subset V_5 $ such that $P = \P(\bw2U_3)$.\ This defines a map 
\begin{equation*}
\tau \colon  F_2^\tau(X) \to \Gr(3,V_5).
\end{equation*}

Finally, any line on $\Gr(2,V_5)$ is a mixed space  and there are  maps
\begin{equation*}
\sigma\colon F_1(X) \to \P(V_5)
\qquad\text{and}\qquad
\tau\colon F_1(X) \to \Gr(3,V_5).
\end{equation*}

The following proposition is crucial for our study of the schemes $F_k(X)$.\ It describes~$F_k(X)$ in terms of the relative Hilbert schemes $\Hilb^{\P^k}$  which parameterize linearly embedded~$\P^k$ 
in the fibers of  the first and second quadratic fibrations (defined in Section~\ref{sec24}).

\begin{prop}\label{prop:fano-relaitve-hilbert}
Let $X$ be a smooth GM variety of dimension $n\ge3$, with associated Lagrangian data $(V_6,V_5,A)$.\
The maps 
\begin{equation*}
\sigma\colon F_1(X) \to \P(V_5),
\qquad 
\sigma\colon F_2^\sigma(X) \to \P(V_5),
\qquad
\sigma\colon F_3(X) \to \P(V_5)
\end{equation*}
lift to isomorphisms with the following relative Hilbert  schemes for the first quadric fibration
\begin{align*}
F_1(X) &\cong \Hilb^{{\P^1}}(\P_X(\cU_X) / \P(V_5)),\\
F_2^\sigma(X) &\cong \Hilb^{{\P^2}}(\P_X(\cU_X) / \P(V_5)),\\
F_3(X) &\cong \Hilb^{{\P^3}}(\P_X(\cU_X) / \P(V_5)).
\end{align*}
Analogously, the maps
$\tau\colon F_1(X) \to \Gr(3,V_5)$ and $\tau\colon F_2^\tau(X) \to {\Gr(3,V_5)}$
lift to isomorphisms with the following relative Hilbert  schemes for the second quadric fibration
\begin{align*}
F_1(X) &\cong \Hilb^{{\P^1}}(\P_X(V_5/\cU_X) / \Gr(3,V_5)),\\
F_2^\tau(X) &\cong \Hilb^{{\P^2}}(\P_X(V_5/\cU_X) / \Gr(3,V_5)).
\end{align*}
\end{prop}

\begin{proof}
Let $\cL_k^\sigma(X) \subset X \times F_k^\sigma(X)$ be the universal family of $k$-dimensional $\sigma$-spaces.\  The map $\sigma \colon F_k^\sigma(X) \to \P(V_5)$ induces a map $\cL_k^\sigma(X) \to X \times \P(V_5)$ that takes a pair $(x,{P})$, where ${P} \subset X$ is a $\sigma$-space of dimension $k$ and $x \in {P}$ is a point, to $(x, \sigma({P}))$.\  By definition of a $\sigma$-space, if $\sigma({P}) = U_1\subset V_5$, the space ${P}$ parameterizes 2-dimensional subspaces $U_2 \subset V_5$ such that $U_1\subset U_2$.\  Therefore, $U_1$ is contained in the 2-space corresponding to the point $x$.\ In other words, $(x,\sigma({P})) \in \P_X(\cU_X)$.\  This means that we have a commutative diagram
\begin{equation}\label{eq:lksigma}
\vcenter{\xymatrix@C=3em{
X \ar@{=}[d] & \cL_k^\sigma(X) \ar[l]_-q \ar[r]^-p \ar@{-->}[d] & F_k^\sigma(X) \ar[d]^\sigma \\
X & \PP_X(\cU_X) \ar[l]_-\pi \ar[r]^-{\rho_1} & \P(V_5).
}}
\end{equation}
Thus, the fibers of $\cL_k^\sigma(X)$ over $F_k^\sigma(X)$ embed into the fibers of $\rho_1$ and this {embedding} is linear on the fibers.\ This means that the map $\sigma$ lifts to a map $F_k^\sigma(X) \to \Hilb^{\P^k}(\P_X(\cU_X)/\P(V_5))$.

Similarly, the projection via $\pi:\P_X(\cU_X) \to X$ of any $\P^k$ contained in the fiber of $\rho_1$ is a $\P^k$ contained in $X$.\ 
This defines a map $\Hilb^{\P^k}(\P_X(\cU_X)/\P(V_5)) \to F_k^\sigma(X)$.

It is straightforward to see that the  maps we constructed are mutually inverse  and thus give the isomorphisms of the first part of the proposition.

The second part is proved analogously.
\end{proof}

{In the next sections, we use this proposition and Lemmas~\ref{lemma:fibers-quadric1} and~\ref{lemma:fibers-quadric2} to describe the Hilbert schemes $F_k(X)$.}

\subsection{Projective 3-spaces on GM varieties}

{As we already mentioned,}
smooth GM varieties  contain no linear spaces of dimension 4 and higher.\ The situation with projective 3-spaces is also quite simple.

\begin{theo}\label{p32}
Let $X$ be a smooth GM variety of dimension $n$, with associated Lagrangian data $(V_6,V_5,A)$.\ If $n \le 5$, we have $F_3(X) = \emptyset$.\ If $n = 6$, there is an \'etale double covering $F_3(X) \to Y^3_{A,V_5}$; 
in particular, $F_3(X)$ is finite  and is empty for $X$ general.
\end{theo}

\begin{proof}
By Proposition~\ref{prop:fano-relaitve-hilbert}, $F_3(X)$ is the Hilbert scheme of projective 3-spaces in the fibers of the first quadric fibration $\rho_1\colon \P_X(\cU_X) \to \P(V_5)$.\ 
On the other hand, by Lemma~\ref{lemma:fibers-quadric1}, the fiber $Q_v = \rho_1^{-1}(v)$ over a point $v \in \P(V_5)$ is either
a quadric in $\P^{n-2}$ if $v \notin \Sigma_1(X)$, or a quadric in~$\P^{n-1}$ if $v \in \Sigma_1(X)$.\
Such a quadric contains a~$\P^3$ only in the following cases:
\begin{itemize} 
\item $n = 6$, $v \notin \Sigma_1(X)$, and $Q_v \subset \P^4$ is a quadric of corank $\ge 3$;
\item $n = 6$, $v \in \Sigma_1(X)$, and $Q_v \subset \P^5$ is a quadric of corank $\ge 2$;
\item $n = 5$, $v \notin \Sigma_1(X)$, and $Q_v \subset \P^3$ is a quadric of corank $\ge 4$;
\item $n = 5$, $v \in \Sigma_1(X)$, and $Q_v \subset \P^4$ is a quadric of corank $\ge 3$;
\item $n = 4$, $v \in \Sigma_1(X)$, and $Q_v \subset \P^3$ is a quadric of corank $\ge 4$.
\end{itemize}
The  {last three cases} do not occur by Lemma~\ref{lemma:fibers-quadric1}, because $Y^4_A = \emptyset$ (we could also invoke Corollary \ref{deg}), and neither does the second case, because  $\Sigma_1(X) = \emptyset$ for $n=6$.

In the first case, we have $v \in Y^{\ge 3}_{A,V_5}$.\  Since $Y^4_A = \emptyset$, the quadric $Q_v$ has rank 2, hence  is a union of two distinct 3-spaces, and the Hilbert scheme of 3-spaces in  $Q_v$ consists of two reduced points, hence is smooth.\ Therefore, the map $F_3(X) \to Y^3_{A,V_5}$ is an \'etale  double cover.\  The finiteness of $F_3(X)$  and its  emptyness for general $X$  follow from the same properties of~$Y^3_A$.
\end{proof}

\subsection{Planes on GM varieties}

Similar arguments provide descriptions of the schemes of planes.\ We start with $\sigma$-planes and consider the map $\sigma\colon F_2^\sigma(X) \to \P(V_5)$.

\begin{theo}\label{theorem:f2}
Let $X$ be a smooth GM variety of dimension $n$, with associated Lagrangian data $(V_6,V_5,A)$.

\noindent{\rm (a)} If $n = 6$ and $Y^3_{A,V_5} = \emptyset$, the map $\sigma$ factors as
\begin{equation*}
F_2^\sigma(X) \xrightarrow{\ \tilde\sigma\ } \tY_{A,V_5} \xrightarrow{\ f_{A,V_5}\ } Y_{A,V_5} \lhra \P(V_5), 
\end{equation*}
where $\tilde\sigma$ is a $\P^1$-bundle. 

If $n = 6$ and $Y^3_{A,V_5} \ne\emptyset$, the scheme $F_2^\sigma(X)$ has one component isomorphic to a generically $\P^1$-fibration over $\tY_{A,V_5}$
 and,  for each point of $Y^3_{A,V_5}$, one pair of irreducible components isomorphic to $\P^3$.

In particular, $F_2^\sigma(X)$ is smooth if and only if the Pl\"ucker point $\bp_X$ lies away from the projective dual $(Y^{\ge 2}_A)^\vee \subset \P(V_6^\vee)$ and $Y^3_{A,V_5} = \emptyset$.

\noindent{\rm (b)} If $n = 5$ and $X$ is ordinary, or special with $\bp_X \notin \pr_{Y,2}(E)$, 
 the map $\sigma$ factors as
\begin{equation*}
F_2^\sigma(X) \isomlra \tY^{\ge 2}_{A,V_5} \lra Y^{\ge 2}_{A,V_5} \lhra \P(V_5), 
\end{equation*}
where $\tY^{\ge 2}_{A,V_5} \to Y^{\ge 2}_{A,V_5}$ is a double covering {of the curve $Y^{\ge 2}_{A,V_5}$} branched along $Y^3_{A,V_5}$.
 
If $n = 5$ and $X$ is special with $\bp_X \in \pr_{Y,2}(E)$, the scheme $F_2^\sigma(X)$ is the union of {a double cover} $\tY^{\ge 2}_{A,V_5}$  and one  {double} or two  {reduced} components
{\rm(}depending on whether the kernel point $\Sigma_1(X)$ is in $Y^3_{A,V_5}$ or in~$Y^2_{A,V_5}${\rm)} isomorphic to $\P^1$ and contracted by the map $\sigma$ onto $\Sigma_1(X)$.

\noindent{\rm (c)} If $n = 4$, the map $\sigma$ factors as
\begin{equation*}
F_2^\sigma(X) \xrightarrow{\ \tilde\sigma\ } Y^3_{A,V_5} \lhra \P(V_5), 
\end{equation*}
where $\tilde\sigma$ is 
 an isomorphism over $Y^3_{A,V_5} \setminus \Sigma_1(X)$ and a double cover over $Y^3_{A,V_5} \cap \Sigma_1(X)$.\ 
In particular, $F_2^\sigma(X)$ is finite and is empty if and only if $Y^3_{A,V_5} = \emptyset$.

\noindent{\rm (d)} If $n \le 3$, we have $F_2^\sigma(X) = \emptyset$.
\end{theo}

\begin{rema}
The double cover $\tY^{\ge 2}_{A,V_5} \to Y^{\ge 2}_{A,V_5}$ appearing in part (b) of the theorem is described in Proposition~\ref{proposition:f3q}.\ We expect it to be the base change to $\P(V_5)$ of a natural double cover of~$Y^{\ge 2}_A$ branched along $Y^3_A$ and analogous to O'Grady's double cover $f_A:\tY_A \to Y_A$.
\end{rema}

\begin{proof}[Proof of the theorem]
In each case, by Proposition~\ref{prop:fano-relaitve-hilbert}, the scheme $F_2^\sigma(X)$ is isomorphic to the relative Hilbert scheme of planes
in the fibers of the first quadric fibration {$\rho_1 \colon \P_X(\cU_X) \to \P(V_5)$}.\ The rest follows from the description of {its fibers $Q_v := \rho_1^{-1}(v)$} in Lemma~\ref{lemma:fibers-quadric1}.

 (a)  We assume $n = 6$.\ The locus $\Sigma_1(X) $ is then empty hence, for any $v \in \P(V_5)$, the fiber $Q_v$ is a quadric in $\P^4$.\ If it is non-degenerate, it contains no planes hence the map $\sigma$ factors through $Y_{A,V_5}$. 

If the corank of $Q_v$ is 1, there are two families of planes on $Q_v$, each parameterized by~$\P^1$.\ Hence, over $Y^1_{A,V_5}$, the map $\sigma$ factors as a $\P^1$-fibration followed by a double covering.

If the corank of $Q_v$ is 2, there is only one family of planes on $Q_v$ parameterized by~$\P^1$.\  Over $Y^2_{A,V_5}$, the map $\sigma$ is therefore a $\P^1$-fibration.

Finally, if the corank of $Q_v$ is 3, we have $Q_v = \P^3 \cup \P^3$ (intersecting along a plane), 
hence planes on $Q_v$ are parameterized by $\P^3 \cup \P^3$ (dual spaces, intersecting in a point).\  It follows that $F_2^\sigma(X)$ has two irreducible components isomorphic to $\P^3$ over each point of $Y^3_{A,V_5}$
and a component that dominates $Y_{A,V_5}$.\ Considering the Stein factorization of the map $\sigma$
restricted to this (main) component of $F_2^\sigma(X)$, we see that it is the composition of a $\P^1$-bundle
(away of the preimage of $Y^3_{A,V_5}$)  and a double cover of $Y_{A,V_5}$ branched along $Y^{\ge 2}_{A,V_5}$.\ This $\P^1$-bundle is \'etale locally trivial and its Brauer class is given by the sheaf of even parts of Clifford algebras (\cite[Lemma~4.2]{K10}).

To show that this double cover is isomorphic to $\tY_{A,V_5}$  (the base change to $\P(V_5)$ of the double EPW sextic), we compute, using Proposition~\ref{proposition:f3q}, the pushforward of $\cO_{F_2^\sigma(X)}$ to~$\P(V_5)$.\  We have
$S = \P(V_5)$, $m = 5$, $\cE = \cO_S \oplus \cO_S(-1) \otimes (V_5/\cO_S(-1)) \cong \cO_S \oplus T_S(-2)$
{(where $T_S $ is the tangent bundle)}, 
and $\cL = \cO_S$.\  By Lemma~\ref{lemma:fibers-quadric1}, the degeneracy loci of the first quadratic fibration are $D_c = Y_{A,V_5}^{\ge c}$.\ Moreover, by~\cite[Proposition~4.5 and Lemma~C.6]{DK}, the cokernel sheaf $\cC$ for the family of quadrics is isomorphic to $  \Coker(\cO(-1) \otimes \bw2(V_6/\cO(-1)) \to A^\vee \ot \cO_S)$, 
hence the double cover can be written as 
 \begin{equation*}
\Spec(\cO_{Y_{A,V_5}} \oplus \cC(-3)).
\end{equation*}
But this is precisely the base change to $\P(V_5)$ of the double EPW sextic (\cite[Section~4]{og1}).
 
The statement about smoothness follows from the above description and Lemma~\ref{lemma:singularities}(b). 

(b) We assume $n = 5$.\ If $v \notin \Sigma_1(X)$, the fiber $Q_v$ is a quadric in $\P^3$.\
It contains a plane if and only if its corank is at least 2.\ Therefore, we have two planes over each point of~$Y^2_{A,V_5} \setminus \Sigma_1(X)$ 
and one {double} plane over each point of $Y^3_{A,V_5} \setminus \Sigma_1(X)$.\ On the other hand, if~$v \in \Sigma_1(X)$, the fiber $Q_v$ is a quadric in $\P^4$.\ It contains  planes if and only if it is degenerate and these planes are then parameterized by $\P^1 \sqcup \P^1$ if the corank is 1 (\ie, if $v \in Y^2_{A,V_5}$) or by a double~$\P^1$ if the corank is 2 (\ie, if $v \in Y^3_{A,V_5}$).

It follows that if $X$ is ordinary (hence $\Sigma_1(X) = \emptyset$)  or if $X$ is special and $\bp_X \notin \pr_{Y,2}(E)$ (so that, by~\eqref{eq:kernel-locus}, the kernel point $\Sigma_1(X)$ is not on $Y^{\ge 2}_{A,V_5}$), there is a double cover $F_2^\sigma(X) \to Y^{\ge 2}_{A,V_5}$ branched along $Y^3_{A,V_5}$, while if $X$ is special and $\bp_X \in \pr_{Y,2}(E)$, we have extra component(s) in~$F_2^\sigma(X)$ as described in the statement of the theorem.

  (c) We assume $n = 4$.\ If $v \notin \Sigma_1(X)$, the fiber $Q_v$ is a quadric in $\P^2$.\ It contains a plane (and is then equal to it) if and only if its corank is 3.\ Hence $F_2^\sigma(X)$ contains $Y^3_{A,V_5} \setminus \Sigma_1(X)$.\ On the other hand, if $v \in \Sigma_1(X)$, the fiber $Q_v$ is a quadric in~$\P^3$.\ It contains a plane if and only if its corank is 2 (and then it contains two planes).\ Hence $F_2^\sigma(X)$ contains two points for each point of $Y^3_{A,V_5} \cap \Sigma_1(X)$.\ We conclude using the fact that $Y^3_A$ is finite (Section~\ref{se22}).

Statement (d)  follows from Corollary \ref{deg}.
\end{proof}

Using the second quadric fibration, we describe the scheme $F_2^\tau(X)$.

\begin{theo}\label{th34}
Let $X$ be a smooth GM variety of dimension $n$, with associated Lagrangian data $(V_6,V_5,A)$.

\noindent{\rm (a)}  If $n = 6$,  
 {the map $\tau$ factors as
\begin{equation*}
F_2^\tau(X) 
\xrightarrow{\ \tilde\tau\ } Z^{\ge 2}_{A,V_5} \lhra \Gr(3,V_5), 
\end{equation*}
where $\tilde\tau$ is a double covering branched along $Z^{\ge 3}_{A,V_5}$.}

\noindent{\rm (b)}  If $n = 5$, the map $\tau\colon F_2^\tau(X) \to \Gr(3,V_5)$ factors as
\begin{equation*}
F_2^\tau(X) \xrightarrow{\ \tilde\tau\ } Z^{\ge 3}_{A,V_5} \lhra \Gr(3,V_5),
\end{equation*}
where $\tilde\tau$ is an isomorphism over $Z^{\ge 3}_{A,V_5} \setminus \Sigma_2(X)$
and a double cover over $Z^{\ge 3}_{A,V_5} \cap \Sigma_2(X)$, branched along $Z^4_{A,V_5}$.

\noindent{\rm (c)}  If $n = 4$, the map $\tau$ factors as
\begin{equation*}
F_2^\tau(X) \xrightarrow{\ \tilde\tau\ } Z^4_{A,V_5} \lhra \Gr(3,V_5), 
\end{equation*}
where $\tilde\tau$ is \'etale, is an isomorphism over $Z^4_{A,V_5} \setminus \Sigma^2_2(X)$, and  is a double cover over $Z^4_{A,V_5} \cap \Sigma^2_2(X)$.\ In particular, $F_2^\tau(X)$ is finite and it is empty if and only if $Z^4_{A,V_5} = \emptyset$.

\noindent{\rm (d)}  If $n \le 3$, we have $F_2^\tau(X) = \emptyset$.
\end{theo}

\begin{rema}
The double cover $F_2^\tau(X) \to Z^{\ge 2}_{A,V_5}$ which appears in part (a) of the theorem is described in Proposition~\ref{proposition:f3q}.\ We expect it to be the base change to $\Gr(3,V_5)$ of the double cover $\tZ^{\ge 2}_A \to Z^{\ge 2}_A$ constructed for general $A$ in~\cite{ikkr}.
\end{rema}

\begin{proof}[Proof of the theorem]
In all cases, by Proposition~\ref{prop:fano-relaitve-hilbert}, the scheme $F_2^\tau(X)$ is isomorphic to the relative Hilbert scheme of planes in the fibers of the second quadric fibration {$\rho_2 \colon \P_X(V_5/\cU_X) \to  \Gr(3,V_5)$}.\ The rest follows from the description of {its fibers $Q_{U_3}: = \rho_2^{-1}(U_3)$} in Lemma~\ref{lemma:fibers-quadric2}.

(a) We assume $n = 6$.\ We have $\Sigma_2(X) = \emptyset$ and,  {by Lemma~\ref{lemma:zav4-empty}, $Z_{A,V_5}^4 = \emptyset$}.\ For any point $U_3 \in \Gr(3,V_5)$, the fiber  {$Q_{U_3}$} is a quadric in $\P^3$.\ It contains a plane only if its corank is at least 2; therefore, the map $\tau$ factors through $Z^{\ge 2}_{A,V_5} \subset \Gr(3,V_5)$.\ The quadric $Q_{U_3}$ is the union of two planes if $U_3 \in Z^2_{A,V_5}$ and a double plane if $U_3 \in Z^3_{A,V_5}$, so the map $\tau\colon F_2^\tau(X) \to Z^{\ge 2}_{A,V_5}$ is a double cover branched along ${Z}^3_{A,V_5}$.\

(b) We assume $n = 5$.\ If $U_3 \notin \Sigma_2(X)$, the fiber $Q_{U_3}$ is a quadric in $\P^2$.\ It contains a plane  {(and is then equal to it)} only if its corank is 3, hence the map $\tau$ factors through $Z^{\ge 3}_{A,V_5}$ and is  an isomorphism over $Z^{\ge 3}_{A,V_5} \setminus \Sigma_2(X)$.\ On the other hand, if $U_3 \in \Sigma_2(X)$, the fiber $Q_{U_3}$ is a quadric in $\P^3$ and it contains a plane only when the quadric has corank at least 2,  hence when $U_3 \in Z^{\ge 3}_{A,V_5}$.\ More precisely, if $U_3 \in Z^3_{A,V_5} \cap \Sigma_2(X)$, the quadric $Q_{U_3}$ is the union of two planes and the fiber of $\tau$
is two points; if $U_3 \in Z^4_{A,V_5}$  (by Lemma~\ref{lemma:zav4-empty} and~\eqref{eq:isotropic-locus}, it is then also automatically in $\Sigma_2(X)$), the quadric $Q_{U_3}$ is a double plane and the fiber of $\tau$ is a point.\ This proves the required statement.

 (c) We assume $n = 4$.\ If $U_3 \notin \Sigma_2(X)$, the fiber $Q_{U_3}$ is a quadric in $\P^1$ and so never contains a plane.\ If $U_3 \in \Sigma^1_2(X)$, the fiber $Q_{U_3}$ is a quadric in $\P^2$, so it contains a plane  {(and is then equal to it)} if and only if its corank is 3.\ This gives one point of $F_2^\tau(X)$ over each point of $Z^4_{A,V_5} \cap \Sigma^1_2(X)$.\  If $U_3 \in \Sigma^2_2(X)$, the fiber $Q_{U_3}$ is a quadric in $\P^3$, so it contains a plane if and only if its corank is at least 2.\ This gives two points of $F_2^\tau(X)$ over each point of $Z^4_{A,V_5} \cap \Sigma^2_2(X)$.\ We conclude using the fact that $Z^4_A$ is finite (Section~\ref{se22}).

Statement  (d) follows from Corollary \ref{deg}.
\end{proof}

\subsection{Lines on GM varieties}

We now consider the scheme $F_1(X)$ of lines contained in $X$. 

\begin{theo}\label{theorem:f1}
Let $X$ be a smooth GM variety of dimension $n$, with associated Lagrangian data $(V_6,V_5,A)$.

\noindent{\rm (a)} If $n = 6$, the map $\sigma\colon F_1(X) \to \P(V_5)$ is dominant  with general fiber isomorphic to $\P^3$.

\noindent{\rm (b)} If $n = 5$, the map $\sigma\colon F_1(X) \to \P(V_5)$ factors as
\begin{equation*}
 F_1(X) \xrightarrow{\ \tilde\sigma\ } {\widetilde{\P(V_5)}} \xrightarrow{\ \ } \P(V_5), 
\end{equation*}
where ${\widetilde{\P(V_5)}} \to \P(V_5)$ is the double cover branched along the sextic hypersurface $Y_{A,V_5} \subset \P(V_5)$ and $\tilde\sigma $ is a $\P^1$-bundle over the complement of  {the preimage of} $Y^{\ge 2}_{A,V_5} \cup \Sigma_1(X)$.

\noindent{\rm (c)} If $n = 4$, the map $\sigma$ factors as
\begin{equation*}
F_1(X) \xrightarrow{\ \tilde\sigma\ } \tY_{A,V_5} \xrightarrow{\ f_{A,V_5}\ } Y_{A,V_5} \lhra \P(V_5), 
\end{equation*}
where $\tilde\sigma$ is birational.\ The non-trivial fibers of $\tilde\sigma$ are 
$\P^2$ over each point of $Y^3_{A,V_5} \setminus \Sigma_1(X)$,
$\P^2 \cup \P^2$  over each point of $Y^3_{A,V_5} \cap \Sigma_1(X)$, and
$\P^1$ over each point of $f_{A,V_5}^{-1}(\Sigma_1(X) \setminus Y^3_{A,V_5})$.

\noindent{\rm (d)} If $n = 3$, the map $\sigma\colon F_1(X) \to \P(V_5)$ factors as 
\begin{equation*}
F_1(X) \xrightarrow{\ \tilde\sigma\ } Y^{\ge 2}_{A,V_5} \lhra \P(V_5), 
\end{equation*}
where $\tilde\sigma $ is an isomorphism over $Y^{\ge 2}_{A,V_5} \setminus \Sigma_1(X)$ and a double cover over 
$Y^{\ge 2}_{A,V_5} \cap \Sigma_1(X)$, branched along $Y^3_{A,V_5} \cap \Sigma_1(X)$.
\end{theo}

In Proposition~\ref{proposition:f1-smooth}, we will make part (c) more precise by showing that for an ordinary GM fourfold~$X$, the scheme $F_1(X)$ is a small resolution of $\tY_{A,V_5}$, under some explicit generality assumptions.

\begin{proof}
In all cases, by Proposition~\ref{prop:fano-relaitve-hilbert}, the scheme $F_1(X)$ is isomorphic to the relative Hilbert scheme of lines
 in the fibers of the first quadric fibration {$\rho_1 \colon \P_X(\cU_X) \to \P(V_5)$}.\ We now check that the rest follows from the description of {its fibers $Q_v := \rho_1^{-1}(v)$} in Lemma~\ref{lemma:fibers-quadric1}.
 
  (a) We assume $n = 6$.\ We have $\Sigma_1(X) = \emptyset$ and all fibers of $\rho_1$ are quadrics in $\P^4$.\ Any such quadric contain a line, hence the map $\sigma\colon F_1(X) \to \P(V_5)$ is dominant.\ Moreover, if $v \in \P(V_5) \setminus Y_{A,V_5}$, the quadric $\rho_1^{-1}(v)$ is  smooth, hence lines on it are parameterized by $\P^3$. 

  (b) We assume $n = 5$.\ If $v \notin \Sigma_1(X)$, the fiber $Q_v $ is a quadric in $\P^3$.\ If $v \notin Y_{A,V_5}$, the quadric $Q_v $ is non-degenerate, hence the family of lines on $Q_v $ is parameterized by the union of two $\P^1$.\ If $v \in Y^1_{A,V_5}$ the quadric $Q_v $ has corank 1  and lines on $Q_v $ are parametrized by a single~$\P^1$.\ Therefore, the Stein factorization of the map $\sigma\colon F_1(X) \to \P(V_5)$ is a composition of a generically $\P^1$-bundle with a double cover of $\P(V_5)$ branched along $Y_{A,V_5}$, as  claimed.\ {The Brauer class of this $\P^1$-bundle is again given by the sheaf of even parts of Clifford algebras.}

 (c) We assume $n = 4$.\ If $v \notin \Sigma_1(X)$, the fiber~$Q_v $ is a conic in $\P^2$.\  If $v \notin Y_{A,V_5}$, the conic~$Q_v $ is non-degenerate, hence contains no lines.\ Therefore, the map $\sigma$ factors through~$Y_{A,V_5}$.\  If $v \in Y^1_{A,V_5}$, the conic $Q_v$ is the union of two lines, hence the map $\sigma$ is \'etale of degree 2 over~$Y^1_{A,V_5} \setminus \Sigma_1(X)$.\ If $v \in Y^2_{A,V_5}$, the conic $Q_v $ is  a double line, hence the above double cover is branched along $Y^2_{A,V_5} \setminus \Sigma_1(X)$.\ Finally, if $v \in Y^3_{A,V_5}$, the fiber $Q_v $ is the whole plane, hence lines on $Q_v $ are parameterized by the dual plane.

To show that this double cover is isomorphic to $\tY_{A,V_5}$, we compute, using Proposition~\ref{proposition:f2q}, the pushforward of $\cO_{F_1(X)}$ to $\P(V_5)$.\ We have $S = \P(V_5) \setminus \Sigma_1(X)$, {$m = 3$}, $\cL = \cO_S$, and the bundle $\cE$ fits into exact sequence
\begin{equation*}
0 \to \cE \to \cO_S \oplus T_S(-2) \to  {\cO_S^{\oplus 2}} \to 0.
\end{equation*}
(where $T_S $ is the tangent bundle).\ By Lemma~\ref{lemma:fibers-quadric1}, the degeneracy loci of the first quadratic fibration are~$D_c = Y_{A,V_5}^{\ge c} \setminus \Sigma_1(X)$.\  Moreover, by~\cite[Proposition~4.5 and Lemma~C.6]{DK}, the cokernel sheaf~$\cC$ for the family of quadrics is isomorphic to $ \Coker(\cO(-1) \otimes \bw2(V_6/\cO(-1)) \to A^\vee \otimes \cO_S )$,  hence again the double cover can be written as $\Spec(\cO_{Y_{A,V_5} \setminus \Sigma_1(X)} \oplus \cC(-3))$.\ This is precisely the base change to $\P(V_5) \setminus \Sigma_1(X)$ of the definition of the double EPW sextic (see~\cite[Section~4]{og1}).

If $v \in \Sigma_1(X)$, the fiber $Q_v $ is a quadric in $\P^3$.\ If $v \in Y^1_{A,X}$, the quadric $Q_v $ is non-degenerate and lines on $Q_v $ are parameterized by $\P^1 \sqcup \P^1$; over each of the two points of $f_{A,V_5}^{-1}(v)$, the fiber of $\sigma$ is $\P^1$.\ 
If $v \in Y^2_{A,X}$, the quadric has corank 1 and lines on $Q_v $ are parameterized by~$\P^1$.\  Finally, if $v \in Y^3_{A,X}$, the quadric $Q_v $ has corank 2,  $Q_v  = \P^2 \cup \P^2$, and lines on $Q_v $ are parameterized by $\P^2 \cup \P^2$.

 (d) We assume $n = 3$.\ If $v \notin \Sigma_1(X)$, the fiber $Q_v $ is a quadric in $\P^1$.\
It contains no lines unless its corank is 2 (in which case it is itself a $\P^1$).\ Thus, the map $\sigma$ factors through $Y^{\ge 2}_{A,V_5}$ and is an isomorphism over $Y^2_{A,V_5} \setminus \Sigma_1(X)$.\ If $v \in \Sigma_1(X)$, the fiber $Q_v $ is a conic in $\P^2$.\ If $v \in Y^2_{A,V_5}$, it is a union of two lines, hence the map $\tilde\sigma\colon F_1(X) \to Y^{\ge 2}_{A,V_5}$ is of degree 2 over $Y^2_{A,V_5} \cap \Sigma_1(X)$.\ Finally, if $v \in Y^3_{A,V_5}$
{(it is then automatically in $\Sigma_1(X)$; see the discussion after Lemma~\ref{lemma:fibers-quadric1})},
the fiber $Q_v $ is a double line, hence the map $\tilde\sigma$ is branched over this locus.
\end{proof}

Regarding  {items (a) and (b)} in the theorem, it is possible to describe the fibers of  $\sigma$ over~$Y_{A,V_5}$  {(resp.\ over the preimage of $Y^{\ge 2}_{A,V_5} \cup \Sigma_1(X)$)}.\  We leave this as an exercise for the interested reader.\  It is also possible to  describe the scheme $F_1(X)$  by using the map~$\tau$.\

Finally, one can use a similar approach to describe the Hilbert schemes of quadrics in GM varieties, and even the Hilbert schemes parameterizing cubic subvarieties (twisted cubic curves, cubic scrolls, and so on), {but the description becomes more and more involved}.
 
\section{Periods of GM varieties}\label{sec3}
\label{section:periods}

In this section, we  relate the periods of GM varieties of dimension 4 or 6 to those of their associated EPW sextic.\ The following theorem is the main result of this article.

\begin{theo}\label{th32}
Let $X$ be a smooth GM variety of dimension $n =4$ or $6$, with associated Lagrangian data $(V_6,V_5,A)$.\ Assume that {the  double EPW sextic} $\tY_A$ is smooth {\rm(}\ie, $Y_A^{\ge3}=\vide${\rm)}.\ There is an isomorphism of polarized Hodge structures 
\begin{equation*}
H^n(X;\Z)_{00} \cong H^2(\tY_A;\Z)_0((-1)^{n/2-1}), 
\end{equation*}
where $(-1)$  is the Tate twist.
\end{theo}

In Sections~\ref{sec42} and~\ref{sec43}, we prove Theorem~\ref{th32} for $X$ general of dimension~4 (Theorem~\ref{theorem:alpha-x4}) or~6 (Theorem~\ref{theorem:alpha-x6}).\ In Section~\ref{sec41}, we define period points and maps and use them to deduce Theorem~\ref{th32} in full generality.

\subsection{Periods of GM fourfolds}\label{sec42}

Our aim in this section is to prove Theorem \ref{th32} for a smooth GM fourfold $X$, with associated Lagrangian data $(V_6,V_5,A)$.\ We will construct an explicit isomorphism when $X$ is general.\ More precisely, we assume  
\begin{equation}\label{assumption-y}
\bp_X \notin (Y^{\ge 2}_A)^\vee \quad\text{and}\quad Y^{\ge3}_{A} = \emptyset.
\end{equation} 
Note that $Y_{A^\perp}^{\ge 2} \subset (Y_A^{\ge 2})^\vee$ by Lemma~\ref{lemma:dual-y2}, hence for a GM fourfold $X$ satisfying~\eqref{assumption-y}, we have $\bp_X \in Y_{A^\perp}^1$, i.e., $X$ is ordinary.\ We will use this observation further on.

Note also that  $(Y^{\ge 2}_A)^\vee$ does not contain $Y_{A^\perp}$, since $Y_{A^\perp}^\vee = Y_A$ is not equal to $Y^{\ge 2}_A$; therefore, the choice of~$\bp_X$ satisfying~\eqref{assumption-y} is possible.\  By Lemma~\ref{lemma:singularities}(a), assumption~\eqref{assumption-y} implies
 \begin{equation}\label{equation:y-a-v}
Y^{\ge 2}_{A,V_5}\ \text{ is a smooth curve and }\ Y^{\ge3}_{A,V_5} = \emptyset.
\end{equation} 
Since $X$ is ordinary, $\Sigma_1(X)$ is a point, which we denote by $v_0$.\ Moreover, $\bp_X \notin \pr_{Y,2}(E)$ by Lemma~\ref{lemma:dual-y2}, hence we have, by~\eqref{eq:kernel-locus},
\begin{equation}\label{assumption-kernel-vector}
v_0 \in Y^1_{A,V_5}. 
\end{equation}
We have $\Sing(Y_{A,V_5}) = \{v_0\} \cup Y_{A,V_5}^{\ge 2}$ by Lemma~\ref{lemma:singularities}{(b)}.\  Furthermore,~\eqref{assumption-kernel-vector} and Lemma~\ref{lemma:fibers-quadric1}(b)  imply that
\begin{equation}\label{defq}
Q_0 := \pi(\rho_1^{-1}(v_0))
\end{equation}  
is a smooth  quadric surface contained in $X$, with span $\Pi := \P(v_0 \wedge V_5)\subset M_X$.

The Hilbert scheme $F_1(X)$ of lines on $X$ was described in Theorem~\ref{theorem:f1}{(c)}.\  Under our generality assumption, this description takes the following simple form.

\begin{coro}\label{corollary-f1-open}
Under assumption \eqref{assumption-y}, the map $ \tsi\colon F_1(X) \to \tY_{A,V_5}$ is an isomorphism over the complement of  {the two points of} $f_{A,V_5}^{-1}(v_0)$ and the fiber of $\tsi$ over each of these two points  is~$\P^1$.
\end{coro}

We  now prove that the threefold $F_1(X)$  is smooth.

\begin{prop}\label{proposition:f1-smooth}
Let $X$ be a smooth ordinary GM fourfold, with associated Lagrangian data $(V_6,V_5,A)$, such that~\eqref{assumption-y} holds.\ The map $ \tsi\colon F_1(X) \to \tY_{A,V_5}$ is then a small resolution of singularities of $\tY_{A,V_5}$.\ In particular, $F_1(X)$ is a smooth irreducible threefold.
\end{prop}

\begin{proof}
By  Lemma~\ref{lemma:singularities}{(b)} and~\eqref{equation:y-a-v}, we have $\Sing(\tY_{A,V_5}) = f_{A,V_5}^{-1}(v_0)$.\ Since  $ \tsi$ is an isomorphism over the complement of $f_{A,V_5}^{-1}(v_0)$, it remains to show that 
$F_1(X)$ is smooth along $\tsi^{-1}(f_{A,V_5}^{-1}(v_0))$.\ In other words, we have to show that $F_1(X)$ is smooth
at points corresponding to lines $L \subset X$ such that $\sigma([L]) = v_0$.\ By deformation theory, it is enough to prove
$H^1(L,\cN_{L/X}) = 0$ for any of these lines. 

By definition of the map $\sigma$, a line $L$ with $\sigma([L]) = v_0$ lies on the 2-dimensional quadric $Q_0  = \rho_1^{-1}(v_0)$, which is smooth by Lemma~\ref{lemma:fibers-quadric1}.\ Therefore, there is an exact sequence 
\begin{equation*}
0 \to \cN_{L/Q_0} \to \cN_{L/X} \to \cN_{Q_0/X}\vert_L \to 0.
\end{equation*}
The first term is $\cO_L$ since $Q_0$ is a smooth quadric.\ It is enough to show that the last term is either $\cO_L \oplus \cO_L$ or $\cO_L(1) \oplus \cO_L(-1)$.\  Since $Q_0$ is the transversal intersection of $\Pi$ and a quadric cutting out $X$ in $M_X$, we have
\begin{equation*}
\cN_{Q_0/X} \cong \cN_{\Pi/M_X}\vert_{Q_0}.
\end{equation*}
On the other hand, since $M_X$ is a hyperplane section of $\Gr(2,V_5)$, we have an exact sequence
\begin{equation*}
0 \to \cN_{\Pi/M_X} \to \cN_{\Pi/\Gr(2,V_5)} \to \cO_\Pi(1) \to 0.
\end{equation*}
 {Finally, one easily proves}
the isomorphism $\cN_{\Pi/\Gr(2,V_5)} \cong T_\Pi(-1)$.\ Combining all these, we obtain an exact sequence
\begin{equation*}
0 \to \cN_{Q_0/X}\vert_L \to T_\Pi(-1)\vert_L \to \cO_L(1) \to 0.
\end{equation*}
Since $\Pi \isom \P^3$, the middle term   is $\cO_L(1) \oplus \cO_L \oplus \cO_L$.\ It follows that $\cN_{Q_0/X}\vert_L$
is either $\cO_L \oplus \cO_L$  or $\cO_L(1) \oplus \cO_L(-1)$.\ 
This completes the proof of the proposition.
 \end{proof}

Under our assumptions, the set $f_{A,V_5}^{-1}(v_0)$ consists of two points, which we denote by $\bp'$ and $\bp''$.\  We also denote by
\begin{equation*}
\P' := \tsi^{-1}(\bp') \subset F_1(X)
\qquad\text{and}\qquad
\P'':= \tsi^{-1}(\bp'') \subset F_1(X)
\end{equation*}
the non-trivial fibers of the map $\tsi\colon F_1(X) \to \tY_{A,V_5}$  {(each of them is isomorphic to $\P^1$).}

\begin{rema}\label{remark:involution-f1}
The involution $\tau_A$ of the double cover $f_A\colon \tY_A \to Y_A$ restricts to the involution of the double cover $f_{A,V_5}\colon \tY_{A,V_5} \to Y_{A,V_5}$.\ However, it does not extend to a regular involution of~$F_1(X)$: the small resolutions of the  two singular points $\bp'$ and $\bp''$ of $\tY_{A,V_5}$ are not compatible with this involution.
\end{rema}

 {Denote by $\iota \colon \tY_{A,V_5} \to \tY_A$ the canonical embedding.}

\begin{prop}\label{proposition:hodge-f1-tyav}
Let $X$ be a smooth ordinary GM fourfold, with associated Lagrangian data $(V_6,V_5,A)$, such that~\eqref{assumption-y} holds.\ The restriction $\iota^*\colon H^2(\tY_A;\Z) \to H^2(\tY_{A,V_5};\Z)$ is an isomorphism and the composition  
\begin{equation*}
H^2(\tY_A;\Z) \xrightarrow{\ \iota^*\ } H^2(\tY_{A,V_5};\Z) \xrightarrow{\ \tsi^*\ } H^2(F_1(X);\Z) 
\end{equation*}
induces an isomorphism of Hodge structures  between $H^2(\tY_A;\Z)$ and $\langle \P', \P'' \rangle^\perp \subset H^2(F_1(X);\Z)$.
\end{prop}

\begin{proof}
Since \eqref{assumption-y} holds,  $\tY_A$ is smooth, hence  $\iota^*$ is an isomorphism  by the Lefschetz Theorem (Lemma \ref{lef}).\ Set  $U := \tY_{A,V_5} \setminus \{ \bp', \bp'' \} \isom F_1(X) \setminus ( \P' \cup \P'' )$.\  We have a commutative diagram  \begin{equation*}
\xymatrix@C=3mm{
\cdots \ar[r] & 
H^1( \{ \bp', \bp'' \};\Z) \ar[r] \ar^-{\tsi^*}[d] &
H^{2}_c(U;\Z) \ar[r] \ar@{=}[d] & 
H^{2}(\tY_{A,V_5};\Z) \ar[r] \ar^-{\tsi^*}[d] &
H^{2}( \{ \bp', \bp'' \};\Z) \ar^-{\tsi^*}[d] \ar[r] &
H^{3}_c(U;\Z) \ar[r] \ar@{=}[d] & 
\cdots \\
\cdots \ar[r] & 
H^1(\P' \cup \P'';\Z) \ar[r] &
H^{2}_c(U;\Z) \ar[r] & 
H^{2}(F_1(X);\Z) \ar[r] &
H^{2}(\P' \cup \P'';\Z) \ar[r] &
H^{3}_c(U;\Z) \ar[r] & 
\cdots 
}
\end{equation*}
of exact sequences in cohomology with compact supports.\ The first column is zero and the fourth column is the inclusion $0 \to \Z \oplus \Z$.\ The third column therefore extends to an exact sequence
\begin{equation}\label{exse}
\xymatrix{
0 \ar[r] & 
H^2(\tY_{A,V_5};\Z) \ar[r]^-{\tsi^*} &  
H^2(F_1(X);\Z) \ar[r]^-r & 
\Z\oplus \Z ,
}
\end{equation}
where  $r(x):= (x\cdot [\P'],x\cdot [\P''])$.\ 
This proves the proposition.
\end{proof}

Let $p\colon \cL_1(X) \to F_1(X)$ be the universal line and let $q\colon \cL_1(X) \to X$ be the natural morphism.\ These two morphisms  define a correspondence between $X$ and $F_1(X)$  hence a map between $H^\bullet(X;\Z)$ and $H^\bullet(F_1(X);\Z)$ which we investigate.\ We extend diagram~\eqref{eq:lksigma}  to a commutative diagram
\begin{equation}\label{diagram:big4}
\vcenter{\xymatrix@M=3pt@C=12pt{
&&& \cL_1(X) \ar[dlll]_-q \ar[drrr]^p \ar@{-->}[dl]^(.4){q'} \ar@{..>}[dr]_(.4){q''}
\\
X && 
\P_X(\cU_X) \ar[ll]_{\pi} \ar[d]_{\rho_1} && 
\,\P_X(\cU_X) \times_{\P(V_5)} Y_{A,V_5} \ar[d] \ar@{_{(}->}[ll] &&
F_1(X) \ar[d]^\tsi
\\
&&
\P(V_5) &&
\,Y_{A,V_5} \ar@{_{(}->}[ll] &&
\tY_{A,V_5} \ar[ll]_{f_{A,V_5}} \ar@{^{(}->}[rr]^\iota && \tY_A,
}}
\end{equation}
 {where the map $q'$ is defined in the same way as the dashed arrow in~\eqref{eq:lksigma}  and the map $q''$ is constructed by the universal property of the fiber product.}

 {Recall that $h$ stands for the hyperplane class of $X$ and $\tilde{h}$ for the hyperplane class of~$\P(V_6)$ and  its restrictions to $\P(V_5)$ and $\tY_A$.}

\begin{lemm}\label{lem36}
The map $q''$ is  finite and birational, and
$ q'_*([\cL_1(X)]) = 6\rho_1^*\tilde{h} $ in $ H^2(\P_X(\cU_X);\Z)$.\
 In particular, the map $q$ is generically finite of degree $6$.
\end{lemm}

\begin{proof}
Let $(x,v) \in \P_X(\cU_X) \times_{\P(V_5)} Y_{A,V_5} \subset \P_X(\cU_X) \subset X \times \P(V_5)$.\ If $\gamma_X(x) = [U_2]$, this  means that
$v $ is in $ \P(U_2) \cap Y_{A,V_5}$ and that $(q'')^{-1}(x,v)$ is  the Hilbert scheme of lines in $\rho_1^{-1}(v)$ passing through $x$.\ But $\rho_1^{-1}(v)$ is either a reducible conic, or a double line {if $v \ne v_0$}, or the quadric $Q_0$ defined in \eqref{defq} if $v=v_0$.\ Therefore, there is a unique line through $x$, unless $x$ is a singular point of a singular conic {or $v = v_0$}, in which case there are two lines through $x$.\ Thus $q''$ is finite and birational.

It follows that  {the pushforward} $q'_*([\cL_1(X)]) \in H^2(\P_X(\cU_X);\Z)$ is the class of  the divisor $\P_X(\cU_X) \times_{\P(V_5)} Y_{A,V_5}$, hence the pullback via $\rho_1$ of the class of $Y_{A,V_5}$, which is equal to $6\tilde{h}$.
\end{proof}
 
Geometrically, this means that for a general point $x$ of  an ordinary GM fourfold $X$, there are~6 lines passing through $x$ and contained in $X$.
 
\begin{coro}\label{corollary:qpsi-h2}
One has $q_*p^*\tsi^*\iota^*\tilde{h}^2 = 6c_2(V_5/\cU_X)=6\gamma_X^*\bsi_2$.
\end{coro}

\begin{proof}
 We need to compute the pullback of $\tilde{h}^2$ to $\cL_1(X)$ and its pushforward via $q$ to $X$.\ We can rewrite this as
$\pi_*q'_*{q'}^*\rho_1^*\tilde{h}^2$.\ By the projection formula and Lemma~\ref{lem36}, this is equal to $\pi_*(\rho_1^*\tilde{h}^2\cdot 6\rho_1^*\tilde{h}) = 6 \pi_*\rho_1^*\tilde{h}^3$.\ Since $\pi$ is the projectivization of $\cU_X$ and $\rho_1^*\tilde{h}$ is a relative hyperplane class, we have
\begin{equation}\label{eq:projective-bundle}
\rho_1^*\tilde{h}^2 + c_1(\cU_X)\rho_1^*\tilde{h} + c_2(\cU_X) = 0. 
\end{equation}
Therefore, $\rho_1^*\tilde{h}^3 = (c_1(\cU_X)^2 - c_2(\cU_X))\rho_1^*\tilde{h} + c_1(\cU_X)c_2(\cU_X)$ and 
\begin{equation*}
\pi_*\rho_1^*\tilde{h}^3 = {c_1(\cU_X)}^2 - c_2(\cU_X) = c_2(V_5/\cU_X).
\end{equation*}
The corollary follows.
\end{proof}

\begin{lemm}\label{lemm:qp-p1p2}
We have $q_*p^*(\P') = q_*p^*(\P'') = [Q_0] \in H^4(X;\Z)$. 
\end{lemm}

\begin{proof}
The subscheme $\P' \sqcup \P'' \subset F_1(X)$ parameterizes lines on $X$ that are contained in the smooth quadric surface $Q_0$.\ The lines in each of the components $\P'$ and $\P''$ sweep out $Q_0$ once, hence the claim.
\end{proof}

Let $X$ be a smooth GM fourfold.\ Consider the morphism
\begin{equation}\label{eq:alpha-4}
\alpha\colon H^4(X;\Z) \lra H^2(F_1(X);\Z),
\qquad
x \longmapsto p_*(q^*x).
\end{equation} 
The classes  {(see Lemma~\ref{lemma:class-q0})} 
\begin{equation}\label{eq:relation-h4}
c_2(\cU_X) = \gamma_X^*\bsi_{1,1}\qquad{\rm and}\qquad [Q_0]  =  \gamma_X^*(\bsi_2-\bsi_{1,1})=  h^2-2 c_2(\cU_X)
\end{equation}
in $ H^4(X;\Z)$ generate the subgroup $\gamma_X^*(H^4(\Gr(2,V_5);\Z))$.\ In the following two lemmas, we compute their images by the map $\alpha$.\ We assume as before that $X$ satisfies the assumptions~\eqref{assumption-y}.

\begin{lemm}\label{lemma:alpha-c2-u}
We have $\alpha(c_2(\cU_X)) = \tsi^*\iota^*\tilde{h}$.
\end{lemm}

\begin{proof}
Consider the diagram~\eqref{diagram:big4}.\ The pullback of  {the bundle} $\cU_X$ to $\P_X(\cU_X)$ is an extension 
 {$0 \to \cO(-\rho_1^*\tilde{h}) \to \pi^*\cU_X \to \cO(\rho_1^*\tilde{h} - \pi^*h) \to 0$} of line bundles, hence
\begin{equation*}
\pi^*c_2(\cU_X) = \rho_1^*\tilde{h}(\pi^*h - \rho_1^*\tilde{h}).
\end{equation*}
Therefore, we have
\begin{equation*}
q^*c_2(\cU_X) = 
{q'}^*\pi^*c_2(\cU_X) = 
{q'}^*(\rho_1^*\tilde{h}(\pi^*h - \rho_1^*\tilde{h})) =
p^*\tsi^*\iota^*\tilde{h}\cdot q^* h - p^*\tsi^*\iota^*\tilde{h}^2
\end{equation*}
and, since $p$ is a $\P^1$-bundle with relative  hyperplane class $q^*h$,  the pushforward by $p$ of the  right  side equals $\tsi^*\iota^*\tilde{h}$.
\end{proof}

Recall that the surface $Q_0 = X \cap \P(v_0 \wedge (V_5/v_0))$ defined in \eqref{defq}  is a smooth quadric.\ To compute the class $\alpha([Q_0]) $ in $ H^2(F_1(X);\Z)$, we   need some preparation.\ 

First, the Hilbert scheme of lines on $Q_0$ is $F_1(Q_0) = \P' \sqcup \P''$ and the corresponding universal line is $\cL_1(Q_0) = Q'_0 \sqcup Q''_0$, where the first (resp.\ second) component corresponds to lines parameterized by $\P'$ (resp.\ $\P''$) and the map $\cL_1(Q_0) \subset \cL_1(X) \xrightarrow{\ q\ } X$ induces isomorphisms $Q'_0 \cong Q_0$ and $Q''_0 \cong Q_0$.

Second, we have $\cU_{Q_0}:=\cU_X\vert_{Q_0} \cong \cO_{Q_0} \oplus \cO_{Q_0}(-1)$  {by Lemma~\ref{lemma:ux-q0-splits}}, hence 
$\rho_1(\P_{Q_0}(\cU_{Q_0}))   \subset \P(V_5)$ is the quadratic cone $ \cone{v_0} Q_0 $ over $Q_0 \subset \P(V_5/v_0)$ with vertex  $v_0$.\ Set 
\begin{equation*}
S := Y_{A,V_5} \cap \cone{v_0}Q_0 \subset \P(V_5).
\end{equation*}
This is an intersection of two distinct irreducible hypersurfaces in $\P(V_5)$ (Lemma \ref{lemma:singularities}(b)) hence a Cohen--Macaulay surface containing $v_0$.

\begin{lemm}\label{lemma:alpha-q}
There is a surface $R \subset F_1(X)$ such that 
\begin{equation*} 
\alpha([Q_0]) = [R] \in H^2(F_1(X);\Z) 
\end{equation*} 
and the map $\sigma = f_A\circ \tilde \sigma \colon R \to  \P(V_5)$ is birational onto~$S$.
\end{lemm}

\newcommand{\tS}{\tilde{S}}

\begin{proof}
We first describe $q^{-1}(Q_0)$.\ Since $\cL_1(X)$ is smooth of dimension 4 (Proposition~\ref{proposition:f1-smooth})  {and $q$ is dominant (Lemma~\ref{lem36})}, $q^{-1}(Q_0)$ has everywhere dimension $\ge2$.\ We see on the diagram~\eqref{diagram:big4} that the map $q \colon \cL_1(X) \to X$ factors through the map $q'' \colon \cL_1(X) \to \P_X(\cU_X) \times_{\P(V_5)} Y_{A,V_5}$; moreover, we have
\begin{equation*}
q^{-1}(Q_0) = q^{\prime\prime -1} (\P_{Q_0}(\cU_{Q_0}) \times_{\P(V_5)} Y_{A,V_5}).
\end{equation*}
The situation is summarized in the following cartesian diagram
\begin{equation*}
\xymatrix
@M=5pt{
Q_0 \ar@{_{(}->}[d] & \P_{Q_0}(\cU_{Q_0}) \ar[l]_-{\pi} \ar@{_{(}->}[d] & \P_{Q_0}(\cU_{Q_0}) \times_{\P(V_5)} Y_{A,V_5} \ar@{_{(}->}[l] \ar@{_{(}->}[d] & q^{-1}(Q_0) \ar[l]_(0.3){q''} \ar@{_{(}->}[d]
\\
X & \P_X(\cU_X) \ar[l]_-{\pi} & \P_X(\cU_X) \times_{\P(V_5)} Y_{A,V_5} \ar@{_{(}->}[l] & \cL_1(X). \ar[l]_(0.3){q''}
}
\end{equation*}
The projection $\P_{Q_0}(\cU_{Q_0}) \times_{\P(V_5)} Y_{A,V_5} \to \P(V_5)$ factors through $S$ and is an isomorphism over the dense open subscheme $S_0 := S \setminus \{v_0\}$.\ Let $\tS_0 \subset \P_{Q_0}(\cU_{Q_0}) \times_{\P(V_5)} Y_{A,V_5}$ be the preimage of~$S_0$ and set
\begin{equation*}
R_0 := q^{\prime\prime -1}( \tS_0) \subset q^{-1}(Q_0)
\subset \cL_1(X).
\end{equation*}
Let $\tS_{00} \subset \tS_0$ be the open subscheme of $\tS_0$ parameterizing pairs $(x,v)$ such that $v \ne v_0$ and~$x$ is not a singular point of the conic $\rho_1^{-1}(v)$.\ Let $S_{00}$ be its isomorphic image in $S_0 \subset S$ and let~$R_{00}$ be its preimage in $R_0$.\ By Lemma~\ref{lem36}, the map $q''\vert_{R_0}\colon R_0 \to \tS_0$ induces an isomorphism of schemes $R_{00} \isomto \tS_{00}$.\ Thus the map $\sigma\circ p\colon R_0 \to \P(V_5)$ induces an isomorphism $R_{00} \isomto S_{00}$; since we
show below that  $S_{00}$ is dense in $S$, it follows that $ R_{00}$ 
has pure dimension 2.

Consider now $R_1 := R_0 \setminus R_{00}$.\ By definition, its image in $S$ is contained in $S_1 := S_0 \setminus S_{00}$.\ We show below  $\dim(S_1) \le 1$, hence $S_{00} $ is dense in $S$ and, since the map $R_1 \to S_1$ is finite (since the Hilbert scheme of lines in any conic is finite), we also have $\dim(R_1) \le 1$.

Finally, consider $q^{-1}(Q_0) \setminus R_0  = q^{-1}(Q_0) \cap p^{-1}(\sigma^{-1}(v_0))$.
Since $\sigma^{-1}(v_0) = F_1(Q_0)$, we have $q^{-1}(Q_0) \setminus R_0 =  \cL_1(Q_0) \times_X Q_0 = Q'_0 \sqcup Q''_0$.\ All this shows that $q^{-1}(Q_0) $ has pure dimension~2 and that we can write
\begin{equation*}
q^*([Q_0]) = a'[Q_0'] + a''[Q_0''] + [\overline{R_{00}}]
\end{equation*}
for some integers $a'$ and $a''$.\ The map $p$ contracts the surfaces $Q_0'$ and $Q_0''$ onto the curves $\P'$ and $\P''$ respectively, hence $p_*$ kills their classes and we obtain
\begin{equation*}
\alpha([Q_0]) = p_*(q^*([Q_0])) = a'p_*([Q_0']) + a''p_*([Q_0'']) + p_*([\overline{R_{00}}]) = p_*([\overline{R_{00}}]).
\end{equation*}
Recall that the map $\sigma\circ p $ induces an isomorphism $R_{00} \isomto S_{00}$. 
Setting 
\begin{equation*}
R := \overline{p(R_{00})} \subset F_1(X),
\end{equation*}
we obtain $\alpha([Q_0]) = [R]$ and  the map $\sigma \colon R \to \P(V_5)$ is birational onto $\overline{S_{00}} = S$.

It remains to prove  $\dim (S_1) \le 1$.\ Let $\cE \subset W \otimes \cO_{S_1}$ be the rank-3 vector bundle over $S_1$ with fiber at a point $v $ of $ S_1$   the linear span of the conic $\rho_1^{-1}(v)$ (see the proof of Theorem~\ref{theorem:f1}(c) for its description).\ Consider the line subbundle $\cL \hookrightarrow \cE$ whose fiber over a point $v \in S_1$ which is the image of a point $(x,v) \in \tS_0$ is given by the point $x$.\ By definition of $S_1$, this is a singular point of the conic, hence the quadratic forms on $\cE$ induces a quadratic form on the quotient bundle $\cE/\cL$.\ Its discriminant locus is a divisor in $S_1$ which, by construction, is contained in the corank-2 locus of $\rho_1$.

Assume by contradiction   $\dim(S_1)\ge 2$.\ The dimension of this locus is then~$\ge 1$.\ Since, away of~$v_0$, it coincides with the smooth irreducible curve~$Y^2_{A,V_5}$, this curve is contained in $S$, hence in the quadric $ \cone{v_0}Q_0$.\ This contradicts  Corollary~\ref{corollary:y2av}.
\end{proof}

\begin{coro}\label{corollary:alpha-h2}
One has $\iota_*\tsi_*\alpha(h^2) = 3\tilde{h}^2$.
\end{coro}

\begin{proof}
By~\eqref{eq:relation-h4}, we have $\iota_*\tsi_*\alpha(h^2) = \iota_*\tsi_*\alpha([Q_0]) + 2\iota_*\tsi_*\alpha(c_2(\cU_X)) = \iota_*\tsi_*([R]) + 2\iota_*\tsi_*\tsi^*\iota^*\tilde{h}$
 (we use Lemma~\ref{lemma:alpha-c2-u} and Lemma~\ref{lemma:alpha-q} in the last equality).\ By the projection formula, the second summand equals $2\tilde{h}^2$, so it remains to show that the first summand $\iota_*\tsi_*([R])$ equals~$\tilde{h}^2$.\ Since $F_1(X) \times_{\P(V_5)} S$ is birational to the double cover $\tY_{A,V_5} \times_{\P(V_5)} S$ of $S$ and contains a surface~$R$ that maps to $S$ birationally, we have 
\begin{equation*}
F_1(X) \times_{\P(V_5)} S = R \cup \tau_A(R),
\end{equation*}
where $\tau_A$ is the birational involution of $F_1(X)$ induced 
by the involution of $\tY_A$.\ Since $S$ is cut out in $Y_{A,V_5}$ by a quadric, we have 
\begin{equation*}
\iota_*\tsi_*([R]) + \iota_*\tsi_*([\tau_A(R)]) = 2\iota_*\iota^*\tilde{h} = 2\tilde{h}^2.
\end{equation*} 
The two summands on the left side are interchanged by the involution $\tau_A$ of the double cover $f_A\colon \tY_A \to Y_A$.\ We would like to show that they are equal.\ For that, let us first assume that~$X$, hence also $A$, is very general.\ The vector space $H^{2,2}(\tY_A)\cap H^4(\tY_A;\Q)$ then has rank 2, generated by  $\tilde{h}^2$ and $c_2({\tY_A})$ (\cite[Proposition~3.2]{og7}) and both of these classes are $\tau_A$-invariant.\ Since $\iota_*\tsi_*([R]) $ and $ \iota_*\tsi_*([\tau_A(R)])$ both belong to this space,  they are also $\tau_A$-invariant, hence equal.\ Finally, since $H^4(\tY_A;\Z)$ is torsion-free (\cite[Theorem~1]{mark}),  they are both   equal to $\tilde{h}^2$.

Going back to the general case where $X$ only satisfies \eqref{assumption-y},  we note that the class $\iota_*\tsi_*([R]) - \tilde{h}^2 \in H^4(\tY_A;\Z)$  depends continuously on $X$ and  is zero for $X$ very general, as we showed above; therefore, it is zero for all $X$.
\end{proof}

We are now ready to prove Theorem \ref{th32} for  {general} GM fourfolds.\ Recall that in \eqref{diagram:big4}, the map $p\colon \cL_1(X)\to F_1(X)$  is a $\P^1$-fibration  for which $q^*h$ is a relative hyperplane class.\  Therefore, $\cL_1(X)$ is isomorphic to the projectivization of a rank-2 {vector bundle}.\ We denote by $c_1$ and $c_2$  {its} Chern classes, so that
\begin{equation}\label{eq5}
q^*h^2 + p^*c_1 \cdot q^*h + p^*c_2 = 0.
\end{equation} 
The primitive and vanishing cohomology subgroups $ H^4(X;\Z)_{00} \subset H^4(X;\Z)_0\subset  H^4(X;\Z)$  and  $H^2(\tY_A;\Z)_0\subset H^2(\tY_A;\Z)$ were defined in~\eqref{defh00} and~\eqref{defprime}, and the map $\alpha$ was defined in~\eqref{eq:alpha-4}.

\begin{theo}\label{theorem:alpha-x4}
Let $X$ be a smooth ordinary GM fourfold, with associated Lagrangian data $(V_6,V_5,A)$, such that assumption~\eqref{assumption-y} holds.\ We have
\begin{equation}\label{sta1}
\forall x \in H^4(X;\Z)_0\qquad \alpha(x)^2\cdot \alpha(h^2) = -6x^2.
\end{equation}
In particular, the restriction $\alpha_0\colon H^4(X;\Z)_0 \to H^2(F_1(X);\Z)$ of $\alpha$ is injective.\ Furthermore, it induces an isomorphism 
\begin{equation*}
\beta\colon H^4(X;\Z)_{00} \isomlra H^2(\tY_A;\Z)_0
\end{equation*}
{compatible with the Beauville--Bogomolov form}
\begin{equation*}
\forall x \in H^4(X;\Z)_{00}\qquad q_B(\beta(x)) = -x^2,
\end{equation*}
hence an isomorphism $H^4(X;\Z)_{00}(-1) \isomto H^2(\tY_A;\Z)_0$ of polarized Hodge structures.
\end{theo}

\begin{proof}
We follow the argument from \cite{bedo}.\  Since $p$ is a $\P^1$-bundle and $q^*h$  a relative hyperplane class, we may write
\begin{equation*}
\forall x \in H^4(X;\Z)\qquad q^*x = p^*x_1\cdot q^*h+ p^*x_2,
\end{equation*}
where $x_i \in H^{2i}(F_1(X);\Z)$ for $i\in\{1,2\}$.\  We have then
\begin{equation*}
\alpha(x) = x_1.
\end{equation*}
To see what  primitivity of $x$ means, we compute,   using~\eqref{eq5},
\begin{align*}
q^*( x \cdot h ) & =   p^*x_1\cdot q^*h^2+
  p^*x_2\cdot q^*h =
  p^*(x_2 - x_1\cdot c_1) \cdot q^*h - p^*(x_1\cdot c_2).
\end{align*}
Thus $x \cdot h  = 0$ implies  
\begin{equation*}
x_1\cdot c_2 = 0,
\qquad
x_2 = x_1 \cdot c_1,
\end{equation*} 
and   we can rewrite
\begin{equation*}
q^*x = p^*(\alpha(x)) \cdot ( q^*h+p^*c_1 ).
\end{equation*} 
Taking squares and using \eqref{eq5}, we obtain
\begin{equation*}
q^*x^2 = 
p^*(\alpha(x)^2) \cdot (q^*h^2 + 2  p^*c_1 \cdot q^*h +p^*c_1^2) =
p^*(\alpha(x)^2) \cdot (p^*c_1 \cdot  q^*h + p^*(c_1^2 - c_2)) = 
\alpha(x)^2 \cdot c_1.
\end{equation*}
On the other hand, since,  {by Lemma~\ref{lem36},} the degree of $q$ is 6, we have $q^*x^2 = 6x^2$.\ We obtain   $\alpha(h^2)=-c_1$   from \eqref{eq5} and all this proves \eqref{sta1}.

This relation implies  $\alpha_0(x_1)\alpha_0(x_2)\alpha(h^2) = - 6x_1x_2$ for all $x_1,x_2 \in H^4(X;\Z)_0$ and the injectivity of $\alpha_0$   follows from the non-degeneracy of the intersection pairing on $H^4(X;\Z)_0$.

If $x \in H^4(X;\Z)_{00}$, we have, by adjunction and Lemma~\ref{lemm:qp-p1p2},
\begin{equation*}
\alpha(x)\cdot [\P'] = x\cdot q_*p^*[\P']  =x\cdot [Q_0] = 0  
\end{equation*}
because, by~\eqref{eq:relation-h4}, the class $[Q_0]$ is in $\gamma_X^*H^4(\Gr(2,V_5);\Z)$.\  Similarly, we have $\alpha(x)\cdot [\P'']=0$.\   From~\eqref{exse}, we obtain that $\alpha$ maps $H^4(X;\Z)_{00}$ into the subgroup $\tsi^*\iota^* H^2(\tY_A;\Z) $ of $H^2(F_1(X);\Z)$.\ By Proposition~\ref{proposition:hodge-f1-tyav}, this defines an injective map $\beta\colon H^4(X;\Z)_{00} \to H^2(\tY_A;\Z)$ such that
\begin{equation*}
\forall x\in  H^4(X;\Z)_{00}\qquad\alpha(x) = \tsi^*(\iota^*(\beta(x))). 
\end{equation*}
It remains to show that the image of $\beta$ is the primitive cohomology $H^2(\tY_A;\Z)_0$  and that $\beta$ is compatible with the Beauville--Bogomolov form.\ Keeping the assumption  $x \in H^4(X;\Z)_{00}$  {and using Corollary~\ref{corollary:qpsi-h2}}, we have
\begin{multline*}
0 = 
x \cdot 6c_2(V_5/\cU_X) =
x \cdot q_*p^*\tsi^*\iota^*(\tilde{h}^2) =
\alpha(x)\cdot \tsi^*\iota^*(\tilde{h}^2) = \\ =
\tsi^*\iota^*(\beta(x))\cdot  \tsi^*\iota^*(\tilde{h}^2) = 
\iota^*(\beta(x))\cdot \iota^*(\tilde{h}^2) =
\beta(x) \cdot \iota_*\iota^*(\tilde{h}^2) = 
\beta(x) \cdot \tilde{h}^3.
\end{multline*}
This proves   $\beta(x) \in H^2(\tY_A;\Z)_0$ by the definition \eqref{defprime} of the primitive cohomology group.

For the compatibility with the Beauville--Bogomolov form $q_B$,  we observe
\begin{equation*}
-x^2 = 
\frac16\alpha(x)^2\alpha(h^2) =
\frac16\tsi^*\iota^*\beta(x)^2\alpha(h^2) = 
\frac16\beta(x)^2 \cdot \iota_*\tsi_*\alpha(h^2) = 
\frac16\beta(x)^2 \cdot 3\tilde{h}^2 =
q_B(\beta(x)).
\end{equation*}
The first equality is~\eqref{sta1}, the second is the definition of $\beta$, the third follows from   adjunction, the fourth is   Corollary~\ref{corollary:alpha-h2}, and the last is~\eqref{eq:bb-form}.

Finally,  the lattices $H^4(X;\Z)_{00}$ and $H^2(\tY_A;\Z)_0$ both have rank 22 and same discriminant group $(\Z/2\Z)^2$ (\cite[Proposition~5.1]{dims}  or Proposition \ref{ltf} for $H^4(X;\Z)_{00} $, and  \cite[(1.0.9)]{og2} for~$ H^2(\tY_A;\Z)_0$), hence the injective anti-isometry $\beta$ is a bijection.
\end{proof}

\subsection{Periods of GM sixfolds}\label{sec43}
Our aim in this section is to prove Theorem \ref{th32} for a smooth GM sixfold $X$, with associated Lagrangian data $(V_6,V_5,A)$.\ Again, we will provide an explicit isomorphism for general $X$, namely for those satisfying the same assumption~\eqref{assumption-y}---which implies~\eqref{equation:y-a-v}.\  Since $X$ is a sixfold, $\Sigma_1(X) $ is empty and, in contrast with the fourfold case, Lemma~\ref{lemma:singularities}(b) says that $Y_{A,V_5}$ is smooth away from $Y_{A,V_5}^{\ge 2}$, and $\tY_{A,V_5}$ is a smooth threefold. 

The   scheme of $\sigma$-planes on $X$ was described in Theorem~\ref{theorem:f2}(a).\ Under our generality assumption, this description takes the following simple form.

\begin{coro}\label{corollary-f2}Let $X$ be a smooth GM sixfold, with associated Lagrangian data $(V_6,V_5,A)$, such that assumption~\eqref{assumption-y} holds.\ Then,  $F_2^\sigma(X)$ is {a smooth fourfold} and the map $\tsi\colon F_2^\sigma(X) \to \tY_{A,V_5}$ is a $\P^1$-fibration.
\end{coro}

We denote by $\P_y  \cong \P^1 $ the fiber of  $\tsi $ over a point $y \in \tY_{A,V_5}$.\ These fibers all have the same cohomology class which we  denote by $[\P] \in H^6(F_2^\sigma(X);\Z)$.\ As before, we let $\iota \colon \tY_{A,V_5} \to \tY_A$ be the canonical embedding.

\begin{prop}\label{proposition:hodge-f2-tyav}
Let $X$ be a smooth GM sixfold, with associated Lagrangian data $(V_6,V_5,A)$, such that~\eqref{assumption-y} holds.\  
The composition  
\begin{equation*}
H^2(\tY_A;\Z) \xrightarrow{\ \iota^*\ } H^2(\tY_{A,V_5};\Z) \xrightarrow{\ \tsi^*\ } H^2(F_2^\sigma(X);\Z) 
\end{equation*}
induces an isomorphism of Hodge structures between  $H^2(\tY_A;\Z)$ and  $[\P]^\perp \subset H^2(F_2^\sigma(X);\Z)$.
\end{prop}

\begin{proof}
The map $\iota^*$ is an  isomorphism by Lemma \ref{lef}.\ We then use the fact that $\tsi$ is a $\P^1$-fibration with   fiber class~$[\P]$.
\end{proof}

Let $p\colon \cL_2^\sigma(X) \to F_2^\sigma(X)$ be the universal $\sigma$-plane and let $q\colon \cL_2^\sigma(X) \to X$ be the natural morphism.\ The analogue of the diagram~\eqref{diagram:big4} is
\begin{equation}\label{diagram:big6}
\vcenter{\xymatrix@C=14pt@M=3pt{
&&& \cL_2^\sigma(X) \ar[dlll]_q \ar[drrr]^p \ar@{-->}[dl]^(.4){q'} \ar@{..>}[dr]_(.4){q''}
\\
X && 
\P_X(\cU_X) \ar[ll]_{\pi} \ar[d]_{\rho_1} && 
\,\P_X(\cU_X) \times_{\P(V_5)} Y_{A,V_5} \ar[d] \ar@{_{(}->}[ll] &&
F_2^\sigma(X) \ar[d]^\tsi
\\
&&
\P(V_5) &&
\,Y_{A,V_5} \ar@{_{(}->}[ll] &&
\tY_{A,V_5} \ar[ll]_{f_{A,V_5}} \ar@{^{(}->}[rr]^-\iota && \tY_A,
}}
\end{equation}
 where $\rho_1$ is a fibration in 3-dimensional quadrics, and 
the dashed and dotted arrows are constructed in the same way.

\begin{lemm}\label{lemma:degree-q}
The map $q''$ is   generically finite   of degree $2$ and $q'_*([\cL_2^\sigma(X)]) = 12\rho_1^*\tilde{h}$.\ 
 {In particular, the map $q$ is generically finite of degree $12$.}
\end{lemm}

\begin{proof}
Take a point $(x,v) \in \P_X(\cU_X) \times_{\P(V_5)} Y_{A,V_5} \subset \P_X(\cU_X) \subset X \times \P(V_5)$.\ If ${\gamma_X}(x) = [U_2]$, we have $v \in \P(U_2) \cap Y_{A,V_5}$ and $(q'')^{-1}(x,v)$ is the Hilbert scheme of planes in $\rho_1^{-1}(v)$ passing through $x$.\ But $\rho_1^{-1}(v)$ is either a cone over $\P^1\times\P^1$ or a cone over a conic with vertex a line. \ In the first case, there is a unique plane of each type through $x$, unless $x$ is the vertex of the cone, and in the second case, 
there is a unique plane (with multiplicity 2) through any  point of the  cone  {not lying on} the vertex.\ Thus, $q''$ is generically finite of degree 2.

It follows that the class $q'_*([\cL_2^\sigma(X)]) \in H^2(\P_X(\cU_X);\Z)$ is twice the class of the fiber product $\P_X(\cU_X) \times_{\P(V_5)} Y_{A,V_5}$,
hence twice the pullback via $\rho_1$ of the class of $Y_{A,V_5}$, which is equal to $6\rho_1^*\tilde{h}$.
\end{proof}

Geometrically, this means that for a general point $x$ of  a GM sixfold $X$, there are 12 planes passing through $x$ and contained in $X$.

For every $v \in \P(V_5)$, set 
\begin{equation*}
Q_v := \pi(\rho_1^{-1}(v))  = \gamma_X^{-1}(\P(v \wedge V_5)) \subset X.
\end{equation*}
 This is a 3-dimensional quadric.

\begin{lemm}\label{lemm:qp-p1p2a}
Let $y \in \tY_{A,V_5}$ and set $v := f_{A,V_5}(y) \in Y_{A,V_5} \subset \P(V_5)$.\  We have $q_*p^*([{\P}]) = [Q_v] = \gamma_X^*\bsi_3 \in H^6(X;\Z)$. 
 \end{lemm}

\begin{proof}
By the proof of Theorem~\ref{theorem:f2}, the line $\P_y$ parameterizes planes (of  one of the two possible types) on the singular quadric $Q_v$.\ Since these planes cover the whole quadric  and there is a unique such plane through any smooth point of $Q_v$, the map $q\colon p^{-1}(\P_y) \to Q_v$ is birational, hence $q_*p^*([{\P}]) = [Q_v]$.\ By definition of $Q_v$, we have $[Q_v] = \gamma_X^*([\P(v \wedge V_5)]) = \gamma_X^*\bsi_3$.
\end{proof}

\begin{coro}\label{corollary:intersection-q-p}
For any $u\in \P(V_5)$, we have $p_*q^*[Q_u] \cdot [\P] = 2$.
\end{coro}

\begin{proof}
By adjunction, it is enough to prove $[Q_u]\cdot [Q_v] = 2$ for any distinct $u,v\in \P(V_5)$.
Since the quadrics are preimages of the spaces $\P(u\wedge V_5)$ and $\P(v \wedge V_5)$ under the double covering $\gamma_X\colon X \to \Gr(2,V_5)$, 
it is enough to show that the intersection of those spaces is 1, \ie, that~$\bsi_3^2 = 1$,
which follows from Schubert calculus.
\end{proof}

\begin{lemm}\label{lemma:class-pullback}
The class $q_*p^*(p_*q^*[Q_u]\cdot\tsi^*\iota^*\tilde{h}^2) \in H^6(X;\Z)$ is contained in the subgroup
$\gamma_X^*(H^6(\Gr(2,V_5);\Z)) $ of $ H^6(X;\Z)$.
\end{lemm}

\begin{proof}
Using   diagram~\eqref{diagram:big6}, we can rewrite the class in question as
\begin{equation*}
q_*p^*(p_*q^*[Q_u]\cdot\tsi^*\iota^*\tilde{h}^2) = 
\pi_*(q'_*(p^*p_*q^*[Q_u]) \cdot \rho_1^*\tilde{h}^2).
\end{equation*}
The class $q'_*(p^*p_*q^*[Q_u])$ belongs to $H^4(\P_X(\cU_X);\Z) \isom H^4(X;\Z) \oplus H^2(X;\Z) {\cdot \rho_1^*\tilde{h}}$, hence,  by Proposition~\ref{intcoh}(b), it can be written as a linear combination of the classes $\pi^*h^2$, $\pi^*c_2(\cU_X)$, and $\pi^*h \cdot \rho_1^*\tilde{h}$.\  It is thus enough to show that each of these classes multiplied  by $\rho_1^*\tilde{h}^2$ and pushed forward to $X$ is in the required subgroup.\  From~\eqref{eq:projective-bundle} and the equality $c_1(\cU_X) = h$, one easily obtains
$\pi_*(\pi^*h^2 \cdot \rho_1^*\tilde{h}^2) = h^3$, $\pi_*(\pi^*c_2(\cU_X) \cdot \rho_1^*\tilde{h}^2) = h\cdot c_2(\cU_X)$,
and 
\begin{equation*}
\pi_*(\pi^*h\cdot \rho_1^*\tilde{h}^3) = h^3 - h\cdot c_2(\cU_X).
\end{equation*}
This proves the lemma.
\end{proof}

Consider the morphism 
\begin{equation*} 
\alpha \colon H^6(X;\Z) \lra H^2(F_2^\sigma(X);\Z),
\qquad 
x \longmapsto p_*(q^*x).
\end{equation*} 

The map $p$ is a $\P^2$-fibration for which   $q^*h$ is a relative hyperplane class.\  Therefore, $\cL_2^\sigma(X)$   is isomorphic to the projectivization of a rank-3 {vector bundle}.\  We denote by $c_1$, $c_2$, and $c_3$ {its} Chern classes, so that
\begin{equation}\label{eq:l2-rel}
q^*h^3 + p^*c_1 \cdot q^*h^2 + p^*c_2 \cdot q^*h + p^*c_3 = 0.
\end{equation} 
Multiplying by $q^*h$, one obtains
\begin{equation}\label{eq:l2-rel2}
q^*h^4 + p^*(c_2 - c_1^2) \cdot q^*h^2 + p^*(c_3 - c_1c_2) \cdot q^*h - p^*(c_1c_3) = 0.
\end{equation} 

The primitive and vanishing cohomology subgroups $ H^6(X;\Z)_{00}=H^6(X;\Z)_0\subset  H^6(X;\Z)$    (they are equal by Lemma~\ref{lemma:primitive-vanishing}) and  $H^2(\tY_A;\Z)_0\subset H^2(\tY_A;\Z)$ were defined in \eqref{defh00} and \eqref{defprime}.

\begin{theo}\label{theorem:alpha-x6}
Let $X$ be a smooth   GM sixfold, with associated Lagrangian data $(V_6,V_5,A)$, such that assumption~\eqref{assumption-y} holds.\ We have
\begin{equation}\label{eq:alpha2-x6}
\forall x \in H^6(X;\Z)_{{00}} \qquad \alpha(x)^2\cdot c_2
= 12x^2.
\end{equation}
In particular, the restriction $\alpha_0\colon H^6(X;\Z)_{{00}} \to H^2(F_2^\sigma(X);\Z)$ is injective.\   Furthermore, it induces an isomorphism 
\begin{equation*}
\beta\colon H^6(X;\Z)_{00} \isomlra H^2(\tY_A;\Z)_0
\end{equation*}
compatible with the Beauville--Bogomolov form
\begin{equation*}
\forall x \in H^6(X;\Z)_{00}\qquad q_B(\beta(x)) = x^2,
\end{equation*}
\ie, an isomorphism of polarized Hodge structures.
\end{theo}

\begin{proof}
We use the same argument as in the proof of Theorem~\ref{theorem:alpha-x4}.\   Since $p$ is a $\P^2$-bundle and $q^*h$ is a relative hyperplane section, we may write
\begin{equation*}
\forall x \in H^6(X;\Z)\qquad q^*x = p^*x_1\cdot q^*h^2 + p^*x_2\cdot q^*h + p^*x_3,
\end{equation*}
where $x_i \in H^{2i}(F_2^\sigma(X);\Z)$ for $i\in\{1,2,3\}$.\  We have then
$\alpha(x) = x_1$.\
To see what the  {primitivity of} $x$ means, we compute, using~\eqref{eq:l2-rel},
\begin{eqnarray*}
q^*( x \cdot h ) & = &
p^*x_1\cdot q^*h^3 + p^*x_2\cdot q^*h^2 + p^*x_3 \cdot q^*h \\ &= &
p^*(x_2 - x_1\cdot c_1) \cdot q^*h^2 + p^*(x_3 - x_1\cdot c_2) \cdot q^*h  - p^*(x_1\cdot c_3).
\end{eqnarray*}
Thus, the condition $x \cdot h  = 0$ implies  
\begin{equation*}
x_1\cdot c_3 = 0,
\qquad
x_2 = x_1 \cdot c_1,
\qquad
x_3 = x_1 \cdot c_2,
\end{equation*} 
and   we can rewrite
\begin{equation*}
q^*x = p^*(\alpha(x)) \cdot ( q^*h^2 + p^*c_1 \cdot q^*h + p^*c_2 ).
\end{equation*} 
Taking squares and using \eqref{eq:l2-rel} and~\eqref{eq:l2-rel2}, 
 we obtain
\begin{eqnarray*}
(q^*x)^2 & = &
p^*(\alpha(x)^2) \cdot (q^*h^4 + 2p^*c_1 \cdot q^*h^3 + p^*(c_1^2 + 2c_2)\cdot q^*h^2 + 2p^*(c_1c_2)\cdot q^*h + p^*c_2^2) \\ & =&
p^*(\alpha(x)^2) \cdot (p^*c_2 \cdot  q^*h^2 + p^*(c_1c_2 - c_3) \cdot  q^*h + p^*(c_2^2 - c_1c_3)) \\ & = &
\alpha(x)^2 \cdot c_2.
\end{eqnarray*}
On the other hand, by Lemma~\ref{lemma:degree-q}, we have $(q^*x)^2 = 12x^2$.\ This proves~\eqref{eq:alpha2-x6}.\ The injectivity of $\alpha_0$ then follows as in the proof of Theorem~\ref{theorem:alpha-x4}.

Since $x \in H^6(X;\Z)_{00}$, we have
\begin{equation*}
\alpha(x)\cdot [\P] = x\cdot q_*p^*[\P]  =x\cdot [Q_v] = { x \cdot \gamma_X^*\bsi_3} = 0.
\end{equation*}
A combination of this equality with Proposition~\ref{proposition:hodge-f2-tyav} shows that there is an injective   map $\beta\colon H^6(X;\Z)_{00} \to H^2(\tY_A;\Z)$ such that
\begin{equation*}
\alpha(x) = \tsi^*(\iota^*(\beta(x))). 
\end{equation*}
It remains to show that the image of $\beta$ is in the primitive cohomology $H^2(\tY_A;\Z)_0$  and that $\beta$ is compatible with the Beauville--Bogomolov form.

Let $x \in H^6(X;\Z)_{00}$.\  Set $m := \beta(x)\cdot \tilde{h}^3 = \beta(x) \cdot \iota_*\iota^*\tilde{h}^2 = \iota^*\beta(x) \cdot \iota^*\tilde{h}^2$.\  Then 
\begin{equation*}
\alpha(x) \cdot \tsi^*\iota^*\tilde{h}^2 = 
\tsi^*\iota^*\beta(x) \cdot \tsi^*\iota^*\tilde{h}^2 = m[\P].
\end{equation*}
Multiplying this by $p_*q^*[Q_u]$ and using Corollary~\ref{corollary:intersection-q-p}, we obtain
\begin{equation*}
\alpha(x) \cdot (p_*q^*[Q_u] \cdot \tsi^*\iota^*\tilde{h}^2) = 2m. 
\end{equation*}
By adjunction, this is equal to $x \cdot q_*p^*(p_*q^*[Q_u] \cdot \tsi^*\iota^*\tilde{h}^2)$ and by Lemma~\ref{lemma:class-pullback},
the latter is  zero since $x$  {is} in the vanishing cohomology; therefore, $m = 0$.\ This proves   $\beta(x) \in H^2(\tY_A;\Z)_0$.

To show   compatibility with the Beauville--Bogomolov form, we observe
\begin{equation*}
12x^2 = 
\alpha(x)^2c_2 =
\tsi^*\iota^*\beta(x)^2c_2 = 
\beta(x)^2\iota_*\tsi_*c_2.
\end{equation*}
On the other hand, by Proposition \ref{prop321} below, we have $\iota_*\tsi_*c_2 = 6\tilde{h}^2$, hence, by~\eqref{eq:bb-form},
\begin{equation*}
x^2 = \frac1{12}\beta(x)^26\tilde{h}^2 = \frac12\beta(x)^2\tilde{h}^2 = q_B(\beta(x)). 
\end{equation*}

Finally, the lattices $H^6(X;\Z)_{00}$ and $H^2(\tY_A;\Z)_0$ have same rank 22 and same discriminant group $(\Z/2\Z)^2$ (Proposition \ref{ltf} for $H^6(X;\Z)_{00} $ and \cite[(1.0.9)]{og2} for $ H^2(\tY_A;\Z)_0$),    hence the injective  isometry $\beta$ is a bijection.
\end{proof}

\subsection{Nearby Lagrangians}

Our aim here is to prove Proposition~\ref{prop321}, which was used in the proof of Theorem \ref{theorem:alpha-x6}.\ This section was inspired by~\cite[Section~6]{Fe}.

We start with some preparation.\  Let  $A_1$ and $A_2$ be  Lagrangian subspaces
in a symplectic vector space $\symv$, such that the intersection
\begin{equation*}
B := A_1 \cap A_2 
\end{equation*}
has codimension 2 in both $A_1$ and $A_2$.

\begin{lemm}\label{lemma:family-lagrangians}
If $\codim_{A_1}(B) = \codim_{A_2}(B) = 2$, the Lagrangian subspaces $A \subset \symv$ such that 
\begin{equation*}
\codim_{A_1}(A \cap A_1) = \codim_{A_2}(A \cap A_2) \le 1
\end{equation*}
are parameterized by the line $\P(A_1/B) \cong \P(A_2/B) \cong \P^1$.\
Moreover, if $A',A'' \subset \symv$ are two distinct such subspaces, $A' \cap A'' = B$.
\end{lemm}

\begin{proof}
 {Since} $B = A_1 \cap A_2$,  {we have} $B^\perp = A_1 + A_2 \subset \symv$.\  
 {The vector space} $B^\perp/B$ is symplectic   of dimension 4 and for any Lagrangian subspace $A \subset \symv$, 
\begin{equation*}
\bar{A} := (A \cap B^\perp)/(A \cap B)
\end{equation*}
is a Lagrangian subspace in $B^\perp/B$
(called the $B$-isotropic reduction of $A$).\  Since the   subspaces $\bar{A}_i = A_i/B$ do not intersect, they give a Lagrangian direct sum decomposition
\begin{equation*}
B^\perp/B = \bar{A}_1 \oplus \bar{A}_2.
\end{equation*}
Assume now   $\codim_{A_i}(A \cap A_i) \le 1$.\  Note that $A \cap A_1 \ne A \cap A_2$ (since otherwise $A_1 \cap A_2$ would be at most 1-codimensional),
hence $A = (A \cap A_1) + (A \cap A_2)$.\  This implies $A \subset A_1 + A_2 = B^\perp$,   hence $B \subset A^\perp = A$.\  Thus $\bar{A} = A/B \subset B^\perp/B$
and in particular,  $A$ is determined by the space $\bar{A}$ as its preimage under the linear projection $B^\perp \to  B^\perp/B$.

The conditions $\codim_{A_i}(A \cap A_i) \le 1$ imply that the line $\P(\bar{A})$ intersects each skew line  $\P(\bar{A}_i)$ in $\P(B^\perp/B)$.\ Finally, the pairing  between $ \bar{A}_1$ and $\bar{A}_2$  induced by the symplectic form on $B^\perp/B$ is non-degenerate, hence for every point of $\P(\bar{A}_1)$,
there is a unique point in $\P(\bar{A}_2)$ such that the line joining them is Lagrangian.\  Thus, the set of $\bar{A}$ (and hence the set of $A$ as well)
is parameterized by either of the lines $\P(\bar{A}_1)$ or $\P(\bar{A}_2)$. 

It follows from the above description  that the lines $\P(\bar{A})$ form one   connected component of the scheme of lines on a smooth quadric in $\P(B^\perp/B)$
(the lines $\P(\bar{A}_1)$ and $\P(\bar{A}_2)$ being   in the other component).\  In particular, two such distinct lines do not intersect, hence their preimages
in $\P(B^\perp)$ intersect along $\P(B)$.\  This proves the second statement.
\end{proof}

Assume now that the symplectic vector space $\symv$ is $\bw3V_6$.\  Let $B \subset \bw3V_6$ be an isotropic subspace of dimension 8
(hence of codimension 2 in any Lagrangian containing it).\  Set
\begin{equation*}
Y_B := \{ v \in \P(V_6) \mid B \cap (v \wedge \bw2V_6) \ne 0 \}
 {\qquad\text{and}\qquad 
Y_{B,V_5} := Y_B \cap \P(V_5).}
\end{equation*}
\begin{rema}
A parameter count shows $\dim(Y_B) \le 2$  {for   general $B$}.\  In fact, this is even true for a general $B$ inside a given Lagrangian subspace which contains no decomposable vectors.
\end{rema}

Let $A_1,A_2 \subset \bw3V_6$ be Lagrangian subspaces with no decomposable vectors such that $B := A_1 \cap A_2$ has codimension 2 in each of them.\ Consider the family $\{A_p\}_{p \in \P^1}$ of Lagrangian subspaces discussed in Lemma~\ref{lemma:family-lagrangians} and set
\begin{equation*}
\cY^{\ge 2}_{A_1,A_2} :=  {\{(v,p) \in \P(V_6) \times \P^1 \mid v \in Y_{A_p}^{\ge 2} \}}
\subset \P(V_6) \times \P^1.
\end{equation*}
 {We denote by $\pr\colon \cY^{\ge 2}_{A_1,A_2} \to \P(V_6)$ the first projection and set $\cY^{\ge 2}_{A_1,A_2;V_5} := \cY^{\ge 2}_{A_1,A_2} \times_{\P(V_6)} \P(V_5)$.}

\begin{lemm}\label{lemma:projection-two-lagrangians}
Let $A_1,A_2 \subset \bw3V_6$ be Lagrangian subspaces with no decomposable vectors such that $B = A_1 \cap A_2$ has codimension $2$ in each of them.\  If $\dim (Y_B) \le 2$, we have 
\begin{equation*}
Y_{A_1} \cap Y_{A_2} = \pr(\cY^{\ge 2}_{A_1,A_2}) 
\end{equation*}
and the map $\pr\colon \cY^{\ge 2}_{A_1,A_2} \to Y_{A_1} \cap Y_{A_2}$ is an isomorphism  over a dense open subset of $Y_{A_1} \cap Y_{A_2}$.\ Moreover, for any $V_5 \subset V_6$ such that $\dim (Y_{B,V_5} ) \le 1$, the map 
$\pr\colon \cY^{\ge 2}_{A_1,A_2;V_5} \to Y_{A_1,V_5} \cap Y_{A_2,V_5}$ is again an isomorphism  over a dense open subset.
\end{lemm}

\begin{proof}
Let us prove $\pr(\cY^{\ge 2}_{A_1,A_2}) \subset Y_{A_1} \cap Y_{A_2}$.\  If $A_p$ is any of the Lagrangian spaces in the family and $v \in Y^{\ge 2}_{A_p}$,
we have $\dim(A_p \cap (v \wedge \bw2V_6)) \ge 2$.\  But $\codim_{A_i}(A_p \cap A_i)  \le 1$, hence 
$A_i \cap (v \wedge \bw2V_6) \ne 0$, so $v \in Y_{A_i}$ both for $i = 1$ and $i = 2$.

Since $\cY_{A_1,A_2}^{\ge 2}$ is proper, it remains to show  {that $\pr$ is an isomorphism over a dense open subset}.\ Since $Y_{A_1} $ and $ Y_{A_2}$ are  distinct hypersurfaces in $\P(V_6) = \P^5$, any irreducible component
of their intersection   has dimension at least 3, so it is enough to show that the map 
$\pr\colon \cY^{\ge 2}_{A_1,A_2} \to Y_{A_1} \cap Y_{A_2}$ is an isomorphism over the complement of $Y_B$.  

Let $v \in (Y_{A_1} \cap Y_{A_2}) \setminus Y_B$.\  We first show that $v$ is   a smooth point of $Y_{A_i}$.\ 
Assume to the contrary $v \in Y^{\ge 2}_{A_1}$.\  Then, $\dim(A_1 \cap (v \wedge \bw2V_6)) \ge 2$ but $B \cap (v \wedge \bw2V_6) = 0$.\
It follows that 
\begin{equation*}
A_1 = B \oplus (A_1 \cap (v \wedge \bw2V_6)) 
\end{equation*}
(and in particular, the second summand is 2-dimensional).\  On the other hand, take any non-zero   $a \in A_2 \cap (v \wedge \bw2V_6)$.\ Then, $a$ is orthogonal to both summands in the above equation (since the first summand is contained in $A_2$ and the second in $v \wedge \bw2V_6$).\ Therefore, $a \in A_1^\perp = A_1$, hence $a \in A_1 \cap A_2 = B$ and $v \in Y_B$, a contradiction.\ The same argument works for $A_2$ instead of $A_1$.

We now know that both spaces $A_i \cap (v \wedge \bw2V_6)$ are one-dimensional.\  If $a_1$ and $a_2$ are   generators, their projections  to $B^\perp/B$ are linearly independent  (otherwise, $v \in Y_B$).\ Furthermore, 
\begin{equation*}
A := B \oplus \langle a_1, a_2 \rangle
\end{equation*}
is a Lagrangian subspace in $\bw3V_6$ and its intersections with  $A_1$ and $A_2$ are both 9-dimensional.\ Therefore, $A = A_p$ for some $p \in \P^1$ and, since  $\langle a_1,a_2 \rangle \subset A \cap (v \wedge \bw2V_6)$, we obtain $(v,p) \in \cY^{\ge 2}_{A_1,A_2}$ and $v \in \pr(\cY^{\ge 2}_{A_1,A_2})$.

Now let $(v,p) \in \cY^{\ge 2}_{A_1,A_2}$ with $v \notin Y_B$.\  The space $A_p$ intersects $v \wedge \bw2V_6$ away from~$B$, hence, by the second part of Lemma~\ref{lemma:family-lagrangians}, $p$ is uniquely determined by $v$.\ This means that the map $\pr$ is an isomorphism over the complement of $Y_B$.

Finally, if a hyperplane $V_5 \subset V_6$ satisfies $\dim(Y_{B,V_5}) \le 1$, 
 the subset $(Y_{A_1,V_5} \cap Y_{A_2,V_5}) \setminus Y_{B,V_5}$ is  dense open in $Y_{A_1,V_5} \cap Y_{A_2,V_5}$ and   the map $\pr$ is an isomorphism over it.
\end{proof}

Let now   $X$ be a smooth special GM sixfold such that~\eqref{assumption-y} holds, with Lagrangian subspace $A_1 \subset \bw3V_6$   and Pl\"ucker hyperplane $V_5 \subset V_6$.\  Since $A_1 \cap \bw3V_5 = 0$, the canonical projection $\bw3V_6 \to \bw3V_6/\bw3V_5 \cong \bw2V_5$ induces an isomorphism $A_1 \cong \bw2V_5$.

Recall that $\cL_2^\sigma(X) = \P_{F_2^\sigma(X)}(\cP)$, where $\cP$ is a rank-3 vector bundle on $F_2^\sigma(X)$ 
and the map $q$ {in~\eqref{diagram:big6}} is induced by an embedding of vector bundles $\cP \to (\C \oplus \bw2V_5) \otimes \cO_{F_2^\sigma(X)}$.\ The composition of the above embedding with the projection to $\bw2V_5 \otimes \cO_{F_2^\sigma(X)}$ is still a monomorphism of vector bundles (since planes on $X$ do not pass through the vertex of the cone $\CGr(2,V_5)$).

The  vector bundle $\cP^\vee$ is  globally generated by its space $\bw3V_5 \cong \bw2V_5^\vee$ of global sections; 
therefore, for $\omega $ general in $ \P(\bw3V_5)$,  the zero-locus in $F_2^\sigma(X)$ of $\omega$ viewed as an element  of $ H^0(F_2^\sigma(X),\cP^\vee)$  has dimension 1  
and the set of $\omega$ such that this dimension jumps has codimension~2 or more.\ Thus, for a general choice of a line $\P^1 \subset \P(\bw3V_5)$, the zero-locus is 1-dimensional for every point $\omega \in \P^1$.

Choose a general codimension-2 subspace $B \subset A_1$ such that
\begin{equation*}
\begin{aligned}
& \bullet X_B := X \times_{\P(\sbw2V_5)} \P(B)
\quad\textnormal{is a smooth special fourfold},\\
& \bullet \dim (Y_{B,V_5}) \le 1, \\
&\bullet \textnormal{for any $\omega \in B^\perp \cap \bw3V_5$ the zero-locus of $\omega$ in $F_2^\sigma(X)$ is $1$-dimensional.}
 \end{aligned}
\end{equation*} 
Note that $\dim(B^\perp \cap \bw3V_5) = 2$ (the dimension is obviously at least 2 and it is at most 2 since $A_1 \cap \bw3V_5 = 0$).\
 {By~\cite[{Proposition~3.14(a)}]{DK}}, the Lagrangian subspace of the fourfold $X_B$ is 
\begin{equation*}
A_2 := B \oplus (B^\perp \cap \bw3V_5).
\end{equation*} 
 Each $\omega \in B^\perp \cap \bw3V_5$ determines a hyperplane in $\P(\bw2V_5)$ containing $\P(B)$.\  We denote by~$X_\omega$
the corresponding hyperplane section of $X$.\  For all $\omega$, we have inclusions
\begin{equation*}
X_B \subset X_\omega \subset X  
\end{equation*}
and every $X_\omega$ is a special GM fivefold  which is smooth for general $\omega$.\ 
We set
\begin{equation*}
\widetilde{D}(B) := \{ (\Pi,\omega) \mid \Pi \subset X_\omega,\ \omega \in \P(B^\perp \cap \bw3V_5) \} \subset F_2^\sigma(X) \times \P^1.
\end{equation*}
 {Let $\pr\colon \widetilde{D}(B) \to F_2^\sigma(X)$ be the projection.\  It}
 gives a birational map $\widetilde{D}(B) \to D(B)$, where $D(B) \subset F_2^\sigma(X)$ is the degeneracy locus of the morphism $(B^\perp \cap \bw3V_5) \otimes \cO_{F_2^\sigma(X)} \to \cP^\vee$.\  It satisfies
\begin{equation*}
[D(B)] = c_2(\cP)
\end{equation*}
{in $H^4(F_2^\sigma(X);\Z)$.}

\begin{prop}\label{prop321}
We have $\iota_*\tsi_*c_2(\cP) = 6\tilde{h}^2$  in $H^4( \tY_{A_1};\Z)$.
\end{prop}

\begin{proof}
The above discussion shows that we need to describe $\iota_*\tsi_*[D(B)]$.\   Let $\P^1_0 \subset \P^1$ be the open subset of those $\omega$ such that $X_\omega$ is a smooth hyperplane section of $X$, let $\widetilde{D}_\omega \subset \widetilde{D}(B)$ be the fiber of $\widetilde{D}(B)$ over $\omega \in \P^1$, and let $\widetilde{D}_0 \subset \widetilde{D}(B)$ be the preimage of $\P^1_0$.

Choose any $\omega \in \P^1_0$ and let $A_\omega := A(X_\omega)$ be the Lagrangian subspace associated with~$X_\omega$.\ By~\cite[{Proposition~3.14(a)}]{DK}, we have
\begin{equation*}
\dim (A_\omega \cap A_1) = \dim (A_\omega \cap A_2) = 9.
\end{equation*}
This shows that the pencil $\P^1$ is the same as the pencil of Lemma~\ref{lemma:family-lagrangians}.

 By Theorem~\ref{theorem:f2}(b), the Stein factorization of the map $\sigma\colon F_2^\sigma(X_\omega) \to \P(V_5)$
is     the double covering $\tY^{\ge 2}_{A_\omega,V_5}$ of $Y^{\ge 2}_{A_\omega,V_5} \subset \P(V_5)$.\ This means that the Stein factorization of the map $\sigma\colon \widetilde{D}_0 \to \P(V_5) \times \P^1_0$ is  a double covering of the subscheme
 \begin{equation*}
\cY^{\ge 2}_{A_1,A_2;V_5} \times_{\P^1} \P^1_0.
\end{equation*}
By Lemma~\ref{lemma:projection-two-lagrangians}, its projection to $\P(V_5)$ is birational onto $Y_{A_1,V_5} \cap Y_{A_2,V_5}$.

This means that the composition $\pr \circ \sigma \colon \widetilde{D}(B) \to \P(V_5)$ is generically finite of degree 2 onto $Y_{A_1,V_5} \cap Y_{A_2,V_5}$.\ Since   it factors through   $\tY_{A_1,V_5}$,    the induced map $\tsi \colon \widetilde{D}(B) \to \tY_{A_1,V_5}$
is either birational onto $\tY_{A_1,V_5} \times_{\P(V_5)} Y_{A_2,V_5}$, or   generically surjective of degree 2 onto a section
of the double cover $\tY_{A_1,V_5} \times_{\P(V_5)} Y_{A_2,V_5} \to Y_{A_1,V_5} \cap Y_{A_2,V_5}$.\ In the first case, we have $\tsi_*[D(B)] = 6\iota^*\tilde{h}$, hence $\iota_*\tsi_*[D(B)] = 6\tilde{h}^2$.\ In the second case, we have
$\iota_*\tsi_*([D(B)]) + \tau_{A_1}^*\bigl(\iota_*\tsi_*([D(B)])\bigr) = 12\tilde{h}^2$, where $\tau_{A_1}^*$ is the action  on $H^4( \tY_{A_1};\Z)$ of the involution of the double covering $f_{A_1}$.\ In the second case,  the same arguments used at the end of the proof of~Corollary~\ref{corollary:alpha-h2} show that we also have $\iota_*\tsi_*[D(B)] = 6\tilde{h}^2$.
\end{proof}

\subsection{Period points and period maps} \label{sec41}

In this section, we discuss  period points and period maps for smooth GM varieties   of  dimensions 4 or 6 and for double EPW sextics.\ We use the notation of Section~\ref{mcl}.\ In particular, we {consider the lattices $\Gamma_4$, $\Gamma_6$, and $\Lambda$ defined by~\eqref{eq:gamma-lattices}.\  Consider the automorphism group $O(\Lambda)$ and}
the {\sf  stable orthogonal group}
\begin{equation*}
\widetilde O(\Lambda)\subset O(\Lambda)
\end{equation*}
of automorphisms of $\Lambda$ which act trivially on its discriminant group $D(\Lambda)=\Lambda^\vee/\Lambda$.\ It has index 2 in $ O(\Lambda)$. 

Another description of $\widetilde{O}(\Lambda)$ will be important.\ Consider the even lattice $\Gamma_6$.\  By choosing vectors $e_1$ and $e_2$ with square 2  in the first and  second copies of $U$  in~$\Gamma_6$, we obtain a primitive embedding of the lattice $I_{2,0}(2)$ into $\Gamma_6$.\  Furthermore, the group $O(\Gamma_6)$ acts transitively on the set of such embeddings~(\cite{jam}).\ The orthogonal sublattice $\langle e_1,e_2 \rangle^\perp   \subset \Gamma_6$  is isomorphic to~$\Lambda(-1)$ (it is even, of signature $(2,20)$,   with discriminant group~$(\Z/2\Z)^2$). 

The subgroup $O(\Gamma_6)_{\langle e_1,e_2 \rangle} \subset O(\Gamma_6)$ stabilizing the sublattice $\langle e_1,e_2 \rangle$ preserves the orthogonal $\Lambda(-1)$.\ This defines a map $O(\Gamma_6)_{\langle e_1,e_2 \rangle} \to O(\Lambda)$ which is surjective and the stable group $\widetilde{O}(\Lambda)$ is the isomorphic image under this map of   the subgroup $O(\Gamma_6)_{e_1,e_2} \subset O(\Gamma_6)_{\langle e_1,e_2 \rangle}$ of elements stabilizing both $e_1$ and $e_2$.

Analogously, in the lattice $\Gamma_4$, there are vectors $e_1$ and $e_2 $ generating a sublattice isomorphic to $I_{2,0}(2)$ such that $e_1+e_2$ is characteristic in $\Gamma_4$.\   Again by~\cite{jam}, the group~$O(\Gamma_4)$ acts transitively on the set of such embeddings, the orthogonal $\langle e_1,e_2 \rangle^\perp $ is  isomorphic to~$\Lambda$, and there are morphisms
$O(\Gamma_4)_{\langle e_1,e_2 \rangle} \thra O(\Lambda)$ and $O(\Gamma_4)_{e_1,e_2} \isomto \widetilde{O}(\Lambda)$ ({see}~\cite[Section~5.1]{dims} {for details}).

The groups $\widetilde{O}(\Lambda)$ and $ {O}(\Lambda)$ act  properly and discontinuously on the complex variety
\begin{equation}\label{defom}
\Omega:=\{\omega\in \P(\Lambda\otimes\C)\mid   \omega\cdot \omega = 0,\ \omega \cdot \bar \omega < 0\} .
\end{equation}
The quotient
 \begin{equation*}
\cD:= \widetilde O(\Lambda)\backslash\Omega 
\end{equation*}
is a quasi-projective 20-dimensional variety.\ It has a canonical involution $r_\cD$, associated with the further degree-2 quotient  $\cD\to  O(\Lambda)\backslash\Omega $. 

\begin{prop}\label{pp}
Let $X$ be a smooth GM variety  of   dimension $n= 4$ or $6$.\  The one-dimensional subspace $H^{n/2+1,n/2-1}(X)\subset  H^n(X,\C)$ gives rise to a well defined point in $\cD$.
\end{prop}

This point is called the {\sf period point} of $X$ and will be denoted by $\wp(X)$.

\begin{proof}
Assume first $n=4$.\ The abelian group $ H^4(\Gr(2,V_5);\Z) $ is generated by the Schubert classes $\bsi_{1,1}$ and $\bsi_2$.\ By \cite[Section~5.1]{dims}, there exists  an isometry $\phi\colon H^4(X;\Z)\isomto  \Gamma_4$, called a {\sf marking} of $X$, such that
\begin{equation}\label{eq:isometry4}
\phi^{-1}(e_1)= \gamma_X^*\bsi_{1,1}\qquad\text{and}\qquad
\phi^{-1}(e_2)= \gamma_X^*\bsi_2-\gamma_X^*\bsi_{1,1},
\end{equation} 
where $e_1,e_2 \in \Gamma_4$ were defined above.\ Any two markings differ by the action of an element of  the group $O(\Gamma_4)_{e_1,e_2} \cong \widetilde{O}(\Lambda)$.
The marking carries the vanishing cohomology lattice $H^4(X;\Z)_{00}$ (defined in \eqref{defh00}) onto the orthogonal $\langle e_1,e_2 \rangle^\perp \cong \Lambda$.\ Its complexification $\phi_\C\colon  H^4(X;\C)\isomto  {\Gamma_4}\otimes\C$ takes the one-dimensional subspace $H^{3,1}(X)$ (see Proposition \ref{hn}), 
which is orthogonal to $\gamma_X^*H^4(\Gr(2,V_5);\C)$, to a point in the manifold $\Omega$ defined in \eqref{defom}.\ The equivalence class of this point  in  the quotient $\cD= \widetilde O(\Lambda)\backslash\Omega $ is well defined.

The situation when $n=6$ is similar:  the abelian group $ H^6(\Gr(2,V_5);\Z) $ is generated by~$\bsi_{2,1}$ and $\bsi_3$, there exists 
a marking $\phi\colon H^6(X;\Z)\isomto  \Gamma_6$ such that 
\begin{equation}\label{eq:isometry6}
\phi^{-1}(e_1)= \gamma_X^*\bsi_{2,1}\qquad\text{and}\qquad
\phi^{-1}(e_2)= \gamma_X^*\bsi_3,  
\end{equation} 
where again $e_1,e_2 \in \Gamma_6$ were defined above, and any two markings differ by the action of an element of  the group $O(\Gamma_6)_{e_1,e_2} \cong \widetilde{O}(\Lambda)$.\ The marking carries the  vanishing cohomology lattice $H^6(X;\Z)_{00}$ (defined in \eqref{defh00}) onto the orthogonal $\langle e_1,e_2 \rangle^\perp \cong \Lambda(-1)$, 
and its complexification $\phi_\C$ takes the one-dimensional subspace $H^{4,2}(X)$ (see Proposition \ref{hn})  to a point in  {the same domain} $\Omega$ 
 {(note that the \emph{anti-isometry} property of $\phi_\C$ is compensated by the change in sign in the Hodge--Riemann relations for a $(4,2)$-class on a 6-fold in comparison with a $(3,1)$-class on a 4-fold)}
whose equivalence class in $\cD $ is well defined.
\end{proof}

\begin{rema}\label{newrem}
If, in the above construction of the period point, we   replace the conditions~\eqref{eq:isometry4} and~\eqref{eq:isometry6} by similar conditions with $e_1$ and $e_2$ exchanged, we obtain  a new period point which is $ r_\cD(\wp(X))$ (this is because there is an  element of $O(\Gamma_n)_{\langle e_1,e_2 \rangle}$ which exchanges $e_1$ and $e_2$, and the image of this  isometry by the surjection $O(\Gamma_n)_{\langle e_1,e_2 \rangle} \thra  {O}(\Lambda) $ is not in $\widetilde O(\Lambda)$).
\end{rema}

An analogous construction can be made   in another situation: if $\tY_A$ is  a smooth double EPW sextic, the one-dimensional subspace $H^{2,0}(\tY_A)\subset  H^2(\tY_A;\C)_0$ gives rise to  a period point $\wwp(\tY_A)$ in {\em the same} variety $\cD$ (see \cite[Section~4.2]{og6}).\ This period point may also be defined for all   Lagrangian subspaces $A$ with no decomposable vectors (\ie, even when $Y_A^{\ge 3}\ne\vide$; see \cite[Section~5.1]{og6}).\ The main result of \cite{og2} is $\wwp(\tY_{A^\bot})=r_\cD(\wwp(\tY_A))$.

\begin{lemm}\label{lemma:periods-pointwise}
For any smooth GM variety $X$ of dimension~$4$
 or $6$, with associated Lagrangian~$A(X)$ satisfying~\eqref{assumption-y}, one has either $\wp(X) = \wwp(\tY_{A(X)})$ or $\wp(X) = r_\cD(\wwp(\tY_{A(X)}))$.
\end{lemm}

\begin{proof}
Consider first the case of fourfolds.\ Choose markings $\phi$ for $X$, and $\psi$  for $\tY_{A(X)}$, and consider the commutative diagram
\begin{equation*}
\xymatrix{
H^4(X; \Z)_{00}(-1) \ar[d]_\phi \ar[r]^-\beta &
H^2(\tY_{A(X)}; \Z)_0 \ar[d]^\psi \\
\Lambda \ar[r]^{\psi \circ \beta \circ \phi^{-1}} & \Lambda,
}
\end{equation*}
where $\beta$ is the isomorphism of Theorem~\ref{theorem:alpha-x4}.\
Since $\beta$ is compatible with polarizations, the bottom map $g := \psi \circ \beta \circ \phi^{-1}$ is in $O(\Lambda)$.\ Since $\beta$ is a morphism of Hodge structures, we have
\begin{equation*}
g(\phi_\C(H^{3,1}(X))) = \psi_\C(H^{2,0}(\tY_{A(X)}).
\end{equation*}
If $g \in \widetilde{O}(\Lambda)$, we have $\wp(X) = \wwp(\tY_{A(X)})$; otherwise $\wp(X) = r_\cD(\wwp(\tY_{A(X)}))$.

For sixfolds, we use the same argument with the isomorphism $\beta$ of  Theorem~\ref{theorem:alpha-x6}.
\end{proof}

\begin{prop}
Either for any smooth GM variety $X$ of dimension~$4$ \textup{(}resp.~$6$\textup{)} whose associated double EPW sextic $\tY_{A(X)}$ is smooth, one has $\wp(X) = \wwp(\tY_{A(X)})$, or for any such variety $X$, one has $\wp(X) = r_\cD(\wwp(\tY_{A(X)}))$.
\end{prop} 

\begin{proof}
Let $\cX\to S$ be a smooth family of  GM varieties of dimension~$4$ (resp.~$6$) over an irreducible base $S$, such that every GM variety of dimension~4 (resp.~6) is isomorphic to some fiber of that family (see the proof of~\cite[Proposition~3.4]{KP} for a construction of such a family).\ It is classical that the period point construction defines a {\sf period map} $\wp_S\colon S\to \cD$ which is algebraic.

By~\cite{DK-moduli}, we have a family of Lagrangian data $(\cV_6,\cV_5,\cA)$, where $\cV_6$ is a rank-6 vector bundle on $S$,
$\cV_5 \subset \cV_6$ is a rank-5 vector subbundle, and $\cA \subset \bw3\cV_6$ is a Lagrangian subbundle.\ We choose an open covering $(S_\alpha)$ of $S$ such that   these vector bundles are all trivial on each~$S_\alpha$.\ Refining further the covering and applying~\cite[Proposition~3.1]{og4}, we construct, for each $ \alpha$, a family $\widetilde\cY_\alpha\to S_\alpha$ of (possibly singular) double EPW sextics.\ These families define period maps which fit together to define an algebraic map $\wwp_{S^0} \colon S^0 \to \cD$, where $S^0\subset S$ is the dense open subset where the double EPW sextics are smooth.

Since $\cD$ is separated, the sets 
\begin{equation*}
S^0_1 := \{ s \in S^0 \mid \wp_S(s) = \wwp_{S^0}(s) \}
\qquad\text{and}\qquad
S^0_2 := \{ s \in S^0\mid \wp_S( s) = r_\cD \circ\wwp_{S^0}(s) \}
\end{equation*}
are closed in $S^0$.\  By Lemma~\ref{lemma:periods-pointwise}, the dense   subset $S^{00} \subset S^0$ corresponding to smooth GM fourfolds (resp.\ sixfolds) satisfying~\eqref{assumption-y}  is the union of its closed subsets $S^0_1\cap S^{00}$ and $S^0_2\cap S^{00}$.\ Since $S^{00} $ is irreducible, one of them, say $S^0_i$, is  $S^{00} $.\ This means that $S^0_i$ contains $S^{00}$, hence its closure $ S^0$, and proves the lemma.
\end{proof}
 
To go from one of the possibilities of the proposition to the other, it suffices to change the convention defining the period point $\wp(X)$ (and the period map) as explained in Remark~\ref{newrem}.\ We may therefore assume that
\begin{equation}\label{eq:periods}
\wp(X) = \wwp(\tY_{A(X)})
\end{equation}
holds for any smooth GM fourfold or sixfold with smooth $\tY_{A(X)}$.\ This implies Theorem~\ref{th32} in full generality.

We end this section with   some consequences of~\eqref{eq:periods} 
based on results from \cite{DK} and~\cite{dims}.
 
\begin{rema}[Period partners]
In \cite[{Section 3.6}]{DK}, we said that smooth GM varieties of the same dimension are {\sf period partners} if they are constructed from the same Lagrangian subspace $A\subset \bw3V_6$  (with no decomposable vectors) but  possibly  different hyperplanes $V_5\subset V_6$.\ By Theorem~\ref{th32}, period partners  {of dimensions~4 or~6} have the same period point.

Conversely, since double EPW sextics have the same period point if and only if they are isomorphic (\cite[Theorem~1.3]{og6}),  smooth GM fourfolds (or sixfolds) are  period partners  if and only if they have the same period point.\  By~\cite[Theorem~3.25]{DK}, isomorphism classes of period partners of a GM fourfold are parametrized by $Y^1_{A^\bot}\sqcup Y^2_{A^\bot} $, modulo the finite group $\Aut(Y_{A^\bot})$ (for $A$ general, $Y^{\ge3}_{A^\bot}$ is empty and  $Y^1_{A^\bot}\sqcup Y^2_{A^\bot}=Y_{A^\bot}$).\ Similarly, isomorphism classes of period partners of a GM sixfold are parameterized by $\P(V_6^\vee) \setminus  {Y_{A^\bot}}$, modulo $\Aut(Y_{A^\bot})$. 
\end{rema}

\begin{rema}[Hodge-special GM {varieties}]
Pretending that  smooth GM varieties have coarse moduli spaces (see \cite{DK-moduli}), we go, following~\cite{dims}, through some geometrically defined subvarieties of these moduli spaces  and discuss, using the period map,  their relation with some natural divisors in the period domain $\cD$.\ We use the notation introduced in~\cite{dims}.
\begin{itemize}[leftmargin=8 mm]
\item {\bf Smooth GM fourfolds   containing  $\sigma$-planes} (\cite[Section~7.1]{dims}).\
They form a co\-dimension-2 family $\cX_{\sigma\text{-planes}}$ whose   period points cover a divisor   $\cD''_{10}\subset \cD$.\  A smooth GM fourfold $X$ contains a $\sigma$-plane if and only if $Y^{\ge3}_{A,V_5} \ne \emptyset$ (Theorem~\ref{theorem:f2}(c)).\ In particular, $Y^{\ge3}_A \ne \emptyset$; this means that $ A $ is in the O'Grady divisor $\Delta $ (\cite[(2.2.3)]{og6}) and implies $\overline{\wwp(\Delta)}= \cD''_{10}$. 

If $A$ is general in $\Delta$, the set $Y^{\ge3}_A$ is just one point $v$ (\cite[Section~5.4]{og8}).\ The condition $Y^{\ge3}_{A,V_5} \ne \emptyset$ is then equivalent to $v \in V_5$, \ie, to~$\bp_X \in v^\perp$.\ Thus the fiber of the period map  $\cX_{\sigma\text{-planes}} \to \cD''_{10}$ is equal to the hyperplane section of $Y^1_{A^\bot}\sqcup Y^2_{A^\bot} $ defined by $v^\bot$ (modulo automorphisms).\ This fiber was also described in \cite[Section~7.1]{dims} as a $\P^1$-bundle  over  a degree-10 K3 surface.
\item {\bf Smooth GM fourfolds   containing  $\tau$-quadratic surfaces} (\cite[Section~7.3]{dims}).\ They form a codimension-1 family $\cX_{\tau\text{-quadrics}}$   whose   period points cover the divisor  $  \cD'_{10}=r_\cD(\cD''_{10})\subset \cD$.\  A general  fiber of the period map $\cX_{\tau\text{-quadrics}}\to \cD'_{10}$  is, on the one hand, isomorphic to $Y_{A^\bot}$ (modulo automorphisms), and, on the other hand, birationally isomorphic to the quotient by  an involution of the symmetric square of a K3 surface (\cite[Section~7.3]{dims}).\ This fits with   \cite[Corollary 3.12 and Theorem 4.15]{og4}: a desingularization of~$\widetilde Y_{A^\bot}$ is the symmetric square of a K3 surface.

\item {\bf Smooth GM fourfolds   containing a cubic scroll} (\cite[Section~7.4]{dims}).\
They form a codimension-1 family  which contains the 3-codimensional family of   smooth GM fourfolds containing a ${\tau}$-plane (called a $\rho$-plane in \cite{dims}) and the period points of both families cover an $r_\cD$-invariant divisor   $\cD_{12}\subset \cD$.\ 
By Theorem \ref{th34}{(c)}, the condition to contain a ${\tau}$-plane implies $Z_A^{\ge 4}   \ne \emptyset$.\  The divisor $\cD_{12}$ is therefore contained in the closure of the image by the period map $\wwp$ of the locus of Lagrangians subspaces $A$ such that   $Z_A^{\ge 4}   \ne \emptyset$;  since this locus is an irreducible divisor (\cite[Lemma~3.6]{ikkr}), they are equal.  

\item {\bf Singular GM fourfolds} (\cite[Section~7.6]{dims}).\ The O'Grady  divisor $\Sigma $ (see \cite{og6}) corresponds to Lagrangian subspaces $A$ containing   decomposable vectors.\  The corresponding period points (under a suitable extension of the period map $\wwp$ discussed in \cite{og6}) fill out a divisor $  {\mathbb S}_2^\star\subset \cD$ (\cite[(4.3.3) and Proposition 4.12]{og6}); this is  the $r_\cD$-stable divisor $\cD_8$ of \cite{dims}, which corresponds to periods of nodal  GM fourfolds  (\cite[Section~7.6]{dims}).

\item {\bf Smooth GM sixfolds   containing a $\P^3$}.\ By Theorem~\ref{p32}, a
  smooth GM sixfold   contains a $\P^3$ if and only if $Y^{\ge3}_{A,V_5} \ne \emptyset$.\  In particular, $A$ is in $\Delta$ and the period point is   in $\cD''_{10}$.\  As above, when $A$ is general in $\Delta$,  one has $Y^{\ge3}_A=\{v\}$ and the condition $Y^{\ge3}_{A,V_5} \ne \emptyset$ is   equivalent  to~$\bp_X \in v^\perp$.\ Thus the fiber of the period map is equal to the hyperplane section of $\P(V_6^\vee)\setminus Y_{A^\bot} $ defined by $v^\bot$ (modulo automorphisms) and the codimension of the family of GM sixfolds containing a $\P^3$ is 2. 
\end{itemize}
\end{rema}

\newcommand{\Deg}{D}
  
\appendix\section{Linear spaces on families of quadrics}
\label{section:linear-spaces-quadrics}

Let $S$ be a base scheme which we assume to be Cohen--Macaulay and  irreducible.\
Let~$\cE$ be a vector bundle on $S$ of rank $m$ and let $\cL \subset \Sym^2\!\cE^\vee$ be a line subbundle.\ Consider the projectivization $\pr\colon \P_S(\cE) \to S$ and 
the relative line bundle $\cO(1)$ on $\P_S(\cE)$.\ Let $\cQ \subset \P_S(\cE)$ be the family of quadrics
defined as the zero-locus of the section of the line bundle $\pr^*\cL^\vee \otimes \cO(2)$ corresponding to the morphism $\cL \hra \Sym^2\!\cE^\vee$ via the isomorphism
\begin{equation*}
H^0(\P_S(\cE),\pr^*\cL^\vee \otimes \cO(2)) \cong 
H^0(S,\cL^\vee \otimes \Sym^2\!\cE^\vee) \cong
\Hom(\cL,\Sym^2\!\cE^\vee).
\end{equation*}
We denote by $\Deg_c(\cQ) \subset S$ the corank-$c$ degeneracy locus of the induced map $\cE \otimes \cL \to \cE^\vee$ of vector bundles and by $\cC$ the cokernel sheaf of this map; it is supported on $\Deg_1(\cQ)$.

In this appendix, we discuss the relative Hilbert scheme $F_k(\cQ) := \Hilb^{\P^{k}}(\cQ/S)$.\ We concentrate on the cases $k \in \{1,2\}$ (\ie, on the Hilbert schemes of lines and planes) and describe the Stein factorization of the canonical morphism $\varphi\colon F_k(\cQ) \to S$.\ Note that $F_k(\cQ)$ is a subscheme in the relative Grassmannian $\pi\colon\Gr_S(k+1,\cE) \to S$.\  We denote by $\cU$ the tautological subbundle of rank $k+1$ on $\Gr_S(k+1,\cE)$.

\begin{prop}\label{proposition:f2q}
Assume   $\Deg_{m-2}(\cQ) \ne S$.\  We have a  resolution
\begin{equation*}
0 \to \cL^3 \otimes  {(\det(\cU))^{\otimes 3}} \to \cL^2 \otimes \Sym^2\!\cU   \otimes \det(\cU) \to \cL \otimes \Sym^2\!\cU \to \cO_{\Gr_S(2,\cE)} \to \cO_{F_1(\cQ)} \to 0
\end{equation*}
on $\Gr_S(2,\cE)$.\  Moreover, the pushforward to $S$ of $\cO_{F_1(\cQ)}$ is given as follows
\begin{itemize} 
\item if $m = 3$, then
$\varphi_*\cO_{F_1(\cQ)} \cong \cO_{\Deg_1(\cQ)} \oplus (\cC \otimes \cL \otimes \det(\cE))$;
\item if $m = 4$, then
$\varphi_*\cO_{F_1(\cQ)} \cong \cO_S \oplus (\cL^2 \otimes \det(\cE))$;
\item if $m \ge 5$, then
$\varphi_*\cO_{F_1(\cQ)} \cong \cO_S$.
\end{itemize}
\end{prop}

\begin{proof}
Since $F_1(\cQ)$ is the zero-locus of a section of the rank-$3$ vector bundle $\cL^\vee \otimes \Sym^2\!\cU^\vee$  on~$\Gr_S(2,\cE)$, its codimension  is at most 3.\ On the other hand, for a quadric of rank $r$ in~$\P^{m-1}$, the dimension of the Hilbert scheme of lines is equal to $2m-7$ for $r \ge 3$ and to $2m - 6$ for~$r \le 2$.\  Stratifying $F_1(\cQ)$ by the preimages of the   subsets $S \setminus \Deg_{m-2}(\cQ)$ and  $\Deg_{m-2}(\cQ)$, we see that the codimension of the first stratum is 3 and the codimension of the second stratum is $\codim(\Deg_{m-2}(\cQ)) + 2$.\ Since $\Deg_{m-2}(\cQ) \ne S$, the codimension of $F_1(\cQ)$ is at least 3.\ Since $\Gr_S(2,\cE)$ is Cohen--Macaulay, the section of $\cL^\vee \otimes \Sym^2\!\cU^\vee$ defining $F_1(\cQ)$ is regular and the Koszul complex provides a resolution of its structure sheaf.\ A standard description of the exterior powers of a symmetric square (\cite[Proposition~2.3.9]{wey}) gives the above explicit form.

For the second part, we apply the Borel--Bott--Weil Theorem to compute the derived pushforwards  to $S$ of the terms of the Koszul complex.\  The result is 
\begin{align*}
R^\bullet\pi_*\cO_{\Gr_S(2,\cE)} & {}\cong \hspace{.7em}\cO_S,\\
R^i\pi_*\Sym^2\cU & {}\cong 
\begin{cases}
\det(\cE) \otimes \cE^\vee\\
0
\end{cases}
&
\begin{aligned}
& \text{if $m = 3$ and $i = 1$},\\
& \text{otherwise},
\end{aligned}
\\
R^i\pi_*\Sym^2\cU \otimes \det(\cU) & {}\cong 
\begin{cases}
\det(\cE) \otimes \cE\\
\det(\cE)\\
0
\end{cases}
&
\begin{aligned}
& \text{if $m = 3$ and $i = 1$},\\
& \text{if $m = 4$ and $i = 2$},\\
& \text{otherwise},
\end{aligned}
\\
R^i\pi_*(\det(\cU))^{\otimes 3} & {}\cong 
\begin{cases}
\det(\cE)^2\\ 
0
\end{cases}
&
\begin{aligned}
& \text{if $m = 3$ and $i = 2$},\\
& \text{otherwise}.
\end{aligned}
\end{align*}
Therefore, the pushforward of the Koszul complex for $m = 3$ gives an exact sequence
\begin{equation*}
0 \to  (\cL^3 \otimes \det(\cE)^2) \oplus (\cL^2 \otimes \det(\cE) \otimes \cE) \xrightarrow{\ \alpha\ } (\cL \otimes \det(\cE) \otimes \cE^\vee) \oplus \cO_S \to \varphi_*\cO_{F_1(\cQ)} \to 0.
\end{equation*}
The map $\alpha$ is the direct sum of a twist of the map $\alpha_1 \colon \cL \otimes \cE \to \cE^\vee$ and of a twist of its determinant $\alpha_0 \colon \cL^3 \otimes \det(\cE)^2 \to \cO_S$.\ The cokernel of $\alpha_1$ is $\cC$ and the cokernel of $\alpha_0$ is the structure sheaf of the degeneracy locus $\Deg_1(\cQ)$.\  This gives the result for $m = 3$.

For $m = 4$, the pushforward of the Koszul complex gives $\varphi_*\cO_{F_1(\cQ))} = \cO_S \oplus (\cL^2 \otimes \det(\cE))$,
and for $m \ge 5$, just $\varphi_*\cO_{F_1(\cQ))} = \cO_S$.
\end{proof}

For $k = 2$, the computation is analogous, but more complicated, since the Koszul complex is longer.\ We denote by $\Sigma^{a,b,c}\cU$ the Schur functor of the rank-3 tautological subbundle~$\cU$ on $\Gr_S(3,\cE)$ corresponding to the highest weight $(a,b,c)$ of the group $\GL_3$.\
We also consider the composition
\begin{equation*}
\cL \otimes \cE \otimes \cE^\vee \xrightarrow{\ \ } \cE^\vee \otimes \cE^\vee \xrightarrow{\ \ } \bw2\cE^\vee,
\end{equation*}
where the first map is given by the family of quadrics and the second is  canonical.\
Denote by~$\cC_2$ its cokernel sheaf; it is supported on the degeneracy locus $\Deg_2(\cQ)$.

\begin{prop}\label{proposition:f3q}
Assume  $\Deg_{m-4}(\cQ) \ne S$ and $\codim (\Deg_{m-2}(\cQ)) \ge 3$.\  There is a resolution
 \begin{multline*}
0 \to \cL^6 \otimes {\Sigma^{4,4,4}\cU} \to \cL^5 \otimes \Sigma^{4,4,2}\cU \to \cL^4 \otimes \Sigma^{4,3,1}\cU  \\
\to(\cL^3 \otimes \Sigma^{4,1,1}\cU) \oplus (\cL^3 \otimes \Sigma^{3,3,0}\cU)  \\
\to\cL^2 \otimes \Sigma^{3,1,0}\cU \to \cL \otimes \Sym^2\!\cU \to \cO_{\Gr_S(3,\cE)} \to \cO_{F_2(\cQ)} \to 0
\end{multline*}
on $\Gr_S(3,\cE) $.\ Moreover, the pushforward to $S$ of $\cO_{F_2(\cQ)}$  is given as follows:
 \begin{itemize} 
\item if $m = 4$, we have
$\varphi_*\cO_{F_2(\cQ)} \cong \cO_{\Deg_2(\cQ)} \oplus (\cC_2 \otimes \cL \otimes \det(\cE))$;
\item if $m = 5$, we have
$\varphi_*\cO_{F_2(\cQ)} \cong \cO_{\Deg_1(\cQ)} \oplus (\cC \otimes \cL^2 \otimes \det(\cE))$;
\item if $m = 6$, we have
$\varphi_*\cO_{F_2(\cQ)} \cong \cO_S \oplus (\cL^3 \otimes \det(\cE))$;
\item if $m \ge 7$, we have
$\varphi_*\cO_{F_2(\cQ)} \cong \cO_S$.
\end{itemize}
 \end{prop}

\begin{proof}
By definition, $F_2(\cQ)$ is the zero-locus of a section of the rank-6 vector bundle $\cL^\vee \otimes \Sym^2\!\cU^\vee$  on~$\Gr_S(3,\cE)$, so its codimension is at most 6.\ On the other hand, for a quadric of rank $r$ in~$\P^{m-1}$, the dimension of the Hilbert scheme of planes is equal to $3m-15$ for $r \ge 5$, $3m - 14$ for $4 \ge r \ge 3$, and $3m - 12$ for $r \le 2$.\ Thus, stratifying $F_2(\cQ)$ by the subsets $S \setminus \Deg_{m-4}(\cQ)$, $\Deg_{m-4}(\cQ) \setminus \Deg_{m-2}(\cQ)$, and $\Deg_{m-2}(\cQ)$, we see that  under our assumption, the codimension is~6, the section of $\cL^\vee \otimes \Sym^2\!\cU^\vee$ defining $F_2(\cQ)$ is regular, and the Koszul complex provides a resolution of its structure sheaf.\ A standard description of the exterior powers of a symmetric square (\cite[Proposition~2.3.9]{wey}) gives the above explicit form.

For the second part, we apply the Borel--Bott--Weil Theorem to compute the derived pushforwards to $S$ of the terms of the Koszul complex.\ The result is the following
{\allowdisplaybreaks
\begin{align*}
R^\bullet\pi_*\cO_{\Gr_S(2,\cE)} & {}\cong \hspace{.7em}\cO_S,\\
R^i\pi_*\Sym^2\cU & {}\cong 
\begin{cases}
\det(\cE) \otimes \bw2\cE^\vee \\
0
\end{cases}
&
\begin{aligned}
& \text{if $m = 4$ and $i = 1$},\\
& \text{otherwise},
\end{aligned}
\\
R^i\pi_*\Sigma^{3,1,0}\cU & {}\cong 
\begin{cases}
\det(\cE) \otimes ((\cE \otimes \cE^\vee)/\cO_{S}) \\
\det(\cE) \otimes \cE^\vee \\
0 
\end{cases}
&
\begin{aligned}
& \text{if $m = 4$ and $i = 1$},\\
& \text{if $m = 5$ and $i = 2$},\\
& \text{otherwise},
\end{aligned}
\\
R^i\pi_*\Sigma^{3,3,0}\cU & {}\cong 
\begin{cases}
\det(\cE)^2 \otimes \Sym^2\!\cE^\vee \\
0 
\end{cases}
&
\begin{aligned}
& \text{if $m = 4$ and $i = 2$},\\
& \text{otherwise},
\end{aligned}
\\
R^i\pi_*\Sigma^{4,1,1}\cU & {}\cong 
\begin{cases}
\det(\cE) \otimes \Sym^2\!\cE \\
\det(\cE) \otimes \cE \\
\det(\cE) \\
0
\end{cases}
&
\begin{aligned}
& \text{if $m = 4$ and $i = 1$},\\
& \text{if $m = 5$ and $i = 2$},\\
& \text{if $m = 6$ and $i = 3$},\\
& \text{otherwise},
\end{aligned}
\\
R^i\pi_*\Sigma^{4,3,1}\cU & {}\cong 
\begin{cases}
\det(\cE)^2 \otimes ((\cE \otimes \cE^\vee)/\cO_{S}) \\
0
\end{cases}
&
\begin{aligned}
& \text{if $m = 4$ and $i = 2$},\\
& \text{otherwise},
\end{aligned}
\\
R^i\pi_*\Sigma^{4,4,2}\cU & {}\cong 
\begin{cases}
\det(\cE)^2 \otimes \bw2\cE \\
\det(\cE)^2 \\
0
\end{cases}
&
\begin{aligned}
& \text{if $m = 4$ and $i = 2$},\\
& \text{if $m = 5$ and $i = 4$},\\
& \text{otherwise},
\end{aligned}
\\
R^i\pi_*{\Sigma^{4,4,4}\cU} & {}\cong 
\begin{cases}
\det(\cE)^3 \\
0
\end{cases}
&
\begin{aligned}
& \text{if $m = 4$ and $i = 3$},\\
& \text{otherwise}.
\end{aligned}
\end{align*}}
Therefore, the pushforward of the Koszul complex for $m = 4$ gives the exact sequence
\begin{multline*}
\dots \to (\cL^3 \otimes \det(\cE)^2 \otimes \Sym^2\!\cE^\vee) \oplus (\cL^2 \otimes \det(\cE) \otimes ((\cE \otimes \cE^\vee)/\cO_{S}))  \\ {}\xrightarrow{\ \alpha\ } (\cL \otimes \det(\cE) \otimes \bw2\cE^\vee) \oplus \cO_S \to \varphi_*\cO_{F_2(\cQ)} \to 0.
\end{multline*}
The map $\alpha$ is the direct sum of a twist of the map $\alpha_1 \colon \cL \otimes ((\cE \otimes \cE^\vee)/\cO_{S}) \to \bw2\cE^\vee$ and of the exterior cube of the family of quadrics $\alpha_0 \colon \cL^3 \otimes \Sym^2(\bw3\cE) \to \cO_S$.\ The cokernel of $\alpha_1$ is~$\cC_2$, and that of $\alpha_0$ is $\cO_{\Deg_2(\cQ)}$.\ This gives the result for $m = 4$.

For $m = 5$, the pushforward of the Koszul complex gives 
\begin{equation*}
0 \to  (\cL^5 \otimes \det(\cE)^2) \oplus (\cL^3 \otimes \det(\cE) \otimes \cE) \xrightarrow{\ \alpha\ } (\cL^2 \otimes \det(\cE) \otimes \cE^\vee) \oplus \cO_S \to \varphi_*\cO_{F_2(\cQ)} \to 0.
\end{equation*}
The map $\alpha$ is described as in Proposition~\ref{proposition:f2q} and gives the result for $m = 5$.\ For $m = 6$, the pushforward of the Koszul complex gives $\varphi_*\cO_{F_2(\cQ))} = \cO_S \oplus (\cL^3 \otimes \det(\cE))$, and, for $m \ge 7$, just $\varphi_*\cO_{F_2(\cQ))} = \cO_S$.
\end{proof}

\section{Resolutions of EPW surfaces}
\label{section:epw-surface}

In this appendix, we discuss a resolution of the structure sheaf of an EPW surface $Y_A^{\ge 2}$ in~$\P(V_6)$ and compute some cohomology spaces related to its ideal sheaf.\ We use freely the notation and results of~\cite[Appendix~B]{DK}, especially those introduced in Proposition~B.3.\ In particular, we set
\begin{equation*}
\hY_A := \{ (v,V_5) \in  \Fl(1,5;V_6)\mid   A \cap (v \wedge \bw2V_5) \ne 0 \}
\end{equation*}
and
\begin{equation}\label{yahatp}
\hY'_A := \{ (a,v,V_5) \in \PP(A) \times \Fl(1,5;V_6)\mid a \in  \P(A \cap (v \wedge \bw2V_5)) \}.
\end{equation}
When $A$ contains no decomposable vectors, the projection $\hY'_A \to \Fl(1,5;V_6)$ induces an isomorphism $\hY'_A \isomto \hY_A$.\ We denote by $H$ and $H'$ the hyperplane classes of $\P(V_6)$ and $\P(V_6^\vee)$ and by $p \colon \hY_A \to Y_A$ and $q \colon \hY_A \to Y_{A^\perp}$ the projections 
(we switch back from the notation $\pr_{Y,1}$ and~$\pr_{Y,2}$ used in the main body of the article to the notation used in~\cite{DK}).\ We also denote by $H_A$ the pullback of the hyperplane class of $\P(A)$ to $\hY_A$ via the map $\hY_A \to \P(A)$ provided by the identification of $\hY_A$ with $\hY'_A$ (when $A$ contains no decomposable vectors).\ We begin with a simple lemma.

\begin{lemm}\label{lemma:ps-ohy}
If $A$ contains no decomposable vectors, there is an isomorphism $p_*\cO_{\hY_A} \isomto \cO_{Y_A}$ and $R^{>0}p_*\cO_{\hY_A} = 0$.
\end{lemm}

\begin{proof}
By definition, $\hY_A = \hY'_A$ is the zero-locus of the composition
\begin{equation*}
\cO_{\P(A)}(-H_A) \boxtimes \cO_{\P(V_6)} \to A \otimes \cO_{\P(A) \times \P(V_6)} \to \bw3V_6 \otimes \cO_{\P(A) \times \P(V_6)} \to \cO_{\P(A)} \boxtimes \bw3T_{\P(V_6)}(-3H)
\end{equation*}
on $\P(A) \times \P(V_6)$, hence equals the zero-locus of the corresponding section of the vector bundle $\cO_{\P(A)}(H_A) \boxtimes \bw3T_{\P(V_6)}(-3H)$.\ Since the codimension of $\hY_A$ in $\P(A) \times \P(V_6)$ equals the rank of that vector bundle, we have a Koszul resolution
\begin{multline}\label{eq:koszul-ya}
\dots \to \cO_{\P(A)}(-2H_A) \boxtimes \bw2(\bw2T_{\P(V_6)})(-6H)  
\\ \to \cO_{\P(A)}(-H_A) \boxtimes \bw2T_{\P(V_6)}(-3H) \to \cO_{\P(A) \times \P(V_6)} \to \cO_{\hY_A} \to 0.
\end{multline}
Pushing it forward to $\P(V_6)$, we obtain an exact sequence
\begin{equation*}
0 \to \det(\bw2T_{\P(V_6)}(-3H)) \to \cO_{\P(V_6)} \to p_*\cO_{\hY_A} \to 0
\end{equation*}
and deduce vanishing of higher pushforwards.\ The first non-zero term of this sequence is isomorphic to $\cO_{\P(V_6)}(-6H)$, hence $p_*\cO_{\hY_A}$ is the structure sheaf of a sextic hypersurface, which clearly coincides with $Y_A$.
\end{proof}

The crucial observation on which the results of this appendix are based is the following.

\begin{lemm}
If $A$ contains no decomposable vectors, there is a linear equivalence of divisors
\begin{equation}\label{eq:hhh}
2H_A \lin H + H'.
\end{equation}
\end{lemm}

\begin{proof}
The definition~\eqref{yahatp} implies that the image of the tautological embedding $\cO_{\hY_A}(-H_A) \hookrightarrow A \otimes \cO_{\hY_A}$ is contained 
in the kernel of the composition 
\begin{equation*}
{A \otimes \cO_{\hY_A} \to \bw3V_6 \otimes \cO_{\hY_A} \to p^*(\bw3T_{\P(V_6)}(-3H)), }
\end{equation*}
itself identified with the kernel of the morphism $p^*(\bw2T_{\P(V_6)}(-3H)) \to (\bw3V_6/A) \otimes \cO_{\hY_A}$.\ Thus we have a natural embedding 
\begin{equation*}
\cO_{\hY_A}(-H_A) \hookrightarrow p^*(\bw2T_{\P(V_6)}(-3H))
\end{equation*}
of vector bundles: over a point $(a,v,V_5)$ of $\hY_A$ with $a = v \wedge \eta$, it is given by $\eta \in \bw2(V_5/v) \subset \bw2(V_6/v)$.\ Its wedge square is a map 
\begin{equation*}
\cO_{\hY_A}(-2H_A) \to p^*(\bw4T_{\P(V_6)}(-6H)) \cong p^*(\Omega_{\P(V_6)})
\end{equation*}
that sends a point $(a,v,V_5)$ of $\hY_A$, with $a = v \wedge \eta$, to $\eta \wedge \eta \in \bw4(V_5/v) \subset \bw4(V_6/v)$
which is non-zero since $a$ is indecomposable.\ This map is therefore an embedding of vector bundles.\ Twisting it by~$\cO(H)$ and composing with the canonical embedding gives a map 
\begin{equation*}
\cO_{\hY_A}(H-2H_A) \hookrightarrow p^*(\Omega_{\P(V_6)}(H)) \hookrightarrow V_6^\vee \otimes \cO_{\hY_A}
\end{equation*}
which defines a map $\hY_A \to \P(V_6^\vee)$, $(a ,v,V_5) \mapsto v \wedge \eta \wedge \eta$.\ In the proof of \cite[Proposition~B.3]{DK}, it was shown that this map is the projection $q \colon \hY_A \to Y_{A^\perp}$.\ Therefore, we have an isomorphism of line bundles $\cO_{\hY_A}(H-2H_A) \cong \cO_{\hY_A}(-H')$, hence~\eqref{eq:hhh}.
\end{proof}

This allows us to find a simple resolution of the ideal sheaf $\cI_{Y_A^{\ge 2},\,Y_A}$ of the EPW surface~$Y_A^{\ge 2}$ in the EPW sextic $Y_A$.

\begin{lemm}
If $A$ contains no decomposable vectors, there is an exact sequence 
\begin{multline}\label{eq:resolution-y2}
0 \to \bw2(\bw2T_{\P(V_6)})(-12H) \to A^\vee \otimes \bw2T_{\P(V_6)}(-9H)  
\\
 \to \Sym^2\!A^\vee \otimes \cO_{\P(V_6)}(-6H) \to \cI_{Y_A^{\ge 2},\,Y_A} \to 0
\end{multline}
of sheaves on $\P(V_6)$.
\end{lemm}

\begin{proof}
Denote by $E$ the exceptional divisor of the birational morphism $p \colon \hY_A \to Y_A$.\ We show first that $E$ coincides with the scheme-theoretic preimage of the EPW surface $Y_A^{\ge 2}$.\ For this, recall that $Y_A^{\ge 2}$ is by definition the corank-2 degeneracy locus of the composition
\begin{equation*}
A \otimes \cO_{\P(V_6)} \to \bw3V_6 \otimes \cO_{\P(V_6)} \to \bw3T_{\P(V_6)}(-3H).
\end{equation*}
When pulled back to $\hY_A$, it extends to a complex
\begin{equation*}
\cO_{\hY_A}(-H_A) \hookrightarrow A \otimes \cO_{\hY_A} \to p^*(\bw3T_{\P(V_6)}(-3H)) \thra \cO_{\hY_A}(H_A),
\end{equation*}
hence the preimage of $Y_A^{\ge 2}$ is the degeneracy locus of the induced map 
\begin{equation*}
(A \otimes \cO_{\hY_A})/\cO_{\hY_A}(-H_A) \to \Ker(p^*(\bw3T_{\P(V_6)}(-3H)) \thra \cO_{\hY_A}(H_A)).
\end{equation*}
This is a morphism between two vector bundles of rank 9, hence the preimage of $Y_A^{\ge 2}$ is the Cartier divisor in $\hY_A$ defined by a section of the line bundle
\begin{equation*}
\det(\Ker(p^*(\bw3T_{\P(V_6)}(-3H)) \thra \cO_{\hY_A}(H_A))) \otimes \det((A \otimes \cO_{\hY_A})/\cO_{\hY_A}(-H_A))^\vee \cong \cO_{\hY_A}(6H - 2H_A).
\end{equation*}
But $6H - 2H_A \lin 6H - H - H' \lin 5H - H'$ and this is linearly equivalent to $E$ by a computation in~\cite{DK} (a paragraph before Lemma~B.6).\ All global sections of the line bundle $\cO_{\hY_A}(E)$ are proportional (since $E$ is the exceptional divisor of a birational morphism), hence the {scheme-theoretic} preimage of~$Y_A^{\ge 2}$ equals~$E$.

Since $E = p^{-1}(Y_A^{\ge 2})$, there is an embedding of schemes $p(E) \subset Y_A^{\ge 2}$.\ On the other hand, $p(E)$ and $Y_A^{\ge 2}$ coincide set-theoretically (\cite[Proposition~B.3]{DK}) and the scheme $Y_A^{\ge 2}$ is reduced and normal (\cite[Theorem~B.2]{DK}), hence $p(E) = Y_A^{\ge 2}$.\ Since the fibers of the map $p \colon E \to Y_A^{\ge 2}$ are connected, it also follows that there is an isomorphism $p_*\cO_E \isomto \cO_{Y_A^{\ge 2}}$.

We now compute $p_*\cO_E$.\ We use the linear equivalence $6H - 2H_A \lin E$ shown above and compute the derived pushforward of the line bundle $\cO_{\hY_A}(-E) \cong \cO_{\hY_A}(2H_A - 6H)$.\ Twisting the Koszul resolution~\eqref{eq:koszul-ya} by $\cO_{\hY_A}(2H_A - 6H)$ and pushing forward to $\P(V_6)$, we obtain an exact sequence
\begin{equation*}
0 \to \bw2(\bw2T_{\P(V_6)})(-12H) \to A^\vee \otimes \bw2T_{\P(V_6)}(-9H) \to \Sym^2\!A^\vee \otimes \cO_{\P(V_6)}(-6H) \to p_*\cO_{\hY_A}(-E) \to 0,
\end{equation*}
and vanishing of higher pushforwards.\ Using this and Lemma~\ref{lemma:ps-ohy} and applying pushforward to the standard exact sequence $0 \to \cO_{\hY_A}(-E) \to \cO_{\hY_A} \to \cO_E$, we obtain an exact sequence
\begin{equation*}
0 \to p_*\cO_{\hY_A}(-E) \to \cO_{Y_A} \to p_*\cO_E \to 0.
\end{equation*}
The right term is isomorphic to $\cO_{Y_A^{\ge 2}}$, as we have shown above, hence the left term is $\cI_{Y_A^{\ge 2},\,Y_A}$.\ This proves the lemma.
\end{proof}

 {\begin{rema}
The equality $E = p^{-1}(Y_A^{\ge 2})$ shown in the proof means that $\hY_A$ is the blowup of~$Y_A$ along $Y_A^{\ge 2}$.
\end{rema}}

The sequence~\eqref{eq:resolution-y2} can be merged with the standard resolution of~$Y_A$ to give an exact sequence
\begin{multline}\label{eq:resolution-y2-p5}
0 \to \bw2(\bw2T_{\P(V_6)})(-12H) \to A^\vee \otimes \bw2T_{\P(V_6)}(-9H)  
\\
\to (\Sym^2\!A^\vee \oplus \C) \otimes \cO_{\P(V_6)}(-6H) \to \cO_{\P(V_6)} \to \cO_{Y_A^{\ge 2}} \to 0
\end{multline}
which can be used to compute the cohomology of line bundles on $Y_A^{\ge 2}$.

\begin{coro}\label{corollary:h-oy2}
If $A$ contains no decomposable vectors, the following table computes the cohomology spaces for some twists of the sheaf $\cO_{Y_A^{\ge 2}}$
\begin{equation*}
 \setlength{\extrarowheight}{1ex}
\begin{array}{|c|c|c|c|c|c|c|c|}
\hline 
& t = 0 & t = 1 & t = 2 & t = 3 & t = 4 & t = 5 & t = 6 
\\
\hline 
H^2(Y_A^{\ge 2},\cO_{Y_A^{\ge 2}}(tH)) & \bw2A & 0 & 0 & 0 & 0 & 0 & 0
\\
H^1(Y_A^{\ge 2},\cO_{Y_A^{\ge 2}}(tH)) & 0 & 0 & \bw2{{V_6}}^\vee & A^\vee & 0 & 0 & 0
\\
H^0(Y_A^{\ge 2},\cO_{Y_A^{\ge 2}}(tH)) & \C & {{V_6}}^\vee & \Sym^2\!{{V_6}}^\vee & \Sym^3\!{{V_6}}^\vee & \Sym^4\!{{V_6}}^\vee & \Sym^5\!{{V_6}}^\vee / {{V_6}} & \Sym^6\!{{V_6}}^\vee / (\Sym^2\!A^\vee \oplus \C)
\\[5pt]
\hline
\end{array}
\end{equation*}

Moreover, $H^0(\P(V_6),\cI_{Y_A^{\ge 2},\, \P(V_6)}(2H)) = H^1(\P(V_6),\cI_{Y_A^{\ge 2},\, \P(V_6)}(H)) = 0$.
\end{coro}
\begin{proof}
It consists of a straightforward computation using~\eqref{eq:resolution-y2-p5} and the Borel--Bott--Weil theorem.
\end{proof}

In Section~\ref{section:periods}, we used the following simple consequence of these computations.

\begin{coro}\label{corollary:y2av}
If $A$ contains no decomposable vectors, the curve $Y^{\ge 2}_{A,V_5} \subset \P(V_5)$ is not contained in a quadric.
\end{coro}

\begin{proof}
We have an exact sequence 
\begin{equation*}
0 \to \cI_{Y_A^{\ge 2},\, \P(V_6)}(-H) \to \cI_{Y_A^{\ge 2},\, \P(V_6)} \to \cI_{Y_{A,V_5}^{\ge 2},\, \P(V_5)} \to 0
\end{equation*}
of sheaves on $\P(V_6)$.\ The cohomology sequence of its twist by $\cO_{\P(V_6)}(2H)$ gives an exact sequence
\begin{equation*}
H^0(\P(V_6),\cI_{Y_A^{\ge 2},\, \P(V_6)}(2H)) \to H^0(\P(V_5),\cI_{Y_{A,V_5}^{\ge 2},\, \P(V_5)}(2H)) \to H^1(\P(V_6),\cI_{Y_A^{\ge 2},\, \P(V_6)}(H)).
\end{equation*}
By Corollary~\ref{corollary:h-oy2}, the spaces at both ends vanish, hence so does the middle space.
\end{proof}


\begin{thebibliography}{IKKR}

\bibitem[B]{bea}  Beauville, A.,  Vari\'et\'es K\"ahleriennes dont la premi\`ere classe de Chern est nulle, {\it J.~Differential  Geom.}  {\bf 18} (1983), 755--782.

 \bibitem[BD]{bedo} Beauville, A., Donagi, R., La vari\'{e}t\'{e} des droites d'une hypersurface cubique de dimension 4, {\it C.\  R.\  Acad.\  Sci.\  Paris S\'er.\  I Math.}  {\bf 301}  (1985),  703--706.

 \bibitem[C]{cor} Cornalba, M.,  Una osservazione sulla topologia dei rivestimenti ciclici di varieta algebriche, {\it
Boll. Unione Mat. Ital.} {\bf18} (1981), 323--328.  

 \bibitem[DIM]{dims} Debarre, O., Iliev, A., Manivel, L., Special prime Fano fourfolds of degree 10 and index 2, {\it  Recent Advances in Algebraic Geometry,} 123--155,  C.\  Hacon, M.\  Musta\c t\u a, and M.\  Popa editors, London Mathematical Society Lecture Notes Series {\bf417}, Cambridge University Press, 2014.
  
\bibitem[DK1]{DK}  Debarre, O., Kuznetsov, A., Gushel--Mukai varieties: classification and birationalities, to appear in {\it Algebr. Geom.} 

\bibitem[DK2]{DK-6fold}  \bysame, On the cohomology of Gushel--Mukai sixfolds, eprint {\tt arXiv:1606.09384.}

\bibitem[DK3]{DK-moduli}  \bysame, Gushel--Mukai varieties: moduli, 
to appear.

 \bibitem[Di]{di}  Dimca, A., {\it Singularities and topology of hypersurfaces,} Universitext, Springer-Verlag, New York, 1992.

\bibitem[EV]{ev} Esnault, H., Viehweg, E., Lectures on Vanishing Theorems, DMV Seminar, Band {\bf 20}, Birkh\"auser Verlag, 1992.

\bibitem[F]{Fe} Ferretti, A., Special subvarieties of EPW sextics, {\it Math.\  Z.} {\bf272} (2012), 1137--1164. 
 
\bibitem[HBJ]{hbj} Hirzebruch, F.,  Berger, T., Jung, R., {\it Manifolds and modular forms}, Friedr.\ Vieweg \& Sohn, Braunschweig, 1992.

  \bibitem[IKKR]{ikkr} Iliev, A.,    Kapustka, G., Kapustka, M., Ranestad, K.,   EPW cubes, {\it J. reine angew. Math.} {\bf} (2016).

  \bibitem[IM]{im} Iliev, A., Manivel, L., Fano manifolds of degree 10 and EPW sextics,
{\it  Ann.\  Sci.\  \'Ecole Norm.\  Sup.}  {\bf 44} (2011), 393--426.

 \bibitem[J]{jam} James, D.G., On Witt's theorem for unimodular quadratic forms, {\it Pacific J. Math.}  {\bf 26} (1968), 303--316. 

\bibitem[K]{K10}
Kuznetsov, A.,
Derived categories of cubic fourfolds, in {\em Cohomological and Geometric Approaches to Rationality Problems. New Perspectives,} 
F.\ Bogomolov and Yu.\ Tschinkel editors,  Progress in Mathematics {\bf282}, Birkh\"auser, 2010.

\bibitem[KP]{KP}
Kuznetsov, A., Perry, A.,
Derived categories of Gushel--Mukai varieties,
eprint {\tt arXiv:1605.06568}.

      
\bibitem[L]{lo} Logachev, D., Fano threefolds of genus 6, {\it Asian J. Math.}  {\bf 16} (2012), 515--560.


   
\bibitem[M]{mark} Markman, E., Integral generators for the cohomology ring of moduli spaces of sheaves
over Poisson surfaces, {\em Adv. Math.} {\bf208} (2007), 622--646.

 \bibitem[N]{nag}  Nagel, J., The generalized Hodge conjecture for the quadratic complex of lines in projective four-space, {\it  Math. Ann.}  {\bf 312} (1998),  387--401.

 \bibitem[Ni]{nik}      Nikulin, V.,     Integral symmetric bilinear forms and some of their geometric applications,   {\it   Izv. Akad. Nauk SSSR Ser. Mat.}  {\bf   43}  (1979), 111--177.
English transl.: {\it Math.\ USSR Izv.} {\bf 14} (1980), 103--167.

  \bibitem[O1]{og1}	O'Grady,  K.,  Irreducible symplectic 4-folds and Eisenbud-Popescu-Walter sextics, {\it Duke Math. J.}  {\bf 134} (2006),  99--137.

   \bibitem[O2]{og2}	\bysame,      Dual double EPW-sextics and their periods, {\it  Pure Appl.\ Math.\ Q.}  {\bf  4}  (2008),   427--468.

  
 \bibitem[O3]{og4}	\bysame,    Double covers of EPW-sextics, {\it Michigan Math.\ J.} {\bf 62} (2013), 143--184. 

 \bibitem[O4]{og6}	\bysame,    Periods of double EPW-sextics,  {\it Math. Z.} {\bf 280} (2015),  485--524.

  \bibitem[O5]{og7}	\bysame, Irreducible symplectic 4-folds numerically equivalent to $(K3)^{[2]}$, {\it Commun. Contemp. Math.} {\bf10} (2008),  553--608.
 
 \bibitem[O6]{og8}	\bysame, Involutions and linear systems on holomorphic 
symplectic manifolds, {\it Geom. Funct. Anal.} {\bf 15} (2005),  1223--1274.

\bibitem[W]{wey} Weyman, J., {\it Cohomology of Vector Bundles and Syzygies,}  Cambridge Tracts in Mathematics {\bf149}, Cambridge University Press, Cambridge, 2003.

 \end{thebibliography}
\end{document}